\documentclass[11pt]{amsart}
\usepackage{geometry}                % See geometry.pdf to learn the layout options. There are lots.
\usepackage{graphicx}
\usepackage{amssymb}
\usepackage{epstopdf}
\usepackage{amsmath}
\usepackage{color}
\usepackage{hyperref}
\usepackage{pdfsync}
\usepackage{tikz}
 \usetikzlibrary{arrows,decorations.pathmorphing,backgrounds,positioning,fit,petri}

\usepackage[all]{xy}
\xyoption{matrix}
\xyoption{arrow}
 
 \tikzset{help lines/.style={step=#1cm,very thin, color=gray},
help lines/.default=.5} % draws a grid spaced #1 cm
\tikzset{thick grid/.style={step=#1cm,thick, color=gray},
thick grid/.default=1} % draws a grid spaced #1 cm

\newtheorem{thm}{Theorem}[subsection]
\newtheorem{lem}[thm]{Lemma}
\newtheorem{cor}[thm]{Corollary}
\newtheorem{prop}[thm]{Proposition}

\theoremstyle{definition}
\newtheorem{defn}[thm]{Definition}
\newtheorem{eg}[thm]{Example}

\newtheorem{rem}[thm]{Remark}
\newtheorem{summ}[thm]{Summary}
 
\numberwithin{equation}{subsection}

\newcommand{\xcolor}[1]{}

\newcommand{\vs}[1]{\vskip .#1 cm} %enter amount of skip wanted at #1

\newcommand\noi{\noindent}

\def\xrarrow{\xrightarrow} %right arrow {label on top}

\def\ot{\leftarrow}
\def\then{\Rightarrow}

\def\-{\text{-}}

\newcommand{\into}{\hookrightarrow}
 \newcommand{\onto}{\twoheadrightarrow}

\newcommand{\Q}{(Q_\vare)}

\def\<{\left<}
\def\>{\right>}
\DeclareMathOperator{\Hom}{Hom}%
\DeclareMathOperator{\Ext}{Ext}%
\DeclareMathOperator{\sgn}{sgn}%
\DeclareMathOperator{\End}{End}%
\DeclareMathOperator{\undim}{\underline{dim}}
 
\newcommand{\field}[1]{\mathbb{#1}}
\newcommand{\ZZ}{\ensuremath{{\field{Z}}}}

\newcommand{\RR}{\ensuremath{{\field{R}}}}
\newcommand{\QQ}{\ensuremath{{\field{Q}}}}
\newcommand{\NN}{\ensuremath{{\field{N}}}}

\newcommand{\commentout}[1]{}

\newcommand{\cA}{\ensuremath{{\mathcal{A}}}}

\newcommand{\cC}{\ensuremath{{\mathcal{C}}}}

\newcommand{\cW}{\ensuremath{{\mathcal{W}}}}

\newcommand{\cX}{\ensuremath{{\mathcal{X}}}}
\newcommand{\cY}{\ensuremath{{\mathcal{Y}}}}

\newcommand\vare{\varepsilon}

\title{Picture groups of finite type \\and cohomology in type $A_n$}
\author{Kiyoshi Igusa}
\address{Department of Mathematics, Brandeis University, Waltham, MA 02454}\email{igusa@brandeis.edu}
\author{Gordana Todorov}
\address{Department of Mathematics, Northeastern University, Boston, MA 02115}
\email{g.todorov@neu.edu}
\author{Jerzy Weyman}
\address{Department of Mathematics, University of Connecticut, Storrs, CT 06269}
\email{jerzy.weyman@uconn.edu}
\thanks{The third author is supported by NSF Grant \#DMS-1400740}

\date{\today}                      % Activate to display a given date or no date

\subjclass[2010]{
16G20; 20F55}

\begin{document}

\begin{abstract} 
For every quiver (valued) of finite representation type we define a finitely presented group called a picture group. This group is very closely related to the cluster theory of the quiver. For example, positive expressions for the Coxeter element in the group are in bijection with maximal green sequences \cite{IT14}. The picture group is derived from the semi-invariant picture for the quiver. We use this picture to construct a finite CW complex which (by \cite{IT13}) is a $K(\pi,1)$ for this group. The cells are in bijection with cluster tilting objects. For example, in type $A_n$ there are a Catalan number of cells.

The main result of this paper is the computation of the  cohomology ring of all picture groups of type $A_n$ with any orientation and any coefficient ring.%. In every degree the cohomology group is free abelian with ranks given by the ``ballot numbers''. We also determine the cup product structure of the cohomology ring.% In the special case when the quiver has straight orientation, the picture groups were previously studied by Loday \cite{Loday} who called them ``Stasheff groups''.
\end{abstract}

\maketitle

%\tableofcontents

 %\newpage

%-----------------------------------------------------------------------------------------
% INTRODUCTION
%-----------------------------------------------------------------------------------------

\section*{Introduction}

Let $\Lambda$ be a finite dimensional hereditary algebra of finite representation type with $n$ simple modules and let $Q$ be the associated modulated quiver with $n$ vertices. Most of the notions and results of this paper depend only of the quiver $Q$ and not on the algebra $\Lambda$ which will be reflected in the notation. To each such quiver there is a well-known associated unipotent group $U_Q(\ZZ)$ (see Definition \ref{def: unipotent group of Q}). This group is given by generators and relations. For any group given by generators and relations there is the notion of ``spherical diagram'' which is a labeled subset of the 2-sphere $S^2$. We extend this definition to the $k$-sphere $S^k$ and define $k$-dimensional ``pictures'' for a group with a presentation (see Definition \ref{def of picture for a group G=(X,Y)}). The definition depends on the specific choice of generators since these generators are used as labels for $(k-1)$-dimensional simplices which we call ``walls'' in the picture. These walls partition the $k$-sphere into regions and, once we choose a basepoint region, each region can be labelled with an element of the group $G$ which can be read off of the generators on the walls of the picture.

For any nontrivial group with fixed presentation $G=\<\cX|\cY\>$, there are infinitely many pictures in each dimension \cite{Ithesis},\cite{IK},\cite{IOr},\cite{Loday},\cite{LyndonS}. However, for each of these pictures $L$ there is a canonically associated group $ G_0(L)=\<\cX_0|\cY_0\>$ where $\cX_0\subseteq \cX$ and $\cY_0\subseteq\cY$ are the generators and relations which actually occur in the picture $L$. We call this group the ``picture group'' of $L$ (see Section \ref{sec2}). 

In this paper we consider the unipotent group $U_Q(\ZZ)$ (section \ref{sec1}) and one particular picture $L(Q)$ for $U_Q(\ZZ)$. This picture is given by domains of semi-invariants on presentation spaces of $\Lambda$. It has dimension $n-1$ where $n$ is the number of vertices of $Q$. The picture $L(Q)$ has the property that the regions are in bijection with the cluster tilting objects of $\Lambda$. So, our construction gives another way to associate an element of this unipotent group to every cluster-tilting object of $\Lambda$. This is already known (see, e.g., \cite{LamPasha}) and this (assignment of an element of $U_Q(\ZZ)$ to cluster tilting objects) can easily be done in other ways. Our purpose is different as we now explain.

The purpose of this paper is to study the picture group $G_0(Q)$ (Definition \ref{def of picture group of a picture}) of the semi-invariant picture $L(Q)$ (Definition \ref{definition of s-i picture L(Q)}). This group has the same set of generators as the unipotent group $U_Q(\ZZ)$ but fewer relations. The picture group has many good properties. For example, it is a $CAT(0)$-group (see \cite{IT13} and \cite{Cat0}), it has finite cohomological dimension just like the unipotent group but, unlike $U_Q(\ZZ)$ it has no torsion in its homology (section \ref{sec4}). There is also a bijection between the set of positive expressions for the ``Coxeter element'' in the picture group of $Q$ and the set of ``maximal green sequences'' for the $Q$ \cite{IT14} which were introduced in \cite{Keller}.

As an example, consider the group $U_{A_n}(\ZZ)$ which is the group of $(n+1)\times (n+1)$ unipotent matrices with integer entries. This group has a presentation given by generators $x_{ij}$ for $0\le i<j\le n$ and relations
\begin{enumerate}
\item $[x_{ij},x_{k\ell}]=1$ if $i\neq \ell$ and $j\neq k$.
\item $[x_{ij},x_{jk}]=x_{ik}$ if $i<j<k$
\end{enumerate}
where we use the notation $[x,y]:=y^{-1}xyx^{-1}$ thoughout this paper. The following labeled diagram is a picture for this group for any $n\ge3$ if $a=x_{01}$, $b=x_{12}$, $c=x_{23}$, $x=x_{02}$, $y=x_{13}$ and $z=x_{03}$. This is an example of the semi-invariant picture $L(Q)$ defined in \ref{definition of s-i picture L(Q)} and shown in Proposition \ref{prop: L determines a group G} and Corollary \ref{cor: L(Lambda) is a picture} to be a picture for the unipotent group $U_{A_3}(\ZZ)$ and thus also a picture for the picture group $G_0(A_3)$.
%\begin{figure}
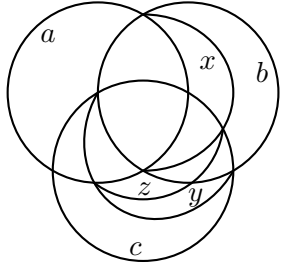
\begin{center}
\begin{tikzpicture}[scale=.4] % Huge CCCCCCCCCCCCCCC
	\draw (-4.5,4)+(.6,-.8) node{$a$};
	\draw (3,4)+(0.2,-2) node{$b$};
	\draw (-3,-4)+(2,.1) node{$c$};
	\draw (2,3)+(-.6,-.7) node{$x$};
	\draw (-1.5,-3)+(2.5,.8) node{$y$};
	\draw (.5,-2.2)+(-1.2,.3) node{$z$};
%
%	\draw[help lines=1] (-5.75,-5) grid (4.25,4.8);
	\clip (-5.75,-4.5) rectangle (4.25,4.8);
%	\draw[fill] (0,0) circle [radius=2pt];
		\draw[thick] (.75,1.3) circle [radius=3cm];
		\draw[thick] (-2.25,1.3) circle [radius=3cm];
		\draw[thick] (-.75,-1.3) circle [radius=3cm];
		\begin{scope}
		\clip (-.75,-5.2) rectangle (4.25,5);
		\draw[thick] (-.75,1.3) ellipse [x radius=3cm,y radius=2.6cm];
		\end{scope}
%	\draw (-2.25,-2.8) rectangle (3.75,2.8);
%\draw[fill, color=blue] (0,0) circle [radius=2pt];
		\begin{scope}[rotate=60]
		\clip (0,-5) rectangle (-5,5);
		\draw[thick] (0,0) ellipse [x radius=3cm,y radius=2.6cm];
		\end{scope}
%		\draw[thick, color=red] (-.75,.43) circle [radius=2.68cm];
\begin{scope}
\clip (-2.4,-3) rectangle (2,.1);
		\draw[thick] (-.75,.43)  circle [radius=2.68cm];
		\end{scope}
%	\draw[font=\huge,color=red] (0,0) node{C};
\end{tikzpicture}
\end{center}
%\end{figure}\label{A3 figure}
This picture has 6 smooth curves without inflection points including 3 circles. These curves meet transversely at 9 vertices breaking each smooth curve into segments. We use the convention that the labels are the same on all of segments of the same curve. For example, there are 5 segments labeled $a$ but with this convention we only need to draw the label on one of these segments. The curvature is constant on each curve and we use it to give the normal orientation of the curves.

Although this picture uses all six generators of $U_{A_3}(\ZZ)$ as labels only six of the $\binom{6}{2}=15$ relations appear at the vertices. These are the relations
\[
	[a,b]=x,[b,c]=y,[a,y]=z,[x,c]=z, [a,c]=1 ,[b,z] =1
\]
The picture group for this picture is therefore the group with the six generators $a,b,c,x,y,z$ and six relations as above.

Following the construction in \cite{IOr} of the nilmanifold for a torsion-free nilpotent group, we view an $(n-1)$-dimensional picture as the attaching map for an $n$-cell in a finite CW complex. The \emph{picture space} $X(Q)$ is the minimal CW complex which supports the attachment of the single $n$-cell given by the spherical semi-invariant picture $L(Q)$, together with this $n$-cell (Section \ref{sec3}). In the paper \cite{IT13} we prove that this is an Eilenberg-MacLane space $K(\pi,1)$ with $\pi_1=G_0(Q)$.

For the quiver $A_n$ with straight orientation: $1\ot 2\ot \cdots\ot n$, the semi-invariant picture group $G_0(A_n)$ has generators $x_{ij}$ for all $0\le i<j\le n$ subject to the following relations.
\begin{enumerate}
\item $x_{ij},x_{k\ell}$ commute if either $j<k$ or $i<k<\ell<j$.
\item $[x_{ij},x_{jk}]=x_{ik}$ for all $i<j<k$ where $[x,y]:=y^{-1}xyx^{-1}$.
\end{enumerate}
Note that the generating set is identical to that of the unipotent group $U_{A_n}(\ZZ)$ but the relations form a subset of the relations for $U_{A_n}(\ZZ)$. So, there is a natural epimorphism $G_0(A_n)\onto U_{A_n}(\ZZ)$.
These relations imply that $G_0(A_n)$ is generated by the $n$ elements $x_{j-1,j}$ for $j=1,\cdots,n$. The cohomology group $H^k(G_0(A_n))$ is free abelian of rank given by ``ballot numbers'' $b(n,n-2k)$. For any quiver $Q$ of type $A_n$, we show that the cohomology groups $H^k(G_0(Q))$ are isomorphic to that of the group $G_0(A_n)$ (Section \ref{sec4}). In Section \ref{sec5} we determine the cup product structure of the cohomology and show it is independent of the orientation of the quiver.

Finally, we recall some of the original motivation for the study of pictures although these comments will not be needed for the rest of this paper. Two dimensional pictures, also called ``spherical diagrams,'' were introduced by Lyndon and Schupp in \cite{LyndonS} to study identities among relations in groups. In \cite{Ithesis} pictures were used to define a K-theory invariant for $\pi_1$ of the diffeomorphism space $\cC(M)=Dif\!f(M\times [0,1] \text{ rel }M\times 0)$ for compact smooth manifolds $M$. The key idea was that elements of $K_3(\ZZ[\pi])$ for any group $\pi$ are represented by pictures for the Steinberg group of the group ring $\ZZ[\pi]$. This idea was used later in \cite{IK} to define and compute the higher Reidemeister torsion invariant for circle bundles over a $2$-sphere.

%\newpage

\section{Spherical semi-invariant picture $L(Q)$}\label{sec1}

We construct the {spherical semi-invariant picture} $L(Q)$ for any valued quiver $Q$ of finite representation type. This is a codimension one subcomplex of the $(n-1)$-sphere with suitable simplicial decomposition where $n$ is the number of vertices of $Q$. This is defined in terms of the representations of a hereditary algebra $\Lambda$ of finite representation type. In this case, indecomposable modules are uniquely determined by their dimension vectors and the dimensions of $\Hom$ and $\Ext$ between these modules can be computed using the Euler-Ringel form $\<\cdot,\cdot\>$. These vectors and the form $\<\cdot,\cdot\>$ can be computed from the underlying valued quiver $Q$ of $\Lambda$. So, we usually denote the semi-invariant picture by $L(Q)$ instead of $L(\Lambda)$.
\vs2

%\subsection{Nilpotent group and picture group}

%The paradigm for picture groups is the following. We start with a group and a presentation $G=\<\cX,\cY\>$, we choose one of the many ``pictures'' for the group.  Then we produce a group from the chosen picture which we call the ``picture group.'' We begin with a classical torsion free nilpotent group associated to any Dynkin diagram. We choose an orientation of the edges of the Dynkin diagram (making it into a Dynkin quiver) then show that the group is independent of the choice of orientation.

\subsection{Notation}

Let $\Lambda$ be a finite dimensional hereditary $K$-algebra of finite representation type. Here is a summary of well-known facts and our notation. See \cite{IOTW1}, \cite{IOTW3} for more details. Also \cite{DR} is the classical reference for valued quivers.

Since $\Lambda$ is of finite representation type, the quiver of $\Lambda$ is a valued quiver which is a disjoint union of Dynkin quivers. Recall that the \emph{quiver} $Q$ for the algebra $\Lambda$ is a directed graph with one vertex for every (isomorphism class of) simple module $S_i$, $i=1,\cdots,n$ with one arrow $i\to j$ if $\Ext^1_\Lambda(S_i,S_j)\neq0$. The quiver $Q$ has \emph{valuation} given by $f_i=\dim_KF_i$ where $F_i=\End_\Lambda(S_i)$ at each vertex $i$ and valuation $(d_{ij},d_{ji})$ on any arrow $i\to j$ where $d_{ij}=\dim_{F_j}\Ext^1_\Lambda(S_i,S_j)$ and $d_{ji}=\dim_{F_i}\Ext^1_\Lambda(S_i,S_j)$. Thus $d_{ij}f_j=d_{ji}f_i$.

Given any $\Lambda$-module $M$, the \emph{dimension vector} $\undim M$ is the vector in $\NN^n$ whose $i$-th coordinate is $\dim_{F_i}\Hom_\Lambda(P_i,M)$ where $P_i$ is the projective cover of $S_i$ with endomorphism ring canonically identified with $F_i=\End_\Lambda(S_i)$. A \emph{virtual representation} is a homomorphism between projective modules $p:P\to P'$ (thought of as an object of the derived category of $mod\text-\Lambda$) with morphisms given by homotopy classes of chain maps. Up to isomorphism, the indecomposable virtual representations are presentations of indecomposable modules and shifted indecomposable projective modules $P_i[1]:=(P_i\to 0)$. The dimension vector of a virtual representation $P\to P'$ is defined to be $\undim P'-\undim P$. Then the dimension vector of the minimal presentation of any module is equal to the dimension vector of the module.

The \emph{Euler matrix} $E$ is the $n\times n$ integer matrix with entries
\[
	E_{ij}=\dim_K\Hom_\Lambda(S_i,S_j)-\dim_K\Ext^1_\Lambda(S_i,S_j)
\]
Then, the \emph{Euler-Ringel form} $\<\cdot,\cdot\>:\ZZ^n\times\ZZ^n\to \ZZ$, defined by $\<v,w\>=v^tEw$, satisfies
\[
	\<\undim M,\undim N\>= \dim_K\Hom_\Lambda(M,N)-\dim_K\Ext^1_\Lambda(M,N).
\]
Let $\Phi^+(Q)$ be the set of \emph{positive roots} of $Q$. These are the dimension vectors of the indecomposable $\Lambda$-modules and we denote by $M_\alpha$ the unique indecomposable module with dimension vector $\alpha$. If $\pi_i=\undim P_i$ then we call $\{-\pi_i\}$ the \emph{negative projective roots}. These are the dimension vectors of the virtual representations $P_i[1]=(P_i\to 0)$ which we also denote by $M_{-\pi_i}$. We say that $\beta$ is an \emph{almost positive root} if it is either a positive root or a negative projective root. Thus $M_\beta$ has been defined for all almost positive roots $\beta$. In the sequel we use the notation $|P[1]|=P$ and $|-\beta|=\beta$. So, $|M_\beta|=M_{|\beta|}$.%We use the notation $|\beta|$ for the unique positive root proportional to any root $\beta$.

\begin{defn}
For any two almost positive roots $\alpha,\beta$, let %$hom(\alpha,\beta),ext(\alpha,\beta)\in \ZZ$ be defined by
\[
hom(\alpha,\beta)=\begin{cases} \dim_K\Hom_\Lambda(M_{|\alpha|},M_{|\beta|}) & \text{if $\alpha,\beta$ are either both positive or both negative} \\
\dim_K\Ext_\Lambda(M_{\alpha},M_{|\beta|}) & \text{if $\alpha$ is positive and $\beta$ is negative} \\
  0 & \text{otherwise}
    \end{cases}
%\dim_K\Hom_\Lambda(M_\alpha,M_\beta)
\]
\[
ext(\alpha,\beta)=\begin{cases} hom(\alpha,-\beta) & \text{if } \beta\in \Phi^+(Q)\\
  0  & \text{otherwise}
    \end{cases}
\]
We say that $\alpha,\beta$ are \emph{hom-orthogonal} if $hom(\alpha,\beta)=0=hom(\beta,\alpha)$ and \emph{ext-orthogonal} if $ext(\alpha,\beta)=0=ext(\beta,\alpha)$. 
\end{defn}

\begin{defn} The \emph{cluster complex} ${\xcolor{blue}\Sigma(\Lambda)}$ of $\Lambda$, which we also denote $\Sigma(Q)$ since it depends only on the valued quiver $Q$, is defined to be the abstract $(n-1)$-dimensional simplicial complex given as follows.
\begin{enumerate}
\item The vertices of $\Sigma(\Lambda)$ are the almost positive roots of $Q$ which, by definition, are the positive roots and the negative projective roots.
\item The $k$-simplices of $\Sigma(\Lambda)$ are $k+1$ tuples of pairwise ext-orthogonal almost positive roots. 
\end{enumerate}
\end{defn}

The vertices of $\Sigma(\Lambda)$ correspond to the indecomposable objects of the cluster category of $\Lambda$ \cite{BMRRT} and the $k$-simplices correspond to partial cluster tilting objects in the cluster category.

Since $\Lambda$ is of finite representation type, it is well-known (\cite{IOTW1}, \cite{IOTW3}, \cite{R}) that $\Sigma(\Lambda)=\Sigma(Q)$ is a finite complex whose geometric realization $|\Sigma(Q)|$ is homeomorphic to the unit sphere $S^{n-1}\subseteq\RR^n$ and the dual complex is a generalized associahedron (\cite{FZ-Y},\cite{MRZ}). 

Although we will define the picture group for arbitrary Dynkin quivers, we give the definition of the unipotent groups $U_Q(\ZZ)$ only in the simply laced case since the unipotent groups are being considered only for motivational purposes. (See \cite{Hum:AlgGrps} for an explanation of the general case.)

\begin{defn}\label{def: unipotent group of Q}
Given any quiver $Q$ of type $A,D,E$ with root system $\Phi(Q)$, let $U_Q(\ZZ)$ be the group given by generators and relations as follows.

\underline{Generators}: There is one generator $X(\alpha)$ for every positive root $\alpha\in\Phi^+$. If $\alpha$ is a sum of $k$ simple roots then we define the \emph{length} of $X(\alpha)$ to be $k$.

\underline{Relations}: a) $X(\alpha),X(\beta)$ commute if $\alpha+\beta$ is not a root.

b) If $\alpha+\beta\in\Phi^+$ then $X(\alpha)X(\beta)=X(\beta) X(\alpha+\beta)^\vare X(\alpha)$ where $\vare=1$ if $\<\alpha,\beta\>=0$ and $\vare=-1$ if $\<\alpha,\beta\>\neq 0$.
\end{defn}

\begin{prop}
For a simply laced Dynkin quiver $Q$, the unipotent group $U_Q(\ZZ)$ is a torsion-free nilpotent group with nilpotent basis $\{X(\beta):\beta\in\Phi^+\}$. Furthermore, up to isomorphism, $U_Q(\ZZ)$ depends only on the underlying unoriented Dynkin diagram.
\end{prop}

\begin{proof}
The first statement is a standard argument since the commutator of any two elements is a product of generators of larger length or their inverses. (See the ``collection process'' in \cite{Hall} or in \cite{IOr}.) Thus, any product of generators and their inverses can be rearranged in canonical order: in the order that the roots appear in the Auslander-Reiten quiver of $Q$. 

To prove the second statement, let $Q'$ be obtained from $Q$ by reversing one of the arrows, say $i\to j$. Then an isomorphism $\varphi:U_Q(\ZZ)\to U_{Q'}(\ZZ)$ is given by $\varphi(X(\beta))=X(\beta)^{\delta(\beta)}$ where $\delta(\beta)=(-1)^{b_ib_j}$ if $\beta=(b_1,\cdots,b_n)$. This gives an isomorphism since, for $Q'$, we have $\<\alpha,\beta\>'=\<\alpha,\beta\>+a_ib_j-a_jb_i$ which has the same parity as $\<\alpha,\beta\>+(a_i+b_i)(a_j+b_j)+a_ia_j+b_ib_j$ if $\alpha=(a_1,\cdots,a_n)$. So,
\[
	\vare(\alpha,\beta)'=(-1)^{\<\alpha,\beta\>'}=\vare(\alpha,\beta)\delta(\alpha)\delta(\beta)\delta(\alpha+\beta)
\]
which implies that $\varphi([X(\alpha),X(\beta)])=[\varphi (X(\alpha)),\varphi( X(\beta))]$.
\end{proof}

\begin{defn}[Semi-invariant picture]\label{definition of s-i picture L(Q)} For any hereditary algebra $\Lambda$ of finite representation type we define the (semi-invariant) \emph{picture} $L(\Lambda)\subset S^{n-1}$ to be the image of the geometric realization $|\Sigma(\Lambda)^{n-2}|$ of the $(n-2)$-skeleton of $\Sigma(\Lambda)$ under the natural mapping $\pi\circ\lambda:|\Sigma(\Lambda)|\to S^{n-1}$ given by the composition of the mapping
$
	\lambda: |\Sigma(\Lambda)|\to \RR^n\backslash 0
$
which is linear on each simplex and the inclusion map on the vertex set, with the projection map $\pi:\RR^n\backslash 0\onto S^{n-1}$ given by $\pi(x)=x/||x||$.
\end{defn}

By Theorem \ref{thm: spherical semi-invariant picture} there is another description of this picture given by semi-invariants: $L(\Lambda)=S^{n-1}\cap \bigcup D(\beta)$ where $D(\beta)$ is the domain of the semi-invariant $c_\beta$ described in subsection \ref{ss: semi-invariants}. We show in Theorem \ref{thm: spherical semi-invariant picture} that $\pi\circ\lambda:|\Sigma(\Lambda)|\to S^{n-1}$ is a homeomorphism giving an embedding $|\Sigma(\Lambda)^{n-2}|\into S^{n-1}$ whose image $L(\Lambda)$ is a ``picture'' for the group $G_0(\Lambda)$ as defined below.

%For the next definition we assume we have an $(n-2)$-dimensional simplicial complex $L$ embedded in $S^{n-1}$ together with normal orientations on $(n-3)$-dimensional (codimension 2 in $S^{n-1}$) simplices $\rho$ of $L$. Such a normal orientation induces a cyclic ordering of the $(n-2)$-simplices of $L$ which contain $\rho$. We use $k=n-1$ in the definition.

\begin{defn}[Pictures for a group]\label{def of picture for a group G=(X,Y)}
Let $G$ be a group given by generators and relations: $G=\<\cX|\cY\>$ where each $y\in\cY$ is a word in $\cX\cup\cX^{-1}$. The elements $y\in\cY$ are called \emph{relators} and the corresponding relation is $y=1$. For $k\ge2$, a $k$-dimensional \emph{picture} for $G$ is defined to be
\begin{enumerate}
\item  a $(k-1)$-dimensional \emph{subcomplex} $L$ of a triangulated $k$-sphere $S^k$ together with 
\item \emph{orientations} of the normal bundles in $S^k$ of all $(k-1)$-simplices and all $(k-2)$-simplices of $L$ and 
\item \emph{labels} $x(\sigma)\in \cX$ for each $(k-1)$-simplex $\sigma$ in $L$.% so that there exists a locally constant function $g:S^k\backslash L\to G$ so that, for every $(k-1)$-simplex $\sigma$ of $L$, $g(\tau')=g(\tau)x(\sigma)$ when $\tau'$ is on the positive side of $\sigma$.
%
%\begin{center}
%\begin{tikzpicture}%[scale=3]
%\draw[help lines=1,thick] (-5,-5) grid (5,4);
%\foreach \x in {-8,-6,...,8}\draw (\x,0) node{\x};\foreach \y in {-6,-4,...,8}\draw (0,\y) node{\y};
%\draw[thick,color=blue] (0,1) ellipse [x radius=2.8cm,y radius=2.1cm];
%\begin{scope}%[xshift=-4cm,yshift=10cm]
%	\draw[fill] (0,0) circle[radius=3pt];
%	\draw[thick] (0,.1)--(0,1.6);
%	\draw (0.25,1.4) node{$\sigma$} (0,.3)node[left]{-}(0,.3)node[right]{$+$};
%	\draw (-1.2,1) node[left]{$g(\tau)$}(1.2,1) node[right]{$g(\tau)x(\sigma)$};
%\end{scope}
%\end{tikzpicture}
%\end{center}
%
\item For every $(k-2)$-simplex $\rho$ of $L$, the $(k-1)$-simplices $\sigma_i$ of $L$ which contain $\rho$, say there are $s$ of them, have a specified \emph{numbering} $\sigma_1,\sigma_2,\cdots,\sigma_s$ in agreement with the cyclic ordering given by the normal orientation of $\rho$ given in (2), so that
\[
	\prod x(\sigma_i)^{\vare_i}\in\cY\cup\{xx^{-1}\,|\,x\in\cX\}
\]
where $\vare_i=+1$ if the positive side of $\sigma_i$ faces $\sigma_{i+1}$ and $\vare_i=-1$ if the negative side of $\sigma_i$ faces $\sigma_{i+1}$. We use the notation $y(\rho)=\prod x(\sigma_i)^{\vare_i}=x(\sigma_1)^{\vare_1}\cdots x(\sigma_s)^{\vare_s}$.
\end{enumerate}
%The specified cyclic numberings of the $(k-1)$-simplices $\sigma_i$ containing any given $(k-2)$-simplex $\rho$ are part of the structure of the picture, however, the function $g(\tau)$ on the regions $\tau$ are not part of the structure. (3) states that the function exists. (The function $g$ is well defined up to left multiplication by any fixed element of $G$.)
\end{defn}

Notice that not all elements of $\cX$ or $\cY$ need occur. In fact only a finite number of the elements of $\cX,\cY$ can occur in any picture.

The simplest example is the empty subset of $S^k$ which is a picture for any presented group $G=\<\cX\,|\,\cY\>$ since each simplex of the empty set has all required properties.

\begin{eg}We give three examples of pictures for the same group
\[
	U_{A_2}(\ZZ)=\< x,y,z\,|\, xyx^{-1}z^{-1}y^{-1}, yzy^{-1}z^{-1},xzx^{-1}z^{-1}\>
\]
where $x=x_{01},y=x_{12},z=x_{02}$. We use the conventions:
\begin{enumerate}
\item[(0)] The relator corresponding to a commutator relation $[a,b]=w$ for any word $w$ in $\cX\cup \cX^{-1}$ will be $aba^{-1}w^{-1}b^{-1}$.
\item We suppress bivalent vertices having the relation $aa^{-1}=1$ for any $a$.
\item We use curvature to indicate the normal orientation of each face: The positive side is in the direction of curvature. We place the label $x(\sigma)$ on the positive side of $\sigma$.
\item Segments of any smooth curve have the same label $x(\sigma)$.
\item When no two elements of $\cY\coprod \cY^{-1}$ are conjugate, the normal orientation and cyclic ordering at codimension 2 simplicies $\rho$ are uniquely determined when they exist. Convention (0) implies that relators always start and end in the ``outermost'' regions abutting each vertex as indicated in the second picture. For example, the top vertex is oriented counterclockwise with relator $xyx^{-1}z^{-1}y^{-1}$.
\end{enumerate}
\end{eg}

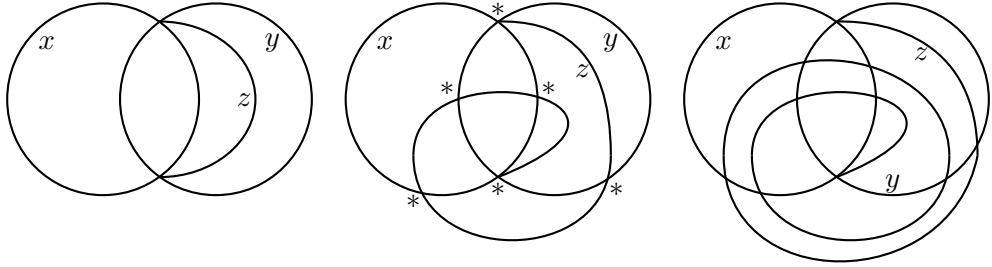
\begin{figure}[ht]\label{fig:3 pictures}
%%
%\begin{minipage}[b]{0.45\textwidth}
\begin{tikzpicture}[scale=.75]%[>=Stealth]
%\draw[help lines, color=blue] (-3,-1) grid (3,3);
%\foreach \x in {-2,...,2}\draw (\x,0) node{\x};\foreach \y in {-1,...,2}\draw (0,\y) node{\y};
\begin{scope}[xshift=-6cm]
\clip(-3.5,-1.5)rectangle(3,3);
\draw[thick] (-1,1) circle[radius=1.7];
\draw[thick] (1,1) circle[radius=1.7];
\draw (-2,2) node{$x$};
\draw (2,2) node{$y$};
\begin{scope}
\clip(2,-1)rectangle(0,3);
\draw[thick] (0,1) ellipse[x radius=1.7,y radius=1.38];
\draw (1.5,1) node{$z$};
\end{scope}
\end{scope}
\begin{scope}[xshift=6cm]
%\draw[help lines, color=blue] (-3,-2) grid (3,3);
%\foreach \x in {-2,...,2}\draw (\x,0) node{\x};\foreach \y in {-1,...,2}\draw (0,\y) node{\y};
\clip(-3.5,-2)rectangle(3,3);
\draw[thick] (-1,1) circle[radius=1.7];
\draw[thick] (1,1) circle[radius=1.7];
\draw (-2,2) node{$x$};
\draw (1,-.5) node{$y$};
\begin{scope}
\draw[thick] (0,-.38) .. controls (3.5,1) and (-1.5,2) .. (-1.5,0); \draw[thick] (2,0) .. controls (2,-2) and (-1.5,-2) .. (-1.5,0); 
\draw[thick] (2,0) .. controls (2,2) and (-2,2.5) .. (-2,0); 
\draw[thick] (2.5,0) .. controls (2,-2.5) and (-2,-2.5) .. (-2,0); 
\draw[thick] (2.5,0) .. controls (2.5,2) and (1,2.4) .. (0,2.38); 
\draw (1.5,1.8) node{$z$};
\end{scope}
\end{scope}
\coordinate (A) at (-1.5,-.8);
\coordinate (Bn) at (-.9,1.2);
\coordinate (Bp) at (.9,1.2);
\coordinate (C) at (0,2.6);
\coordinate (D) at (0,-.6);
\coordinate (E) at (2.1,-.6);
\begin{scope}
\clip(-3.5,-1.6)rectangle(3,3);
\foreach \z in {A,Bn,Bp,C,D,E}
\draw[fill] (\z) node{$\ast$};%circle [radius=1pt];
\draw[thick] (-1,1) circle[radius=1.7];
\draw[thick] (1,1) circle[radius=1.7];
\draw (-2,2) node{$x$};
\draw (2,2) node{$y$};
\begin{scope}
\draw[thick] (0,-.38) .. controls (3.5,1) and (-1.5,2) .. (-1.5,0); \draw[thick] (2,0) .. controls (2,-2) and (-1.5,-2) .. (-1.5,0); 
\draw[thick] (2,0) .. controls (2,2) and (1,2.4) .. (0,2.38); 
\draw (1.5,1.5) node{$z$};
\end{scope}
\end{scope}
\end{tikzpicture}
%
%\end{center}
	\caption{Three pictures for $U_{A_2}(\ZZ)$ using all three generators $x,y,z$. The first picture uses one of the relations, the others use all three relations.}
	\label{three pictures}
\end{figure}

\begin{defn}\label{def of picture group of a picture}
Let $L\subset S^k$ be a picture for a presented group $G=\<\cX|\cY\>$. Then we define the \emph{picture group} of $L$ to be the group $G_0(L):=\<\cX_0\,|\,\cY_0\>$ where $\cX_0\subseteq \cX$ and $\cY_0\subseteq \cY$ are the sets of labels which actually occur in the picture. More precisely, $\cX_0$ is the set of all labels $x(\sigma)$ of all $(k-1)$-simplices $\sigma$ of $L$ and $\cY_0\subseteq \cY$ is the set of all words $y(\rho)$ in $\cX_0\coprod \cX_0^{-1}$ given by reading the elements of $\cX_0$ which occur as labels $x(\sigma_i)^{\varepsilon_i}$ of the $(k-1)$-simplices $\sigma_i$ containing the same $(k-2)$-simplex $\rho$ of $L$ as explained in Definition \ref{def of picture for a group G=(X,Y)}(4) above. 
\end{defn}

%and $\cY_0\subseteq \cY$ is the set of $y(\rho)\in\cY$ associated to each $(k-2)$-simplex $\rho$ of $L$ as explained in Definition \ref{def of picture for a group G=(X,Y)}(4) above.

For the three pictures $L_1,L_2,L_3$ in Figure \ref{three pictures}, the picture groups are not the same. The picture group for $L_1$ is $
G_0(L_1)=\<x,y,z\,|\, xyx^{-1}z^{-1}y^{-1}\>\cong \<x,y\>=F_2$, the free group on two generators, and the picture groups for the other two are equal to the original group: $G_0(L_2)=G_0(L_3)=U_{A_2}(\ZZ)$.

\begin{rem}\label{rem: universality of picture group}
If $L$ is a picture for $G=\<\cX\,|\,\cY\>$ then there is a canonical homomorphism $G_0(L)\to G$ from the picture group of $L$ to $G$ induced by the inclusion $\cX_0\into \cX$.
\end{rem}

We now describe a method for producing a group $G$ and a picture $L$ for that group at the same time so that $G=G_0(L)$.

Let $L$ be a one-dimensional subcomplex of any triangulation of the 2-sphere $S^2$ which, when considered as a graph, contains no leaves. (Every vertex of $L$ is adjacent to at least two {edges}.) Choose a normal orientation of each edge and vertex in $L$. Choose one edge adjacent to each vertex. Let $x:L_1\onto \cX$ by any surjective mapping of the set of edges of $L$ to any finite set $\cX$. For each vertex $v\in L_0$, let $y(v)$ be the product of the labels $x(e_i)^{\vare_i}$ on the edges adjacent to $v$ starting with the chosen adjacent edge and going either clockwise or counterclockwise according to the orientation of $v$, with exponent $\vare_i=\pm1$ according to the orientation of $e_i$. Then $L$ is a picture for the group $G_0=\<x(e),e\in L_1\,|\,y(v),v\in L_0\>$ and $G_0$ is the picture group of $L\subset S^2$. Figure \ref{figureone} gives an example of a group defined in this way.

More generally we have% statement which follows from the definitions.

\begin{prop}\label{prop: L determines a group G}
Let $L$ be a codimension one subcomplex of a triangulated $k$-sphere $S^k$ with normal orientations on its $(k-1)$-simplices and labels in a set $\cX$ on the $(k-1)$-simplices so that every $(k-2)$-simplex of $L$ lies on the boundary of at least two $(k-1)$-simplices of $L$. Then $L$ is a picture for some group $G=\<\cX\,|\,\cY\>$ with generating set $\cX$.
\end{prop}

\begin{proof}
The relations of $G$ are given as follows. For each codimension 2 simplex $\rho$ of $L$, choose a normal orientation of $\rho$ in $S^k$. This gives a cyclic ordering to the $(k-1)$-simplices of $L$ which contain $\rho$. Number these $\sigma_1,\sigma_2,\cdots$. The labels and orientations of the faces $\sigma_i$ give a word $y=\prod x_i^{\vare_i}$ in the letters $\cX\coprod \cX^{-1}$. We take these words as the relators of the group. Although each relator is only well defined up to cyclic orientation and inversion, the corresponding relation $y=1$ is essentially well-defined. By definition, $L$ is then a picture for $G=\<\cX\,|\,\cY\>$ where $\cY=\{y\}$ is the set of words chosen in this way.
\end{proof}

%Figure #1
%
\begin{figure}[htbp]
\begin{center}
%
%\vs5
%
\begin{tikzpicture}%[scale=3]
%\draw[help lines=1,thick] (-3,0) grid (3,4);
%\foreach \x in {-8,-6,...,8}\draw (\x,0) node{\x};\foreach \y in {-6,-4,...,8}\draw (0,\y) node{\y};
%\draw[thick,color=blue] (0,1) ellipse [x radius=2.8cm,y radius=2.1cm];
\draw[thick] (-2,0)-- (2,0)-- (0,3)--(-2,0) (0,0)--(1,1.5)--(-1,1.5)--(0,0);
\draw (-1,0) node[above]{$x$};
\draw (1,0) node[above]{$z$};
\draw (-1.5,0.8) node[right]{$x$};
\draw (-.51,0.8) node[right]{$y$};
\draw (.52,0.8) node[left]{$x$};
\draw (0,1.5) node[below]{$z$};
\draw (1.5,0.8) node[left]{$z$};
\draw (-.51,2.2) node[right]{$y$};
\draw (.51,2.2) node[left]{$y$};
\end{tikzpicture}
\caption{This graph with indicated labels and normal orientation given by placing the labels on the positive side of each edge and taking positive (counterclockwise) orientation at each vertex, determines the group
\[
G_0=\<x,y,z\,|\, xyz^{-1}y^{-1},yzx^{-1}z^{-1},zxy^{-1}x^{-1}\>\qquad\qquad
\]
which is the fundamental group of the complement of the trefoil knot.}
\label{figureone}
\end{center}
\vspace{-12pt}
\end{figure}
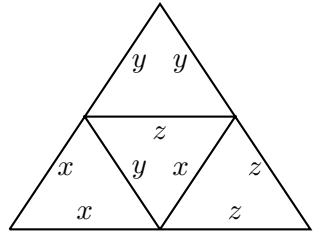
%
%\end{eg}

We will use semi-invariants to provide a system of labels and normal orientations for  $L(\Lambda)\subseteq S^{n-1}$. This will simultaneously define a group $G_0(\Lambda)$ and show that $L(\Lambda)$ is an $(n-1)$-dimensional picture for this group.

\subsection{Semi-invariants}\label{ss: semi-invariants}

For every positive root $\beta\in\Phi^+(Q)$, let $M_\beta$ be the unique indecomposable $\Lambda$-module with dimension vector $\beta$. Then the (integral) \emph{support} of $\beta$ is defined to be the set of all dimension vectors $\undim V:=\undim P_0-\undim P_1$ of all virtual representations $V=(p:P_1\to P_0)$ where $P_0,P_1$ are projective $\Lambda$-modules so that %$\Hom_\Lambda(V,M_\beta)=0=\Ext^1_\Lambda(V,M_\beta)$ or, equivalently,
\[
	\Hom_\Lambda(p,M_\beta):\Hom_\Lambda(P_0,M_\beta)\to \Hom_\Lambda(P_1,M_\beta)
\]
is an isomorphism. The determinant of this linear map is called the (value at $V$ of the determinantal) \emph{virtual semi-invariant} of determinantal \emph{(det-)weight} $\beta$. (See \cite{IOTW3}.)%, especially Appendix B for definitions and proofs.) 

\begin{rem}\label{eq: virtual stability theorem} The \emph{real support} or \emph{domain} of $\beta$, denoted $D(\beta)$, is defined to be the closure in $\RR^n$ of the set of all vectors in $\QQ^n$ an integer multiple of which lies in the integral support of $\beta$. The \emph{virtual stability theorem} \cite{IOTW3}, Theorem 3.1.1, states that
\[%\begin{equation}\label{eq: virtual stability theorem}
	D(\beta)=\{v\in \RR^n\,|\, \<v,\beta\>=0\text{ and} \<v,\beta'\>\le 0\ \text{for all } \beta'\subseteq\beta\}
\]%\end{equation}
where $\beta'\subseteq \beta$ means that $M_\beta$ contains a submodule isomorphic to $M_{\beta'}$. This formula implies in particular that $D(\beta)$ depends only on the valued quiver $Q$ and positive root $\beta$.
\end{rem}

Note that $D(\beta)$ is the closure of a convex open subset of the hyperplane
\[
	H(\beta)=\{v\in \RR^n\,|\, \<v,\beta\>=0\}.
\]
This hyperplane has a normal orientation. The positive side is given by
\[
	H_+(\beta)=\{v\in \RR^n\,|\, \<v,\beta\>\ge0\}.
\]
Thus, each $D(\beta)$ is a normally oriented codimension one subspace of $\RR^n$.

\begin{lem}\label{lem: faces of top dim simplices are in D(|c-vector|)}\cite{IOTW3}
For every cluster tilting object $T_1\oplus\cdots\oplus T_n$ in the cluster category of $\Lambda$ there are unique roots $\gamma_1,\cdots,\gamma_n\in\Phi(Q)$ so that
\[
	\<\undim T_i,\gamma_j\>=\delta_{ij}\dim_K\End_\Lambda(T_i).
\]
For each $j$, the vectors $\undim T_i$ for $i\neq j$ lie in $D(|\gamma_j|)$ and span the hyperplane $H(\gamma_j)$.
\end{lem}

Since $D(|\gamma_i|)$ is convex it contains all nonnegative linear combinations of $\undim T_j,j\neq i$.

\begin{thm}\label{thm: spherical semi-invariant picture}
Let $\Lambda$ be a hereditary algebra of finite representation type. Let $L(\Lambda)$ be the semi-invariant picture for $\Lambda$ and $S^{n-1}$ the unit sphere in $\RR^n$. Then
\[
	L(\Lambda)=\bigcup_{\beta\in\Phi^+(Q)}D(\beta)\cap S^{n-1}.
\]
\end{thm}

\begin{proof} Let $D(\Lambda)\subset \RR^n$ be the union of all $D(\beta)$ where $\beta\in\Phi^+(Q)$. Then the statement of the theorem is that $L(\Lambda)=D(\Lambda)\cap S^{n-1}$. It is clear that $L(\Lambda)$ is a subset of $D(\Lambda)\cap S^{n-1}$ since $D(\Lambda)$ contains the $(k-2)$-skeleton of $\Sigma(\Lambda)$. Conversely, suppose that $v\in D(\Lambda)\cap S^{n-1}$ and $v\notin L(\Lambda)$. Since $L(\Lambda)$ is a closed set and every point in $D(\Lambda)$ is a limit of rational points, there is a rational vector $w\in D(\Lambda)$ so that $w/||w||$ is not in $L(\Lambda)$. By definition of $L(\Lambda)$ this implies that some positive scalar multiple of $w$ has the form $mw=\sum a_i\undim T_i$ for some cluster tilting object $T_1\oplus\cdots\oplus T_n$ where $a_i$ are positive integers. But, $\bigoplus T_i^{a_i}$ is the generic module of dimension vector $mw$. So, $\sum a_i\undim T_i\in D(\beta)$ implies $\undim T_i\in D(\beta)$ for all $i$. But this is impossible since $\undim T_i$ are linearly independent and $D(\beta)$ is a subset of a hyperplane through the origin.
\end{proof}

\begin{cor}\label{cor: L(Lambda) is a picture}
The semi-invariant picture $L(\Lambda)\subset S^{n-1}$ is an $(n-1)$-dimensional picture for a group with generators $x(\beta)$ for $\beta\in \Phi^+(Q)$. Also, $L(\Lambda)$ together with its normal orientation and system of labels depends only on the underlying valued quiver $Q$ of $\Lambda$.
\end{cor}

\begin{proof} This follows from Proposition \ref{prop: L determines a group G} since the subsets $D(\beta)\cap S^{n-1}\subseteq L(\Lambda)$ are normally oriented and labeled with positive roots $\beta\in \Phi^+(Q)$.
\end{proof}

Since $L(\Lambda)$ depends only on $Q$ we write $L(\Lambda)=L(Q)$.

%{\xcolor{blue} Summary: 
\begin{summ} Section \ref{sec1} constructs the spherical semi-invariant picture $L(\Lambda)=D(\Lambda)\cap S^{n-1}$ where $D(\Lambda)$ is the union of domains $D(\beta)$ of virtual semi-invariants of det-weight $\beta$. These sets are normally oriented and labelled $\beta$. So, $L(Q)$ is a picture for some group with generators $x(\beta)$.
\end{summ}
%}

\section{Picture group $G_0(Q)$}\label{sec2}

Let $G_0(Q)=G_0(L(Q))$ denote the picture group of the subcomplex $L(Q)\subseteq S^{n-1}$.
\vs2

The generators of the picture group are, by definition, the labels of the walls in $L(Q)$. Since we sometimes think of $L(Q)$ as an $(n-2)$-dimensional subcomplex of $S^{n-1}$ and sometimes as an $(n-1)$-dimensional subcomplex of $\RR^n$, we will refer to the codimension instead of the dimension of its pieces. The walls are the codimension-one sets. Since these walls are $D(\beta)$ for all positive roots $\beta$ of $Q$, we have a generator ${\xcolor{blue}x(\beta)}$ for each $\beta\in\Phi^+(Q)$. %These correspond to wide subcategories of $mod\,KQ$ of rank $1$.

\subsection{Simplices of $L(Q)$}

We now consider a codimension $p\ge2$ simplex $\rho$ of $L(Q)=L(\Lambda)$. In this section we are only interested in the case $p=2$, but the general case is needed for the next section. By definition of $L(\Lambda)$, the vertices of any simplex $\rho$ form a partial cluster tilting object $T_1\oplus\cdots\oplus T_{n-p}$ in the cluster category of $\Lambda$. By Lemma \ref{lem: faces of top dim simplices are in D(|c-vector|)}, the codimension 1 simplices of $L(Q)$ which contain $\rho$ are contained in $D(\beta_j)$, $j>n-p$, for some completion $T_1\oplus\cdots\oplus T_n$ of the partial cluster tilting object to a full cluster tilting object where $\beta_j=|\gamma_j|\in\Phi^+(Q)$ in the notation of the lemma. Furthermore, the condition $T_i\in D(\beta_j)$ for $i=1,2,\cdots,n-p$ is equivalent to the condition that $M_{\beta_j}$ lies in the right hom-ext perpendicular category $|T|^\perp$ of the underlying module $|T|$ of $T= T_1\oplus\cdots\oplus T_{n-p}$. Since $T$ has $n-p$ components, $|T|^\perp$ is a (finitely generated) wide subcategory of $mod\text-\Lambda$ of rank $p$. % which is isomorphic to the module category of a hereditary algebra with $p$ nonisomorphic simple objects.
We recall from \cite{InTh} that a subcategory of $mod\text-\Lambda$ is called a \emph{wide subcategory} if it is an abelian subcategory which is closed under extensions and which is exactly embedded in $mod\text-\Lambda$. We consider wide subcategories which are finitely generated which means that there is one object $M$ so that every other object is a quotient of $M^k$ for some $k$. The wide subcategory has rank $p$ if the minimal generator $M$ has $p$ direct summands. One of the basic theorem about finitely generated wide subcategories is that they are subcategories $\cW$ so that $\cW=(^\perp \cW)^\perp=\,^\perp(\cW^\perp)$.

%Let $M_{\alpha_1},\cdots,M_{\alpha_p}$ be the simple objects of the wide subcategory $|T|^\perp$. Since $\Lambda$ is of finite representation type, we can number the roots so that $ext(\alpha_i,\alpha_j)=0$, or equivalently, $\<\alpha_i,\alpha_j\>=0$ for $i<j$. All other objects of $|T|^\perp$ have the form $M_\gamma$ where $\gamma=\sum r_j\alpha_j$, $r_j\ge0$, is a nonnegative integer linear combination of the $\alpha_j$. The following lemma determines which $M_\gamma$ are exceptional modules, i.e., when $\gamma$ is a positive root.

\begin{lem}\label{lem: equivalent conditions on gamma=sum rj aj}
Let $T\!=\!T_1\oplus\cdots\oplus T_{n-p}$ be a partial cluster tilting object, $\rho$ the simplex in $L(Q)$ spanned by $\undim T_i$ and $M_{\alpha_1},\cdots,M_{\alpha_p}$ the simple objects of the wide subcategory $|T|^\perp$. Then, for any positive root $\gamma\in\Phi^+(Q)$, the following are equivalent.
\begin{enumerate}
%\item There is an $(n-2)$-simplex $\sigma$ in $L(Q)$ so that $\rho\subset\sigma\subseteq D(\gamma)$.
\item The indecomposable module $M_\gamma$ lies in $|T|^\perp$.
\item The modules $|T_i|$, $i=1,\cdots,n-p$, lie in $\,^\perp M_\gamma$.
\item $\rho\subseteq D(\gamma)$.
\item $\gamma$ has the form $\gamma=\sum r_i\alpha_i$ where $r_i\ge0$.
\end{enumerate}
\end{lem}

\begin{proof} %As explained above, Lemma \ref{lem: faces of top dim simplices are in D(|c-vector|)} implies that (1) and (2) are equivalent. 
(1) and (2) are clearly equivalent. (2) and (3) are equivalent since $|T_i|\in \,^\perp M_\gamma$ is equivalent to the statement that $\undim T_i\in D(\gamma)$. One needs to observe that, when $T_i$ is projective, $\undim T_i\in D(\gamma)$ if and only if $-\undim T_i\in D(\gamma)$. So, restricting to the positive vector $\undim |T_i|$ does not hurt. Also, (1),(2),(3) imply (4) since $M_{\alpha_i}$ are the unique simple objects in the category $|T|^\perp$. To see $(4)\then(1)$, suppose that $\gamma=\sum r_j\alpha_j\in \Phi^+(Q)$. Then 
\[
	\<\undim T_i,\gamma\>=\sum r_j\<\undim T_i,\alpha_j\>=0
\]
Since $\Lambda$ is of finite representation type, this implies that $M_\gamma\in |T|^\perp$. 

So, all four statement are equivalent.
\end{proof}

We denote by ${\xcolor{blue}\Phi^+(\alpha_\ast)}$ the set of all $\gamma\in\Phi^+(Q)$ which can be written as $\gamma=\sum r_i\alpha_i$ where $r_i\ge0$.

\begin{lem}\label{lem: which D(beta) occur twice}\label{lem: hom-orthogonal pairs of roots}
Let $T\!=\!T_1\oplus\cdots\oplus T_{n-p}$ be a partial cluster tilting object, $\rho$ the simplex in $L(Q)$ spanned by $\undim T_i$ and $M_{\alpha_1},\cdots,M_{\alpha_p}$ the simple objects of the wide subcategory $|T|^\perp$.%the following.

(a) The interior of the simplex $\rho$ with vertices $\undim T_i$ lies in the interior of each $D(\alpha_j)$ but it lies on the boundary of $D(\gamma)$ for any $\gamma=\sum r_j\alpha_j$ which is not one of the $\alpha_j$.

(b) Let $\gamma_1,\cdots,\gamma_p\in \Phi^+(\Lambda)$. Then $M_{\gamma_j}$ are the simple objects of a wide subcategory of $mod\text-\Lambda$ of rank $p$ if and only if they are pairwise hom-orthogonal. %Let $\alpha_1,\cdots,\alpha_p\in \Phi^+(\Lambda)$. Then $M_{\alpha_j}$ are the simple objects of a wide subcategory of $mod\text-\Lambda$ of rank $p$ if and only if they are pairwise hom-orthogonal. 
\end{lem}

\begin{proof}
(a) Take any fixed $v\in\,int\,\rho$. Then, $v$ lies in $D(\gamma)$ if and only if $\rho\subseteq D(\gamma)$. This happens if and only if $\gamma=\sum r_j\alpha_j$ for some $r_j\ge0$. It follows that $v$ would not be contained in $D(\beta)$ if $\beta$ were a subroot of any $\alpha_j$. So, take $\beta$ which is not a subroot of any $\alpha_j$. By the virtual stability theorem Remark \ref{eq: virtual stability theorem}, this implies that $
	\<v,\beta\><0
$. Since this is an open condition, we have $\<w,\beta\><0$ for all $w$ in some neighborhood of $v$. Therefore, $v$ lies in the interior of each $D(\alpha_j)$.

Any nontrivial linear combination $\gamma=\sum r_j\alpha_j$, will contain some $\alpha_j$ as a subroot. Also, $\alpha_j,\gamma$ will be linearly independent. So, the hyperplanes $H(\alpha_j),H(\gamma)$ intersect transversely along a codimension 2 subspace which contains the simplex $\rho$. Since $\<v,\alpha_j\>\le 0$ for all $v\in D(\gamma)$, the set $D(\gamma)$ is restricted to the negative side of $H(\alpha_j)$. So, $\rho$ lies on the boundary of $D(\gamma)$ as claimed.
%\end{proof}
%Finally, we need to characterize which pairs of roots $\alpha,\beta\in\Phi^+(\Lambda)$ arise in the way that we described.
%These categories are determined by their sets of simple objects which consist of any two hom-perpendicular indecomposable modules. We let $M_\beta$ be the unique indecomposable module with dimension vector $\beta$. (The $k$-th coordinate of $\beta$ is the dimension over $F_k$ of the vector space $(M_{\beta})_k$.) Recall that two modules $A,B$ are \emph{hom-perpendicular} if $\Hom_\Lambda(A,B)=0=\Hom_\Lambda(B,A)$.
%We prove only the sufficiency of this condition as it is clearly necessary. Since $\Lambda$ is of finite representation type, we can number the roots so that $ext(\alpha_i,\alpha_j)=0$ for $i<j$. Then, reversing the order gives an exceptional sequence $M=(M_{\alpha_p}, \cdots,M_{\alpha_1})$ making $\cA=(^\perp M)^\perp$ into a rank $p$ wide subcategory with complete exceptional sequence $M$. Since the $\alpha_j$ are hom-orthogonal, $M_{\alpha_j}$ are the simple objects of $\cA$.

(b) We prove only the sufficiency of this condition as it is clearly necessary. Since $\Lambda$ is of finite representation type, we can number the roots so that $ext(\gamma_i,\gamma_j)=0$ for $i<j$. Then, reversing the order gives an exceptional sequence $M=(M_{\gamma_p}, \cdots,M_{\gamma_1})$ making $\cA=(^\perp M)^\perp$ into a rank $p$ wide subcategory with complete exceptional sequence $M$. Since the $\gamma_j$ are hom-orthogonal, $M_{\gamma_j}$ are the simple objects of $\cA$.
\end{proof}

We will use the notation $\cA b(\gamma_\ast)=(^\perp M)^\perp$ for the wide subcategory of part (b). By Lemma \ref{lem: equivalent conditions on gamma=sum rj aj}, $\Phi^+(\gamma_\ast)$ is the set of dimension vectors of indecomposable objects of $\cA b(\gamma_\ast)$. We call $\cA b(\gamma_\ast)$ the \emph{wide subcategory spanned by $\gamma_\ast$} since ``generated'' is not the right word.

%We will later need the following description of the valued quiver of the category $\cA$.

%We repeat: For any two hom-orthogonal roots $\alpha,\beta\in\Phi^+(Q)$, either $ext(\alpha,\beta)=0$ or $ext(\beta,\alpha)=0$.

%We say that $\alpha,\beta$ are \emph{hom-orthogonal} if they satisfy these conditions.

%We say that $\alpha, \beta$ are \emph{hom-orthogonal} if they have this property.

%In case $\alpha,\beta$ are hom-orthogonal roots, then we have either $ext(\alpha,\beta)=0$ or $ext(\beta,\alpha)=0$ where we use the common shorthand notation $ext(\alpha,\beta):=\dim_K\Ext^1_\Lambda(M_\alpha,M_\beta)$. We assume by symmetry that $ext(\alpha,\beta)=0$. Let $\cC(\alpha,\beta)$ be the smallest wide subcategory of $mod\text-\Lambda$ containing both $M_\alpha$ and $M_\beta$. Let $\Phi^+(\alpha,\beta)\subseteq \Phi^+$ be the set of dimension vectors of the indecomposable objects in $\cC(\alpha,\beta)$.

%\begin{lem}
%The set $\Phi^+(\alpha,\beta)$ consists of all positive roots in $\Phi^+(Q)$ of the form $r\alpha+s\beta$ where $r,s$ are nonnegative integers.
%\end{lem}

\subsection{Picture group} We now describe the picture group $G_0(Q)$ (Definition \ref{def of picture group of a picture}) for the semi-invariant picture $L(Q)$.

\begin{thm}\label{thm: presentation of the picture group determined by the spherical semi-invariant picture} If $Q$ is a valued Dynkin quiver, the picture group $G_0(Q)$ determined by the spherical semi-invariant picture $L(Q)$ has the presentation:
\begin{enumerate}
\item $G_0(Q)$ has one generator $x(\beta)$ for every positive root $\beta\in\Phi^+(Q)$.
\item For each pair $(\alpha,\beta)$ of hom-orthogonal roots in $\Phi^+(Q)$ so that $ext(\alpha,\beta)=0$, we have the relation:
\[%begin{equation}\label{eq: relation around rho}
	x(\alpha)x(\beta)=\prod x(r_i\alpha+s_i\beta)
\]%end{equation}
where the product is over all positive roots of the form $r_i\alpha+s_i\beta$ in increasing order of the ratio $r_i/s_i$ (going from $0/1$ to $1/0$).
\end{enumerate}
\end{thm}

\begin{proof} Each codimension one face simplex of $L(Q)$ lies in $D(\beta)$ for some positive root $\beta$ and is labeled $x(\beta)$. By Lemma \ref{lem: which D(beta) occur twice}, the relation which occurs around a codimension two simplex $\rho$ of $L(Q)$ is a word in $x(r\alpha+s\beta)$ in which the letters $x(\alpha),x(\beta)$ occur twice and the other letters occur once. In the semi-simple case where $M_\alpha,M_\beta$ do not extend each other, the only $D(\gamma)$ containing $\rho$ are $D(\alpha),D(\beta)$ which meet transversely with $\rho$ in their intersection. So, the relation around $\rho$ is $x(\alpha)x(\beta)=x(\beta)x(\alpha)$ in this case.

If $\Ext^1_\Lambda(M_\beta,M_\alpha)\neq0$ then there are extensions $M_\gamma$ where $\gamma=r\alpha+s\beta$. (Example \ref{eg: examples of relations} below gives a case by case description.) Figure \ref{fig:relation around rho} shows where $D(r\alpha+s\beta)$ occur. They are oriented counterclockwise as shown in the figure and the slope of the positive normal direction is proportional to $r/s$. Therefore, the sets $D(r\alpha+s\beta)$ are in cyclic order according to this slope and we get the relation (2).
\end{proof}

\begin{figure}[htbp]
\begin{center}
%
%\vs5
{
\setlength{\unitlength}{2cm}
%\centerline
{\mbox{
\begin{picture}(4,2)
      \thicklines
%    \thinlines
\put(1,0){
      \qbezier(-.5,1)(1,1)(2.5,1) % x-axis
     \qbezier(1,0)(1,1)(1,2) % y-axis
%     \qbezier(0,0)(1,1)(2,2) % diagonal     
     \qbezier(1,1)(1.5,.5)(2,0) % counter diagonal     
     \qbezier(1,1)(1.6,.5)(2.2,0) % counter diagonal     
     \qbezier(1,1)(1.7,.5)(2.4,0) % counter diagonal     
%     \put(0,1.1){$\eta_{ik}$}
%     \put(0.25,1.8){$D(\beta)$}
     \put(1.1,1.8){$D(\beta)$}
     \put(1.1,0.1){$D(\beta)$}
     \put(2,1.1){$D(\alpha)$}
     \put(-.4,1.1){$D(\alpha)$}
%     \put(-.2,.3){$D(\alpha)$}
     \put(2,.4){$D(r\alpha+s\beta)$}
  %
%    \put(.96,.95){$\bullet$
    \put(.85,.83){$\rho$}
%    \put(1.5,2){$\bullet$
%    \put(-.3,0){$P$}}
    }
\end{picture}}
}}
%\vs5
\caption{Image of $L(Q)$ under the projection $\RR^n\to \RR^2$ given by $v\mapsto (\<v,\beta\>,\<v,\alpha\>)$. By definition, $D(\alpha)$, $D(\beta)$ map to the $x$ and $y$ axes. In the non-semisimple case, $\<\alpha,\beta\>=0$ and $\<\beta,\alpha\>< 0$, all sets $D(r\alpha+s\beta)$ for $r,s>0$ map to the fourth quadrant as shown. The slope of these lines increase with the ratio $r/s$ and therefore, read counterclockwise, the lines in Quadrant IV are in order of $r/s$.}
\label{fig:relation around rho}
\end{center}
\end{figure}
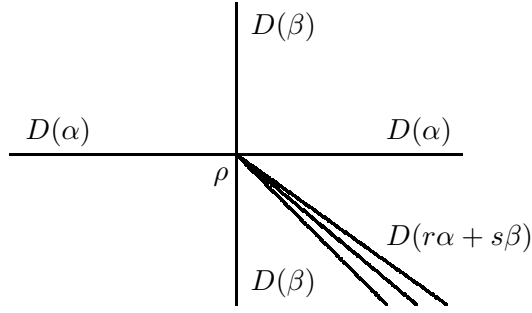

\begin{eg}\label{eg: examples of relations}
There are only six types of relations which occur in the presentation given in Theorem \ref{thm: presentation of the picture group determined by the spherical semi-invariant picture}. This is because the wide category $(^\perp M)^\perp$ is equivalent to the module category of a hereditary algebra of finite representation type with two vertices. And there are only four possibilities as listed below. (But Cases (3) and (4) have two subcases depending on whether the arrow points towards the short root or the long root. So, the total is six.)
\begin{enumerate}
\item $A_1\times A_1$. This corresponds to the case when the modules $M_\alpha,M_\beta$ do not extend each other and the wide category that they generate is semi-simple. So, $\Phi^+(\alpha,\beta)=\{\alpha,\beta\}$ and the relation is:
\[
	x(\alpha)x(\beta)=x(\beta)x(\alpha).
\]
\item $A_2$. Here $\Ext^1_\Lambda(M_\beta,M_\alpha)$ is one dimensional over both $F_\beta=\End_\Lambda(M_\beta)$ and $F_\alpha=\End_\Lambda(M_\alpha)$. The wide category has 3 indecomposable objects forming an exact sequence $M_\alpha\to M_{\alpha+\beta}\to M_\beta$ and $G_0(Q)$ has relation:
\[
	x(\alpha)x(\beta)=x(\beta)x(\alpha+\beta)x(\alpha).
\]
\item $B_2\cong C_2$. In this case, either $\Ext^1_\Lambda(M_\beta,M_\alpha)$ is $1$-dimensional over $F_\beta$ and $2$-dimensional over $F_\alpha$ or vise versa. In the first case, where $\beta$ is the long root, we have $\Phi^+(\alpha,\beta)=\{\alpha,\beta,\alpha+\beta,2\alpha+\beta\}$ and the relation is
\[
	x(\alpha)x(\beta)=x(\beta)x(\alpha+\beta)x(2\alpha+\beta)x(\alpha).
\]
\item $G_2$. Here $\Ext^1_\Lambda(M_\beta,M_\alpha)$ is 1-dimensional over $F_\beta$ and 3-dimensional over $F_\alpha$ or vise versa. There are six positive roots giving the relation:
\[
	x(\alpha)x(\beta)=x(\beta)x(\alpha+\beta)x(3\alpha+2\beta)x(2\alpha+\beta)x(3\alpha+\beta)x(\alpha).
\]
\end{enumerate}
In all cases there are irreducible morphism between the corresponding modules in the opposite order than how they appear in the relations. For example, in Case (4) there are irreducible morphisms
\[
	M_\alpha\to M_{3\alpha+\beta}\to M_{2\alpha+\beta}\to M_{3\alpha+2\beta}\to M_{\alpha+\beta}\to M_{\beta}.
\]
\end{eg}

If we compare these relations with the Chevalley relations for the generators of the maximal unipotent subgroup $U_Q$ of the algebraic group of the underlying Dynkin diagram of $Q$, we see that there is an epimorphism $G_0(Q)\onto U_Q(\ZZ)$ when $Q$ has two vertices. (Send $x(\beta)$ to $\epsilon_\beta(1)$ in the notation of \cite{Hum:AlgGrps}, section 33. Send $x(\beta)$ to $X(\beta)$ in the notation of Definition \ref{def: unipotent group of Q}.)

\begin{summ} In Section \ref{sec2} (Theorem \ref{thm: presentation of the picture group determined by the spherical semi-invariant picture}) we gave a presentation of the picture group $G_0(Q)$ which is determined by the labeled picture $L(Q)$ constructed in Section \ref{sec1}.
\end{summ}

\section{Picture space $X(Q)$}\label{sec3}

In {Section \ref{sec3}} we will construct the picture space $X(\Lambda)$ assuming that $\Lambda$ is a hereditary algebra of finite representation type. This will be a finite CW-complex together with a system of closed codimension-one subsets $J(\beta)\subset X(Q)$ for all $\beta\in\Phi^+(Q)$. Since $X(\Lambda)$ will depend only on the underlying valued quiver $Q$ of $\Lambda$, we will write $X(Q)=X(\Lambda)$.

\subsection{Local properties of $D(\beta)$}

The construction of the space $X(Q)$ depends on the local properties of the sets $D(\beta)$ as given in Proposition \ref{prop: D(beta,rho)} below for $\beta$ in a wide subcategory $\cA b(\alpha_\ast)$ spanned by a pairwise hom-orthogonal set of positive roots $\alpha_\ast=\{\alpha_1,\cdots,\alpha_p\}$. Roughly speaking, it says that the intersection pattern of these sets depends only on the valued quiver of $\cA b(\alpha_\ast)$ which we now define.

\begin{defn}\label{def: Q(alpha)}
For any set  of pairwise hom-orthogonal positive roots $\alpha_\ast=\{\alpha_1,\cdots,\alpha_p\}$, let $Q(\alpha_\ast)$ denote the quiver with one vertex for each $\alpha_j$ with valuation $f_j=hom(\alpha_j,\alpha_j)$ and an arrow $i\to j$ whenever $ext(\alpha_i,\alpha_j)\neq0$ with valuation $(d_{ij},d_{ji})$ so that $d_{ij}f_j=d_{ji}f_i=ext(\alpha_i,\alpha_j)$. Then $Q(\alpha_\ast)$ depends only on the numbers $\<\alpha_i,\alpha_j\>$ since $f_j=\<\alpha_j,\alpha_j\>$ and $ext(\alpha_i,\alpha_j)=-\<\alpha_i,\alpha_j\>$ when $i\neq j$.
\end{defn}

\begin{lem}\label{lem: partial cluster spanning perp Ab(a)}
Let $\alpha_\ast=\{\alpha_1,\cdots,\alpha_p\}\subseteq \Phi^+(Q)$ be a set of hom-orthogonal roots for the underlying valued quiver $Q$ of a hereditary algebra $\Lambda$ of finite representation type. Then
\begin{enumerate}
\item There exists a partial cluster tilting object $T=T_1\oplus\cdots\oplus T_{n-p}$ in the cluster category of $\Lambda$ so that $\cA b(\alpha_\ast)=|T|^\perp$.
\item $
\,^\perp\RR\alpha_\ast=\{w\in\RR^n\,|\,\<w,\alpha_j\>=0 \text{ for all }j\}
$ is the linear span of the roots $\undim T_i$.
\item There is a unique linear map $\pi_{\alpha_\ast}:\RR^n\to \RR\alpha_\ast$ having $\,^\perp\RR\alpha_\ast$ as kernel so that $\pi_{\alpha_\ast}$ is the identity map on $\RR\alpha_\ast$.
\item For each $x\in\RR^n$, $\pi_{\alpha_\ast}(x)\in \RR\alpha_\ast$ is the unique vector so that $\<x,\alpha_j\>= \<\pi_{\alpha_\ast}(x),\alpha_j\>$ for all $j$.
\end{enumerate}
\end{lem}

\begin{proof}
By Lemma \ref{lem: hom-orthogonal pairs of roots}, $\cA b(\alpha_\ast)$ is a wide subcategory of rank $p$. Therefore, $\,^\perp \cA b(\alpha_\ast)$ is a wide subcategory of $mod\text-\Lambda$ of rank $n-p$ which therefore has a cluster tilting object $T$ as claimed. The rest is basic linear algebra.
\end{proof}

\begin{lem}\label{lem: factorization lemma for pi}
For $\beta_\ast$ a hom-orthogonal set of roots in $\Phi^+(\alpha_\ast)$, $\pi_{\beta_\ast}:\RR^n\to\RR\beta_\ast$ factors uniquely through $\pi_{\alpha_\ast}:\RR^n\to \RR\alpha_\ast$ as $\pi_{\beta_\ast}=\left(\pi_{\beta_\ast}|_{\RR\alpha_\ast}\right)\circ\pi_{\alpha_\ast}$.
\end{lem}

\begin{proof}
Since $\RR\beta_\ast\subseteq\RR\alpha_\ast$, $\,^\perp \RR\alpha_\ast=\ker\pi_{\alpha_\ast}\subseteq\,^\perp\RR\beta_\ast=\ker\pi_{\beta_\ast}$. The lemma follows.
\end{proof}

\begin{defn}\label{def: instead of D(beta,rho)}
Let $\alpha_\ast=\{\alpha_1,\cdots,\alpha_p\}$ be any hom-orthogonal set of positive roots and let $\beta$ be any element of $\Phi^+(\Lambda)$. Define the following subsets of $\RR^n$ and $\RR\alpha_\ast$.
\[
D^{\alpha_\ast}(\beta)=\begin{cases} \{x\in\RR^n\,|\, \<x,\beta\>=0\text{ and $\<x,\beta'\>\le 0$ $\forall\beta'\subseteq\beta$ s.t. $\beta'\in\Phi^+(\alpha_\ast)$} 
	\} & \text{if } \beta\in \Phi^+(\alpha_\ast)\\
   \emptyset & \text{if }\beta\notin \Phi^+(\alpha_\ast)
    \end{cases}
    \]
$\ D_{\alpha_\ast}(\beta)=D^{\alpha_\ast}(\beta)\cap \RR\alpha_\ast$.
%If $\beta\notin\Phi^+(\alpha_\ast)$ we define $D^{\alpha_\ast}(\beta)$ and $D_{\alpha_\ast}(\beta)$ to be empty sets.
\end{defn}

\begin{rem}\label{rem: relative D}By Lemma \ref{lem: partial cluster spanning perp Ab(a)}(4), the conditions defining $D^{\alpha_\ast}(\beta)$: $\<x,\beta\>=0$ and $\<x,\beta'\>\le 0$ are equivalent to the same conditions on $\pi_{\alpha_\ast}(x)$. Therefore:

\noi$\ 
D^{\alpha_\ast}(\beta)=\pi_{\alpha_\ast}^{-1}D_{\alpha_\ast}(\beta).
$

Let $\beta_\ast$ be a hom-orthogonal set of roots in $\Phi^+(\alpha_\ast)$. If $\beta\in \Phi^+(\beta_\ast)$ then $D^{\beta_\ast}(\beta)$ contains $D^{\alpha_\ast}(\beta)$ since it is given by a subset of the conditions which define $D^{\alpha_\ast}(\beta)$. By Lemma \ref{lem: factorization lemma for pi}, $\pi_{\beta_\ast}:\RR^n\to \RR\beta_\ast$ factors through $\RR\alpha_\ast$. So, $D_{\alpha_\ast}(\beta)\subseteq \RR{\alpha_\ast}$ is contained in the inverse image of $D_{\beta_\ast}(\beta)$ under the induced map $\RR{\alpha_\ast}\to\RR{\beta_\ast}$.
\end{rem}

\begin{prop}\label{prop: D(beta,rho)}Let $\alpha_\ast=\{\alpha_1,\cdots,\alpha_p\}$ and let $T=T_1\oplus\cdots\oplus T_{n-p}$ be a partial cluster tilting object in the cluster category of $\Lambda$ so that $\cA b(\alpha_\ast)=|T|^\perp$. Let $\rho\subset \,^\perp\RR\alpha_\ast$ be the $(n-p-1)$-simplex in the cluster complex $\Sigma(\Lambda)$ spanned by the almost positive roots $\undim T_i$ and let $v$ be any point in the interior of $\rho$. Thus $v=\sum v_i \undim T_i$ where $v_i>0$ for each $i$. Then, for any $\beta\in \Phi^+(Q)$, $D^{\alpha_\ast}(\beta)$, defined above, is equal to the set of all vectors $x\in \RR^n$ so that $v+\epsilon x\in D(\beta)$ for all sufficiently small $\epsilon>0$.% In particular, this set depends only on $\beta$ and $\alpha_\ast$ and is independent of $T$ and $\rho$.
\end{prop}

\begin{proof} Let $D^{\alpha_\ast}(\beta)'$ denote the subset of $\RR^n$ defined by the $\epsilon$ condition.
Since $D(\beta)$ is a closed set, the condition $v+\epsilon x\in D(\beta)$ for small $\epsilon$ implies that $v\in D(\beta)$. This implies that $\rho\subset D(\beta)$ which holds if and only if $\beta\in \Phi^+(\alpha_\ast)$ by Lemma \ref{lem: equivalent conditions on gamma=sum rj aj}. So, $D^{\alpha_\ast}(\beta)'$ is nonempty only in this case.

Now assume that $\beta\in \Phi^+(\alpha_\ast)$ so that $D^{\alpha_\ast}(\beta)$ and $D^{\alpha_\ast}(\beta)'$ are both nonempty.

Let $x\in D^{\alpha_\ast}(\beta)'$. Then $v+\epsilon x\in D(\beta)$ for small $\epsilon>0$. So $\<v+\epsilon x, \beta\>=0$ and $\<v+\epsilon x,\beta'\>\le0$ for all $\beta'\subseteq \beta$. Since $\<v,\beta\>=0=\<v,\beta'\>$ these conditions imply $\epsilon\<x,\beta\>=0$ and $\epsilon\<x,\beta'\>\le0$. Since $\epsilon>0$, this implies $x\in D^{\alpha_\ast}(\beta)$. 

Conversely, let $x\in D^{\alpha_\ast}(\beta)$. Then $\<v+\epsilon x,\beta\>=\<v,\beta\>+\epsilon\<x,\beta\>=0$. Let $\gamma\subset\beta$ be any subroot. If $\gamma\in\Phi^+(\alpha_\ast)$ then $\<v+\epsilon x,\gamma\>\le0$ for all $\epsilon>0$. If $\gamma\notin\Phi^+(\alpha_\ast)$ then we know that $v\notin D(\gamma)$. Therefore $\<v,\gamma\><0$. (Otherwise, $\<v,\gamma\>=0$ and $\<v,\gamma'\>\le0$ for all $\gamma'\subseteq\gamma\subset \beta$ making $v\in D(\gamma)$.) So, $\<v+\epsilon x,\gamma\><0$ for sufficiently small $\epsilon$ which implies that $v+\epsilon x$ lies in $D(\beta)$ for sufficiently small positive $\epsilon$. So, $x\in D^{\alpha_\ast}(\beta)'$ and the two sets are equal.
\end{proof}

%\begin{eg}For example, take $n=3,p=1$, $Q=A_3:1\ot 2\ot 3$, $\alpha_\ast=\beta_{03}=(1,1,1)$, the dimension vector of the projective module $P_3$. Then there are two possibilities for $T$: \end{eg}

Recall that the cluster complex $\Sigma(Q)=\Sigma(\Lambda)$ is a simplicial complex whose geometric realization is $|\Sigma(Q)|\cong S^{n-1}$. Since $S^{n-1}$ is a manifold, the dual cell decomposition is an $(n-1)$-dimensional CW complex. We attach a single $n$-cell to this dual cell complex to get an $n$-dimensional CW complex which we denote by $E(Q)$. Thus $E(Q)\cong D^n$. Proposition \ref{prop: D(beta,rho)} gives us an equivalence between certain $p$-cells in this cell decomposition. %By identifying these equivalent cells we will obtain the picture space $X(Q)$.

\begin{figure}[ht]
\begin{center}
\begin{tikzpicture}[scale=.6]
%\draw[help lines=1,thin] (-8,-5) grid (7,2);
%\draw[help lines=.2,color=red] (-6,-6) grid (6,3);
%\foreach \x in {-6,...,6}\draw (\x,0) node{\x};\foreach \y in {-5,...,2}\draw (0,\y) node{\y};
%
\clip (-8,-6) rectangle (7,3);
\begin{scope}
\draw[very thick] (0,0)--(0,-2)--(-1.5,1.8) (4,2)--(0,-2);
\draw[very thick] (-4,-1)--(0,-2);
\draw[very thick] (1.5,-1.5)--(-4,-1);
\draw[very thick] (4,2)--(1.5,-1.5)--(0,-2);
\end{scope}
\begin{scope}[xshift=2cm,yshift=1cm]
\draw[fill=gray!15!white] (-1.25,-1.75)--(-.17,-.83)--(0,0)--(-1.17,.17)--(-2.75,-.15)--cycle;
\draw[fill,thick] (-1.17,.17) circle[radius=2pt]--(0,0);
\draw[fill,thick] (-.17,-.83) circle[radius=2pt]--(0,0);
\draw[thick] (-2,-1)--(-.17,-.83)--(-1.25,-1.75);
\draw[thick] (-2,-1)--(-1.17,.17)--(-2.75,-.15);
\end{scope}
\begin{scope}[xshift=-2cm,yshift=-.5cm]
\draw[fill=gray!15!white] (0,0)--(.17,.67)--(1.25,1.35)--(2.75,-.25)--(1.17,-.33);
\end{scope}
\begin{scope}%[xshift=-2cm,yshift=-1cm]
\begin{scope}[xshift=-2cm,yshift=-.5cm]
\draw[fill,thick] (.17,.67) circle[radius=2pt]--(2,.5);
\draw[fill,thick] (1.17,-.33) circle[radius=2pt]--(2,.5);
\draw[thick] (0,0)--(.17,.67)--(1.25,1.35) ;
\draw[thick] (0,0)--(1.17,-.33)--(2.75,-.25) ;
\end{scope}
\begin{scope}%[xshift=2cm,yshift=.5cm]
\draw[thin,color=black] (0,0)--(-4,-1)--(1.5,-1.5)--cycle;
\draw[thin,color=black] (0,0)--(-1.5,1.8)--(-4,-1);
\draw[thin,color=black] (0,0)--(4,2)--(-1.5,1.8);
%\draw[thick,color=blue] (4,2)--(1.5,-1.5);
\draw[fill,thick] (.75,-.75) circle[radius=2pt]--(0,0);
\draw[fill,thick] (-.75,.85) circle[radius=2pt]--(0,0);
\draw[fill,color=brown] (0,0) circle[radius=2pt];
\draw[fill,thick,color=brown] (-2,-.5) circle[radius=2pt]--(0,0);
\draw[fill,thick,color=brown] (2,1) circle[radius=2pt]--(0,0);
\end{scope}
\end{scope}
\begin{scope}
\draw[very thick] (0,2.5)--(-1.5,1.8);
%\draw[very thick] (0,0)--(0,2.5)--(4,2); %(0,0)--(0,2.5)--(-1.5,1.8)
\draw[very thick,color=red] (-1.5,1.8)--(4,2); % back edge of Lk(rho)
\draw[very thick,color=red] (-1.5,1.8)--(-4,-1); % left edge of Lk(rho)
\draw[very thick] (-4,-1)--(0,2.5);
\draw[very thick,color=red] (1.5,-1.5)--(-4,-1); % front edge of Lk(rho)
\draw[very thick] (1.5,-1.5)--(0,2.5);
\draw[very thick,color=red] (4,2)--(1.5,-1.5); % right edge of Lk(rho)
\end{scope}
\begin{scope}[xshift=0cm,yshift=-4cm] % lower figure
%\draw (-2,.5) node{$S^{p-1}$};
\draw[thick,<-] (-1.7,.3)--(-3,1) node[above]{$S^{p-1}$};
\draw (3,-.5) node[right]{$\RR\alpha_\ast\cong \RR^p$};
\draw[thick,color=red] (-4,-1)--(1.5,-1.5)--(4,1)--(-1.5,1.5)--cycle;
\draw[very thick,color=blue] (-4.4,-1.1)--(4.8,1.2);
\draw (5,.7) node{$D_{\alpha_\ast}(\beta)$};
\draw[very thick,color=black] (-1.7,1.7)--(1.7,-1.7);
\draw[fill] (.8,-.8) circle[radius=2pt];
\draw[fill] (-.8,.8) circle[radius=2pt];
\draw[fill] (1.92,.47) circle[radius=2pt];
\draw[fill] (-1.92,-.47) circle[radius=2pt];
\begin{scope}[rotate=10]
\draw (0,0) ellipse [x radius=2cm,y radius=1cm];
\end{scope}
\end{scope}
\draw[<-,thick] (0,1.3).. controls (-1.5,1.3) and (-1.5,3)..(-3,2);
\draw (-3,2) node[left]{$\rho$};
\draw[<-,thick] (-1.1,.3).. controls (-1.4,1) and (-2.5,1)..(-3.2,1);
\draw (-3.2,1) node[left]{$E(\rho)$};
\draw[<-,thick] (-1.5,-.37).. controls (-2,.1) and (-2.5,.1)..(-3.7,0);
\draw (-3.7,0) node[left]{$J(\beta,\rho)$};
\draw[thick,->] (3,-1)--(3,-2); 
\draw (3,-1.5) node[right]{$\pi_{\alpha_\ast}$};
\begin{scope}[xshift=.5cm,yshift=-1.75cm]
\draw[color=red] (-5,0) node[left]{$Lk(\rho)$};
\draw[thick,<-] (-2,.5)..controls (-2.2,0) and (-3,0)..(-5,0) ;
\end{scope}
\begin{scope}[xshift=.5cm,yshift=-3.5cm]
\draw (-5,0) node[left]{$\pi_{\alpha_\ast}(Lk(\rho))$};
\draw[thick,<-] (-3.5,-.5)..controls (-4,0) and (-5,0)..(-5,0) ;
\end{scope}
\draw[very thick] (0,0.07)--(0,2.5)--(4,2); %(0,0)--(0,2.5)--(-1.5,1.8)
\end{tikzpicture}
\caption{The dual cell $E(\rho)$ projects homeomorphically onto a $p$-cell in $\RR\alpha_\ast\cong \RR^p$. The link $Lk(\rho)$ of $\rho$ maps to a triangulated $(p-1)$-sphere in $\RR\alpha_\ast$ which by normalization (dividing by lengths of vectors) maps to the unit sphere $S^{p-1}$ giving the isomorphism $Lk(\rho)\cong \Sigma(Q(\alpha_\ast))$ of Corollary \ref{cor: Lk(r)=S(Q(a))}.}
\label{fig: dual cell and projection}
\end{center}
\end{figure}
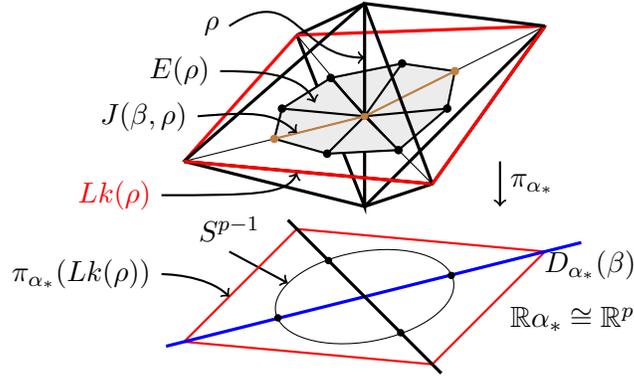

\begin{defn} We define $E(\rho)$, $J(\beta,\rho)$ and $Lk(\sigma)$. 

The \emph{dual cell} $E(\rho)$ to the $(n-p-1)$-simplex $\rho$ in $|\Sigma(Q)|$ is the $p$-dimensional triangulated space (a $p$-cell by Remark \ref{rem: naturality of varphi}(b)) which is the union of all simplices $\tau$ in the first barycentric subdivision of $|\Sigma(Q)|$ so that $\tau\cap \rho$ is the barycenter of $\rho$. This implies that the other vertices of $\tau$ are barycenters of simplices $\sigma$ which contain $\rho$. 

For every $\beta\in\Phi^+(Q)$ let $J(\beta,\rho)=E(\rho)\cap D(\beta)$. This is the subcomplex of $E(\rho)$ consisting of all simplices $\tau$ whose vertices are barycenters of simplicies $\sigma$ which are contained in $D(\beta)$. Since $D(\beta)$ meets $E(\rho)$ if and only if $D(\beta)$ contains $\rho$, it follows from Lemma \ref{lem: equivalent conditions on gamma=sum rj aj} that $J(\beta,\rho)$ is nonempty if and only if $\beta\in\Phi^+(\alpha_\ast)$.

We use the general fact that $E(\rho)$ is simplicially isomorphic to the cone on the first barycentric subdivision of the {link} of $\rho$ in $\Sigma(Q)$. Recall that the \emph{closed star} of a simplex $\sigma$ in any simplicial complex $K$ is defined to be the union of the set of all simplices $\tau$ of $K$ which contain $\sigma$ and the \emph{link} $Lk(\sigma)$ of $\sigma$ is the union of all simplices in the star of $\sigma$ which are disjoint from $\sigma$. In the case at hand, the link of $\rho$ is the simplicial subcomplex of $\Sigma(Q)$ whose vertices are the almost positive roots $\gamma_i$ which are ext-orthogonal to the vertices $\undim T_1,\cdots,\undim T_{n-p}$ of $\rho$. These vertices $\gamma_i$ span a simplex in $Lk(\rho)$ if and only if they are ext-orthogonal.
\end{defn}

\begin{eg}\label{eg: A2 figure}
We illustrate these terms and concepts on the quiver $A_2: 1\ot 2$.
\begin{enumerate}
\item $\Sigma(A_2)$ is a pentagon with five vertices and 5 edges. These vertices are the almost positive roots $\alpha,\beta,\gamma,-\alpha,-\beta$ connected in a cycle. The simple roots are $\alpha,\gamma$. This is the left part of Figure \ref{Fig:A2 pentagon}.
\item Figure \ref{Fig:A2 pentagon}, right represents the cone on the first barycentric subdivision of $\Sigma(A_2)$. Except for the cone point $\emptyset$, each point $b_\rho$ is the barycenter of a simplex $\rho$ in $\Sigma(A_2)$. The point $b_\rho$ is labeled by the set of vertices of $\rho$.
\item The entire object (solid pentagon) is $E(\emptyset)$ and $E(\alpha),E(\beta),E(\gamma),E(-\alpha),E(-\beta)$ are the stars of the original vertices. The new vertices (black spots in Figure \ref{Fig:A2 pentagon}) are the dual cells $E(\rho_i)$ of the original 2-simplices $\rho_i$ of $\Sigma(A_2)$ (Figure \ref{Fig:A2 pentagon}, left). 
%\item Since $|\rho_i|^\perp=0$, these vertices are all identified to one point $e^0$ in $X(A_2)$. The 1-cells $E(\alpha),E(-\alpha)$ are identified since $|\alpha|^\perp=|-\alpha|^\perp=\cA b(\gamma)$ and $E(\beta),E(-\beta)$ are also identified since $\beta^\perp=\cA b(\alpha)$. Finally, $\gamma^\perp=\cA b(\beta)$. 
%\item $J(\alpha)$ is the cone on the two points $\beta,-\beta$, $J(\beta)$ is the cone on the point $\gamma$ and $J(\gamma)$ is the cone on the points $\alpha,-\alpha$.
\end{enumerate}
\end{eg}
\begin{figure}[ht]
\begin{center}
\begin{tikzpicture}[scale=.7]
%\draw[help lines=1,thick] (-6,-1.7) grid (8,3.7);
%\foreach \x in {-2,...,5}\draw (\x,0) node{\x};\foreach \y in {-2,...,4}\draw (0,\y) node{\y};
%\draw[thick,color=blue] (0,1) ellipse [x radius=2.8cm,y radius=2.1cm];
\clip (-6,-1.7) rectangle (8,3.7);
\draw[color=blue] (2,0) circle[radius=3pt] (6,1)circle[radius=3pt];
\draw[color=red] (2,2) circle[radius=3pt] (4,-1)circle[radius=3pt];
\draw (4,3) circle[radius=3pt];
\begin{scope}[xshift=-5cm,scale=.7]
	\draw (0,0) node[left]{$\alpha$};
	\draw (0,2) node[left]{$\beta$};
	\draw (2,-1) node[below]{$-\beta$};
	\draw (2,3) node[above]{$\gamma$};
	\draw (4,1) node[right]{$-\alpha$};
	\draw[thick] (0,0)--(0,2)--(2,3)--(4,1)--(2,-1)--(0,0);
\end{scope}
\begin{scope}[xshift=2cm,scale=1]
\draw (1.7,1)-- (0,1) node[left]{$\{\alpha,\beta\}$};
\draw[fill] (0,1) circle[radius=3pt];
\draw (1.7,1)-- (1,-.5)(1,-.7) node[left]{$\{\alpha,-\beta\}$};
\draw[fill] (1,-.5) circle[radius=3pt];
\draw (1.7,1)-- (1,2.5)(1,2.7) node[left]{$\{\beta,\gamma\}$};
\draw[fill] (1,2.5) circle[radius=3pt];
\draw (1.7,1)-- (3,2) (3,2.2) node[right]{$\{\gamma,-\alpha\}$};
\draw[fill] (3,2) circle[radius=3pt];
\draw (1.7,1)-- (3,0)(3,-.2) node[right]{$\{-\beta,-\alpha\}$};
\draw[fill] (3,0) circle[radius=3pt];
\draw (1.7,1)-- (0,0);
	\draw (0,0) node[left]{$\alpha$} ;
\draw[very thick,color=red] (1.7,1)-- (0,2);
	\draw (0,2) node[left]{$\beta$};
\draw[very thick,color=red] (1.7,1)-- (2,-1);
	\draw (2,-1) node[below]{$-\beta$};
\draw (1.7,1)-- (2,3);
	\draw (2,3) node[above]{$\gamma$};
\draw (1.7,1)-- (4,1);
	\draw (4,1) node[right]{$-\alpha$};
	\draw[thick] (0,0)--(0,2)--(2,3)--(4,1)--(2,-1)--(0,0);
	\draw[very thick,color=blue] (1,-.5)--(0,0)--(0,1);
	\draw[color=blue] (-1.4,0)node[left]{$E(\alpha)$}
	(3.6,0.4)node[right]{$E(-\alpha)$};
	\draw[color=blue,->] (-1.4,0)--(-.1,.5);
	\draw[color=red] (-1.4,2)node[left]{$J(\alpha,\emptyset)$};
\draw[color=red,->] (-1.4,2)..controls (-.8,1.6) and (.2,1.2)..(.7,1.5);
\draw[very thick,color=blue](3,2)--(4,1)--(3,0);
\draw (1.7,1) node{$\emptyset$};%circle[radius=3pt];
\end{scope}
\end{tikzpicture}
\caption{$\Sigma(A_2)$ is the (boundary of the) pentagon on the left. The right hand figure is the cone on the barycentric subdivision of $\Sigma(A_2)$.}\label{fig: two pentagons}\label{Fig:A2 pentagon}
\end{center}
\end{figure}
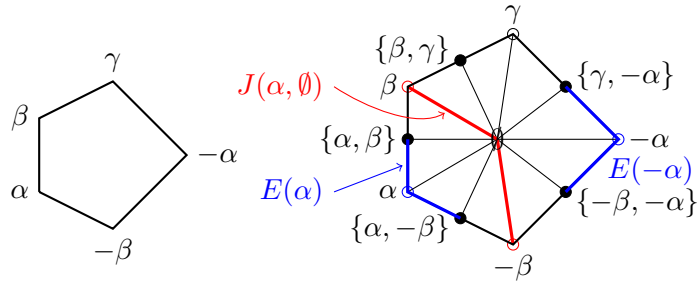
\begin{cor}\label{cor: Lk(r)=S(Q(a))} %As in Proposition \ref{prop: D(beta,rho)}, l
Let $\alpha_\ast=\{\alpha_1,\cdots,\alpha_p\}$ be a set of hom-orthogonal roots of $Q$, let $T=T_1\oplus\cdots\oplus T_{n-p}$ be a partial cluster tilting object so that $\cA b(\alpha_\ast)=|T|^\perp$. Let $\rho$ be the $(n-p-1)$-simplex in $\Sigma(\Lambda)$ spanned by $\undim T_i$. Then there is a simplicial isomorphism \[\varphi_\rho: Lk(\rho)\cong \Sigma(Q(\alpha_\ast))\] which is uniquely determined by the property that, for every vertex $\gamma$ of $Lk(\rho)$, $\varphi_\rho(\gamma)$ is a positive scalar multiple of $\pi_{\alpha_\ast}(\gamma)\in\RR\alpha_\ast$. 
\end{cor}

\begin{proof}
By definition of $Lk(\rho)$, a root $\gamma$ of $Q$ lies in $Lk(\rho)$ if and only if $M_\gamma\oplus T_1\oplus \cdots\oplus T_{n-p}$ is a partial cluster tilting object. This can be completed to a complete cluster tilting object by adding $p-1$ summands. By Lemma \ref{lem: faces of top dim simplices are in D(|c-vector|)}, we obtain $p-1$ roots $\beta_i$ so that $D(\beta_i)$ contains $\rho$, $\gamma$ and all but one of these new summands. This implies that $v+\epsilon( \gamma-v)$ lies in each of these $D(\beta_i)$ for all small $\epsilon>0$ where $v\in int\,\rho$. So, $\gamma-v\in D^{\alpha_\ast}(\beta)=\pi_{\alpha_\ast}^{-1}D_{\alpha_\ast}(\beta_i)$. Since $\pi_{\alpha_\ast}(v)=0$,  this implies that $\pi_{\alpha_\ast}(\gamma)$ lies in each $D_{\alpha_\ast}(\beta_i)$. So, it is a scalar multiple of an almost positive root in $\Phi(\alpha_\ast)$ which we define to be $\varphi_\rho(\gamma)$.

To see that $\varphi_\rho$ takes simplices to simplices, take a maximal simplex in $Lk(\rho)$ spanned by $p$ vertices $\gamma_1,\cdots,\gamma_p$. Then, for sufficiently small $\epsilon_j>0$, we have that $v+\sum \epsilon_j (\gamma_j-v)$ does not lie in $D(\beta_i)$ for any $\beta_i$ since it lies in the interior of a top dimensional simplex of the cluster complex. This condition characterizes which sets of vertices in $Lk(\rho)$ form a simplex.

By Proposition \ref{prop: D(beta,rho)}, this implies that, when $r_j>0$, $\sum r_j \varphi_\rho(\gamma_j)$ does not lie in $D_{\alpha_\ast}(\beta_i)$ for any $\beta_i\in\Phi^+(\alpha_\ast)$. This is equivalent to the condition that $\varphi_\rho(\gamma_j)$ are ext-orthogonal and therefore form a simplex in $\Sigma(Q(\alpha_\ast))$. So, $\varphi_\rho$ is a simplicial isomorphism.
\end{proof}

\begin{rem}\label{rem: naturality of varphi}
(a) By Proposition \ref{prop: D(beta,rho)}, Corollary \ref{cor: Lk(r)=S(Q(a))} implies that a vertex $\gamma$ of $Lk(\rho)$ lies in $D(\beta)$ for some $\beta\in\Phi^+(\alpha_\ast)$ if and only if $\varphi_\rho(\gamma)$ lies in $D_{\alpha_\ast}(\beta)$.

(b) Since $E(\rho)$ is homeomorphic to the cone on $|Lk(\rho)|\cong|\Sigma(Q(\alpha_\ast))|\cong S^{p-1}$ which is a $p$-disk, this implies that $E(\rho)$ is a $p$-cell.

(c) We also obtain as a consequence the following naturality condition on $\varphi_\rho$. Let $\tau$ be a simplex in $Lk(\rho)$ and $\sigma=\rho\ast\tau$, the \emph{join} of $\rho$ and $\tau$, which in this case is the smallest simplex in $\Sigma(Q)$ containing $\rho$ and $\tau$. Then $Lk(\sigma)\subseteq Lk(\rho)$. Take $\varphi_\sigma:Lk(\sigma)\cong \Sigma(Q(\beta_\ast))$ where $\beta_\ast$ gives the simple objects in $|\sigma|^\perp$. Let $\tau'=\varphi_\rho(\tau)$ and let $Lk'(\tau')$ be the link of $\tau'$ in $\Sigma(Q(\alpha_\ast))$. Then, $\beta_\ast$ also gives the simple objects in the right perpendicular category of $|\tau'|$ in $\cA b(\alpha_\ast)$. So, we get two isomorphisms $Lk(\sigma)\cong \Sigma(Q(\beta_\ast))$. We need to know that they agree. Equivalently, the following diagram commutes.
\[
%\xymatrixrowsep{10pt}\xymatrixcolsep{10pt}
\xymatrix{%begin xy matrix
&Lk(\sigma)\ar[d]^{\varphi_\rho|_{Lk(\tau)}}\ar[r]^\subset\ar[dl]_{\varphi_\tau} &
	Lk(\rho)\ar[d]\\
\Sigma(Q(\beta_\ast))&Lk'(\sigma)\ar[l]_(.4){\varphi_{\tau'}}\ar[r]^\subset& 
	\Sigma(Q(\alpha_\ast))
	}%end xy matrix
\]
It suffices to show that the triangle commutes. But this follows from Corollary \ref{cor: Lk(r)=S(Q(a))} above since each vertex $\gamma$ of $Lk(\sigma)$ maps by $\varphi_\tau$ to the unique vertex of $\Sigma(Q(\beta_\ast))$ which is proportional to $\pi_{\beta_\ast}(\gamma)$ which comes from $\pi_{\alpha_\ast}(\gamma)\propto\varphi_\rho(\gamma)$ by  Remark \ref{rem: relative D}.
\end{rem}

\begin{cor}
Let $\rho,\rho'$ be two $(n-p-1)$-simplices spanned by the dimension vectors of the components of two partial cluster tilting object $T=T_1\oplus\cdots\oplus T_{n-p}$ and $T'=T_1'\oplus\cdots\oplus T_{n-p}$ in the cluster category of $\Lambda$ so that $\cA b(\alpha_\ast)=|T|^\perp=|T'|^\perp$. Then there is a simplicial isomorphism $\psi_\rho:E(\rho)\cong E(\rho')$ which sends $J(\beta,\rho)$ onto $J(\beta,\rho')$. Furthermore, if $\rho\subseteq\sigma=\rho\ast\tau$ so that $E(\sigma)\subset E(\rho)$, then the isomorphism $\psi_\rho$ restricts to the isomorphism $\psi_\sigma:E(\sigma)\cong E(\sigma')$ where $\sigma'=\rho'\ast\tau$. We also note that $J(\beta,\sigma)=J(\beta,\rho)\cap E(\sigma)$.
\end{cor}

\begin{proof} 
We use the general fact that $E(\rho)$ is the cone on the first barycentric subdivision of the \emph{link} $Lk(\rho)$ of $\rho$ in $\Sigma(Q)$. By Corollary \ref{cor: Lk(r)=S(Q(a))}, $Lk(\rho),Lk(\rho')$ are both isomorphic to $\Sigma(Q(\alpha_\ast))$. So, $Lk(\rho)\cong Lk(\rho')$ and, therefore, $E(\rho)\cong E(\rho')$. We refer to the elements of $E(\rho)$ corresponding to vertices of $Lk(\rho)$ as the \emph{corners} of the cell $E(\rho)$. (The vertex $\gamma$ of $Lk(\rho)$ corresponds to the barycenter of the $(n-p)$-simplex $\rho\ast\gamma$.)

By Remark \ref{rem: naturality of varphi}(a), the set $J(\beta,\rho)\subseteq E(\rho)$ is the cone on the inverse image of the subsets $D_{\alpha_\ast}(\beta)\subseteq \Sigma(Q(\alpha_\ast))$ under the isomorphism $\partial E(\rho)\cong Lk(\rho)\cong \Sigma(Q(\alpha_\ast))$. So, $\psi_\rho:E(\rho)\to E(\rho')$ must send $J(\beta,\rho)$ to $J(\beta,\rho')$.

By Remark \ref{rem: naturality of varphi}(c), the subset $Lk(\rho\ast\tau)\subseteq Lk(\rho)$ maps to $Lk(\rho'\ast\tau)$ under $\varphi_{\rho'}^{-1}\varphi_\rho$ and the induced map is equal to $\varphi_{\rho'\ast\tau}^{-1}\varphi_{\rho\ast\tau}$. The two maps agree on where they send each vertex of $Lk(\rho\ast\tau)$. Therefore they agree on where they send each corner of $E(\rho\ast\tau)$ under the map $\psi_{\rho\ast\tau}$. So, $\psi_{\rho\ast\tau}$ agrees with $\psi_\rho$.
\end{proof}

\subsection{Construction of the picture space}

\begin{defn}\label{def: picture space}
The \emph{picture space} $X(Q)$ is defined to be the CW complex obtained from $E(Q)$ by identifying $p$-cells $\psi_\rho:E(\rho)\cong E(\rho')$ using the simplicial isomorphisms given by the corollary above. The compatibility of the map $\psi_\rho$ with $\psi_{\sigma}:E(\sigma)\cong E(\sigma')$ that we just proved implies that the identifications on the $p$-cells agrees with the identifications on lower cells. So, $X(Q)$ is a well defined CW complex constructed one cell at a time by induction on dimension of cells.
\end{defn}

\begin{thm}\label{thm: picture space}
The picture space $X(Q)$ is an $n$-dimensional CW complex with one cell of dimension $k$ for every set $\alpha_\ast=\{\alpha_1,\cdots,\alpha_k\}$ of pairwise hom-orthogonal roots in $\Phi^+(Q)$. Denote it $e^k_{\alpha_\ast}$. The closure of the cell $e^k_{\alpha_\ast}$ is a subcomplex of $X(Q)$ isomorphic to $X(Q(\alpha_\ast))$. In particular, $X(Q)$ is the closure of the single cell $e^n_{\epsilon_1,\cdots,\epsilon_n}$ where $\epsilon_i$ are the simple roots in $\Phi^+(Q)$. The cell $e^p_{\beta_\ast}$ is in the boundary of $e^k_{\alpha_\ast}$ if and only if $\beta_\ast\subseteq\Phi^+(\alpha_\ast)$ and $p<k$.
\end{thm}

\begin{proof} The cell $e^k_{\alpha_\ast}$ is the one obtained by identifying all $E(\rho)$ where $\cA b(\alpha_\ast)=|\rho|^\perp$, equivalently $\rho$ is a cluster tilting object inside the cluster category of $\,^\perp\cA b(\alpha_\ast)$. So, the cells of $X(Q)$ are indexed by all such sets $\alpha_\ast$ which are the spanning sets of wide subcategories $\cA b(\alpha_\ast)$. When $\rho\subset\sigma$ then $\sigma=\rho\ast\tau$ and $E(\sigma)\subset E(\rho)$. So, $|\sigma|^\perp=\cA b(\beta_\ast)\subseteq \cA b(\alpha_\ast)$. And conversely ($|\sigma|^\perp\subseteq |\rho|^\perp$ iff $\rho\subseteq \sigma$). But $E(\rho)$ is the cone on $\partial E(\rho)\cong sd\,Lk(\rho)\cong sd\,\Sigma(Q(\alpha_\ast))$ by Corollary \ref{cor: Lk(r)=S(Q(a))}. And, $X(Q(\alpha_\ast))$ is obtained from $sd\,\Sigma(Q(\alpha_\ast))$ by the same recipe as $X(Q)$: 
\[
X(Q(\alpha_\ast))=\coprod_{\tau} E_{\alpha_\ast}(\tau)/\sim
\]
where $E_{\alpha_\ast}(\tau)$ is the cell in $sd\,\Sigma(Q(\alpha_\ast))$ dual to $\tau$. The subscript indicates that we are working in the quiver $Q(\alpha_\ast)$. In the larger quiver, we have $E_{\alpha_\ast}(\tau)=E(\rho\ast\tau)$. Thus the cells of $X(Q(\alpha_\ast))$ correspond to those cells of $X(Q)$ which are identified with $E(\rho\ast\tau)$ for various $\tau$. These are exactly the simplices which contain $\rho$ and the cells are identified in the same way in both cell complexes by Remark \ref{rem: naturality of varphi}(b). 
\end{proof}

%%%%%%%%%%%

\begin{eg}
Continuing with Example \ref{eg: A2 figure} of the quiver $A_2: 1\ot 2$:
\begin{enumerate}
\item[(4)] Since $|\rho_i|^\perp=0$, the new vertices (spots) are all identified to one point $e^0$ in $X(A_2)$. The 1-cells $E(\alpha),E(-\alpha)$ (blue in Figures \ref{Fig:A2 pentagon} and \ref{Fig:XA2}) are identified since $|\alpha|^\perp=|-\alpha|^\perp=\cA b(\gamma)$ and $E(\beta),E(-\beta)$ are also identified since $\beta^\perp=\cA b(\alpha)$. After identifications, these are no longer disks and they are labeled $e_\alpha^1,e_\beta^1$ in Figure \ref{Fig:XA2}. Finally, $\gamma^\perp=\cA b(\beta)$ and the line segment $E(\gamma)$ becomes the loop $e_\gamma^1$. 
\item[(5)] $J(\alpha)=J(\alpha,\emptyset)$ (red in Figures \ref{Fig:A2 pentagon} and \ref{Fig:XA2}) is the cone on the two points $\beta,-\beta$, $J(\beta)$ is the cone on the point $\gamma$ and $J(\gamma)$ is the cone on the points $\alpha,-\alpha$.
%\item[(6)] 
\end{enumerate}
\begin{figure}[htbp]
\begin{center}
%
%\vs5
%\vs5
%\caption{default}
%\begin{center}
\begin{tikzpicture}[scale=.7]
%\draw[help lines=1,thick] (-5,-3) grid (5,3);
%\foreach \x in {-7,...,8}\draw (\x,0) node{\x};\foreach \y in {-3,...,3}\draw (0,\y) node{\y};
%\draw[thick,color=blue] (0,1) ellipse [x radius=2.8cm,y radius=2.1cm];
\draw[color=blue] (6.9,1) circle[radius=3pt] (-4.1,1) circle[radius=3pt];
\draw[color=red] (-2,-1) circle[radius=3pt];
\draw[color=blue] (-6,0.5) node[left]{$e_\alpha^1$};
\draw[color=blue] (7,-1) node[right]{$e_\alpha^1$};
\draw (-3.2,-1) node[above]{$e_\beta^1$};
%\draw (3.2,-1) node[below]{$e(-\beta)$};
\draw[color=red] (-.8,1.6) node[left]{$J(\alpha)$};
\draw[very thick, color=blue] (6,0) ellipse[x radius=1cm,y radius=2.5cm]; % right blue circle
\begin{scope}%[xshift=-4cm,yshift=10cm] left blue circle
\clip (-5,-3) rectangle (-7,3);
	\draw[very thick, color=blue] (-5,0) ellipse[x radius=1cm,y radius=2.5cm];
\end{scope}
\begin{scope}%[xshift=-4cm,yshift=10cm]
\clip (-5,-3) rectangle (-3,3);
	\draw[thin, color=blue] (-5,0) ellipse[x radius=1cm,y radius=2.5cm];
\end{scope}
\draw[color=blue] (-4,.8) node[left]{$\alpha$};
\draw[color=blue] (7,1) node[right]{$-\alpha$};
\draw[thick] (-5.9,-1) .. controls (-4,2) and (-3,2.5) .. (-2,2.5);
\draw[thin] (-2,2.5) .. controls (-1,2.5) and (0,2)..(.9,1);
\draw[thick] (-5.9,-1) .. controls (-4.5,-2) and (-4,-2.5) .. (-3,-2.5);
\draw[thin] (-3,-2.5) ..controls (-2,-2.5) and (-1,-2)..(.9,1);
\begin{scope}[xshift=4cm] % red circle
\begin{scope}%[xshift=-4cm,yshift=10cm]
%\clip (-4,-3) rectangle (-6,3);
%	\draw[thick, color=red] (-4,0) ellipse[x radius=1cm,y radius=2.5cm];
\draw[thick,color=red] (-6,-1)..controls (-5.5,1)and (-4.5,2.5)..(-4,2.5);
\draw[thick,color=red] (-6,-1)..controls (-5.5,-1.5)and (-4.5,-2.5)..(-4,-2.5);
\end{scope}
\begin{scope}%[xshift=-4cm,yshift=10cm]
\clip (-4,-3) rectangle (-2,3);
	\draw[thin, color=red] (-4,0) ellipse[x radius=1cm,y radius=2.5cm];
\end{scope}
\end{scope}
\draw (-5,2.5)--(6,2.5);
\draw (-5,-2.5)--(6,-2.5);
\begin{scope}[xshift=.1cm,yshift=.7cm]
\draw[ thick] (-6,-1.7)--(5,-1.7);
\draw[fill] (-6,-1.7) circle[radius=3pt]node[left]{$e^0$};
\draw[fill] (5,-1.7) circle[radius=3pt]node[right]{$e^0$};
\end{scope}
\draw[color=red] (-2,-.7) node[right]{$\beta$};
\draw[color=red] (-2,-1.3) node[left]{$-\beta$};
\draw[thin] (-4.1,1)--(5.1,1);
\draw[thick] (5.1,1)--(6.9,1);
\draw (.9,1) node{$\emptyset$};
%
% gamma loop
\begin{scope}[xshift=5.1cm,yshift=-1cm]
\draw[thick,fill,color=gray!5!white] (0,0) .. controls (-2,4) and (-4,2) .. (0,0);
\draw[thick] (0,0) .. controls (-2,4) and (-4,2) .. (0,0);
\end{scope}
\draw[thick] (2.8,1)--(3.65,1);
\draw[thick] (2,2.5) .. controls (2.2,2.5) and (2.5,2.3)..(2.9,1.25);
\draw (2.9,1.28) circle[radius=3pt] ;
\draw[thin] (2,2.5)..controls (1.8,2.5) and (1.6,2.3) ..(.9,1);
\draw[thin] (.9,1)..controls (2.3,2.3) and (2.8,2.5)..(3.2,2.5);
\draw[thick] (3.2,2.5)..controls (3.7,2.5) and (4.2,2.3)..(5.1,-1);
\draw[thick] (5.1,-1) .. controls (4.5,-2) and (4,-2.5)..(3.5,-2.5);
\draw[thin](3.5,-2.5)..controls (3,-2.5) and (2.5,-2.3).. (.9,1);
\draw[thin] (.9,1)..controls (1.1,2.3) and (1.2,2.5)..(1.3,2.5);
\draw[thick] (1.3,2.5)..controls (1.7,2.5)and (2.5,-.8)..(5.1,-1);
\draw (2.8,1.5)node[right]{$\gamma$};
\draw (3.7,.4)node{\tiny$e_\gamma^1$};
\end{tikzpicture}
\caption{The space $X(A_2)$ is a torus with one boundary labeled $e(\gamma)$. It is given by pasting together the two blue end circles of this cylinder which has one disk cut out. As CW-complex, $X(A_2)=e^0\cup e_\alpha^1\cup e_\beta^1\cup e_\gamma^1\cup e_\emptyset^2$.}
\label{Fig:XA2}
\end{center}
\end{figure}
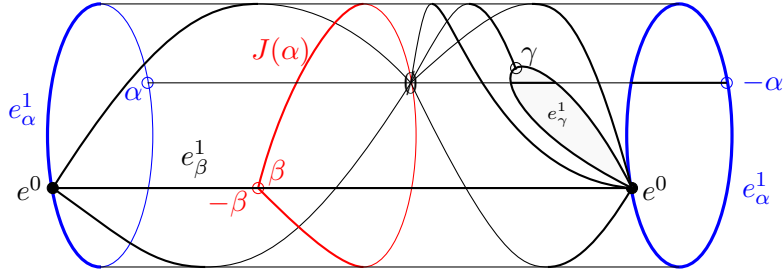
\end{eg}

We will examine in more detail the structure of the space $X(Q)$ and $J(\beta)\subset X(Q)$ one cell at a time by induction on dimension. We use the notation $X(Q)^k$ for the $k$-skeleton of $X(Q)$ and $J(\beta)^{k-1}=J(\beta)\cap X(Q)^k$.% When we reach $k=3$, we will give a more precise descriptions of the attaching maps of the cells.

The cell complex $X(Q)$ has a single 0-cell (vertex) $e^0$. It has a 1-cell $e_\beta^1$ for every positive root $\beta\in\Phi^+(Q)$. The endpoints of each 1-cell are attached to the unique 0-cell. This gives a 1-dimensional CW complex $X(Q)^1$ whose fundamental group is the free group with generators $x(\beta)$ where $\beta\in \Phi^+(Q)$. The generator $x(\beta)$ is represented by the cell $e^1_\beta$. 

Before attaching more cells to $X(Q)^1$, we give a recursive description of the sets $J(\beta)^k$ in terms of the attaching maps of the cells.

\begin{prop}\label{def: J(beta)} 
The $0$-dimensional subset $J(\beta)^0\subset X(Q)^1$ consists of the center point of the cell $e^1_\beta$. Given $J(\beta)^{k-1}\subset X(\beta)^k$ for $k\ge1$ and attaching maps $\eta_i:S^{k}=\partial D^{k+1}\to X(Q)^{k}$, the set $J(\beta)^{k}$ is the union of $J(\beta)^{k-1}$ and certain subsets of each cell as follows. For each $(k+1)$-cell $e_i^{k+1}$, $J(\beta)^k\cap e_i^{k+1}$ is the cone of the inverse image of $J(\beta)^{k-1}$ under the attaching map $\eta_i:S^k\to X(Q)^k$, assuming the image of $\eta_i$ meets $J(\beta)^{k-1}$. Otherwise $J(\beta)^k\cap e_i^{k+1}$ is empty.
\end{prop}

Continuing with the construction of $X(Q)$, we take one 2-cell $e^2_{\alpha,\beta}$ for every unordered pair of hom-orthogonal roots $\alpha,\beta$. This 2-cell is attached to the 1-skeleton $X(Q)^1$ using the relation (2) in Theorem \ref{thm: presentation of the picture group determined by the spherical semi-invariant picture} corresponding to the pair $\{\alpha,\beta\}$. In Example \ref{eg: examples of relations}, Case (1), this gives a torus $S^1\times S^1$ with $J(\alpha)=S^1\times\ast$ and $J(\beta)=\ast\times S^1$. In Case (2) we get a torus with one boundary component given by the 1-cell $e^1_{\alpha+\beta}$. Cases (3) and (4) also give closed subsets of tori. To be more precise, we define each 2-cell to be a convex polygon with $m+2$ sides where $m$ is the number of elements of $\Phi^+(\alpha,\beta)$ ($m=2,3,4,6$ in Cases (1), (2), (3), (4), respectively). The attaching map sends these $m+2$ sides to the 1-cells corresponding to the letters in the relation (2) in Theorem \ref{thm: presentation of the picture group determined by the spherical semi-invariant picture}.

For all $k$, the cell complex $X(Q)$ will have one $k$-cell $e^k_{\alpha_\ast}$ for every unordered set of $k$ pairwise hom-orthogonal roots $\alpha_\ast=\{\alpha_1,\alpha_2,\cdots,\alpha_k\}$.

%{\xcolor{blue}Summary: 
\begin{summ}In Section \ref{sec3} we constructed the picture space $X(Q)$. It is a quotient space of the cone on the first barycentric subdivision of the cluster complex $\Sigma(Q)$ under certain identifications. The space $C(sd\,\Sigma(Q))$ is decomposed as a union of cells $E(\rho)$ for all simplices $\rho$ in $\Sigma(Q)$ (plus one big cell $E(\emptyset):=C(sd\,\Sigma(Q))$) and $E(\rho)$ is identified with $E(\rho')$ if and only if the right perpendicular categories of $|\rho|,|\rho'|$ agree.
\end{summ}
%}

%\newpage

\section{Homology of $X(A_n)$}\label{sec4}

%In {\xcolor{blue}Section \ref{sec4}} we compute the integral homology of the space $X(A_n)$. It is generated by indecomposable elements with square zero. There are $(n-2k+2)C_k$ indecomposable elements of degree $k$ for $2k-1\le n$. 
\vs2

In this section we compute the homology of $X(A_n)$ with {any} orientation. The first steps in the computation work in general. The cellular chain complex of $X(Q)$ has a weight filtration {for any modulated quiver $Q$ of finite representation type} and the homology of the associated graded complex is equal to the homology of the actual complex in type $A_n$ by Corollary \ref{cor:homology is given by basic ss sets}. We also show that the integral homology of $X(A_n)$ has no torsion. So the cohomology is free abelian with the same rank. The cup product structure will be determined in Section \ref{sec5}.%later using the results of the next section.% So, we delay that argument to keep the logic clear.

\subsection{Cellular chain complex of $X(Q)$}

First we recall the basic construction of the cellular chain complex $C_\ast(X)$ of a CW-complex $X$. For more details see \cite{Hatcher}. In degree $k$, $C_k(X)$ is the free abelian group generated by the set of oriented $k$-cells of $X$ modulo the relation that $e'=-e$ if $e'$ is the cell $e$ with orientation reversed. A standard notation for $k$-cells is: $e^k_\beta$ where $\beta$ is an element of some indexing set. Recall that each $k$-cell $e_\beta^k$ is the image of a $k$-disk $E(\beta)\cong D^k$ under a continuous mapping $E(\beta)\to X$ called the \emph{characteristic map} of $e_\beta^k$. The restriction of this map to $S_\beta^{k-1}=\partial E(\beta)$ has image in $X^{k-1}$ and is called the \emph{attaching map} of the cell $e_\beta^k$ and denoted by $\eta_\beta:S_\beta^{k-1}\to X^{k-1}$.

The boundary map $d:C_k(X)\to C_{k-1}(X)$ is given by $de^k_\beta=\sum_\alpha n_{\alpha\beta}e^{k-1}_\alpha$ where $n_{\alpha\beta}\in\ZZ$ is called the \emph{incidence number} of $e^k_\beta$ and $e^{k-1}_\alpha$. This number is defined to be the degree of the composite mapping
\[
	S_\beta^{k-1}\xrightarrow{\eta_\beta}X^{k-1}\to X^{k-1}/X^{k-2}\to S_\alpha^{k-1}
\]
where $S^{k-1}_\alpha$ is the $(k-1)$-sphere $S^{k-1}_\alpha=E(\alpha)/S_\alpha^{k-2}$ and the mapping $X^{k-1}\to S_\alpha^{k-1}$ is the unique map having the property that the composition $E(\alpha)\to X^{k-1}\to E(\alpha)/S_\alpha^{k-2}$ is the quotient map and $E(\alpha')\to X^{k-1}\to E(\alpha)/S_\alpha^{k-2}$ has image one point for any $\alpha'\neq\alpha$. Another description of the same mapping: when $X^{k-2}$ is collapsed to one point we get $X^{k-1}/X^{k-2}=\vee S^{k-1}_\alpha$, a bouquet of $k-1$-spheres, and we project to the $S^{k-1}_\alpha$ summand.

In good cases, such as in the example of the picture space $X=X(Q)$, the boundary of each $k$-cell $e^k_\beta$ is a union of $k-1$-cells and the incidence number is just the number of times that isomorphic copies of the $k-1$-disk $E(\alpha)$ occurs in the boundary sphere of the disk $E(\beta)$. We will show (Proposition \ref{prop:formula for d<beta>}) that these incidence numbers are $0$ or $\pm1$.

\subsubsection{Description of $C_k(Q,\ZZ)$}
Recall that the $k$-cells of $X(Q)$ are indexed by all sets of $k$ pairwise hom-orthogonal positive roots $\beta_i$ of $Q$. Let $[\beta_1,\cdots,\beta_k]\in C_k(Q,\ZZ)$ denote the corresponding free generator of the cellular chain complex $C_\ast(Q,\ZZ)$ of $X(Q)$. We understand the order of the $\beta_i$ to be given, up to even permutation. Under an odd permutation, the sign changes. Thus:
\[
	[\beta_{\sigma(1)},\cdots,\beta_{\sigma(k)}]=\sgn(\sigma)[\beta_1,\cdots,\beta_k]
\]
This generator has \emph{degree} $k$. We often call this generator (with either sign) a \emph{cell} and denote it by $[\beta_\ast]$. We define the \emph{weight} of the cell $[\beta_\ast]=[\beta_1,\cdots,\beta_k]$ and the \emph{weight} of the set $\beta_\ast=\{\beta_1,\cdots,\beta_k\}$ to be the sum $wt[\beta_\ast]=wt(\beta_\ast):=\sum \beta_i\in\NN^n$ of the vectors $\beta_i$. Given two weights $w,w'$ we say that $w\le w'$ if $w_i\le w_i'$ for $i=1,\cdots,n$ and $w<w'$ if $w\le w'$ and $w\neq w'$. Note that if $w<w'$ then $w$ comes before $w'$ in lexicographic order.

Given the set $\beta_\ast=\{\beta_1,\cdots,\beta_k\}$, recall that $\cA b(\beta_\ast)$ is the abelian subcategory of $mod\text-\Lambda$ {whose simple objects are the modules $M_{\beta_i}$} and $\Phi^+(\beta_\ast)$ is the set of all positive roots which can be written as nonnegative integer linear combinations of the roots $\beta_i$. {Thus $\Phi^+(\beta_\ast)$ is the set of dimension vectors of indecomposable objects in $\cA b(\beta_\ast)$.}

For each $\alpha\in \Phi^+(\beta_\ast)$, recall (Definition \ref{def: instead of D(beta,rho)}) that $D_{\beta_\ast}(\alpha)=\{x=\sum v_i\beta_i\in \RR\beta_\ast\,|\, \left<x,\alpha\right>=0\text{ and }\left<x,\alpha'\right>\le0\ \forall \alpha'\subseteq \alpha\text{ s.t. } \alpha'\in\Phi^+(\beta_\ast)\}$. {In particular, a positive root $\gamma\in \Phi^+(\beta_\ast)$ lies in $D_{\beta_\ast}(\alpha)$ if and only if $M_\alpha\in M_\gamma^\perp$.} Then, by Theorem \ref{thm: spherical semi-invariant picture}, the union of the $D_{\beta_\ast}(\alpha)$ intersected with the unit sphere $S^{k-1}$ in $\RR\beta_\ast$ is the spherical semi-invariant picture for the hereditary abelian category $\cA b(\beta_\ast)$. By Theorem \ref{thm: picture space}, the boundary of the cell $e_{[\beta_\ast]}^k$ is the union of cells $e_{[\alpha_\ast]}^p$ over all hom-orthogonal subsets of $\Phi^+(\beta_\ast)$ having less then $k$ elements. The following lemma gives the list of all such cells for $p=k-1$.

\begin{lem}\label{lem: first weight lemma}
Let $\beta_\ast=\{\beta_1,\cdots,\beta_k\}$ be a set of hom-orthogonal positive roots. Then there is a 1-1 correspondence between positive roots $\gamma\in \Phi^+(\beta_\ast)$ and $k-1$ element hom-orthogonal subsets $\alpha_\ast=\{\alpha_1,\cdots,\alpha_{k-1}\}$ of $\Phi^+(\beta_\ast)$. The correspondence is given by {$M_\gamma^\perp\cap \cA b(\beta_\ast)=\cA b(\alpha_\ast)$}. Furthermore, $\gamma$ is in the interior of $D_{\beta_\ast}(\alpha_i)$ for all $i$. Finally, $wt(\alpha_\ast)\ge wt(\beta_\ast)$ if and only if $M_\gamma$ is not a projective object in the abelian category $\cA b(\beta_\ast)$.
\end{lem}

\begin{proof}
The formula {$M_\gamma^\perp\cap \cA b(\beta_\ast)=\cA b(\alpha_\ast)$} gives the 1-1 correspondence. Assume for simplicity of notation that $k=n$ and $\beta_i$ are simple roots. Then $wt(\beta_\ast)=(1,1,\cdots,1)$ and $wt(\alpha_\ast)\ge wt(\beta_\ast)$ if and only if $\sum \alpha_j$ is sincere, i.e., there is no index $i$ so that the $i$-th coordinate of each $\alpha_j$ is zero. But, if this happens then the $i$-th projective root $\pi_i$ is left perpendicular to all $\alpha_j$ which implies $\gamma=\pi_i$. So, the last statement holds. The statement that $\gamma$ lies in the interior of each $D_{\beta_\ast}(\alpha_i)$ was already shown in Lemma \ref{lem: which D(beta) occur twice}.
%equivalent to saying that, for any root $\alpha'\in \Phi^+(\beta_\ast)$ which is a subroot of $\alpha_i$ so that the condition $\<v,\alpha'\><0$ holds for some $v\in D_{\beta_\ast}(\alpha_i)$, then $\<\gamma,\alpha'\><0$. But the condition on $\alpha'$ implies that it must be linearly independent from all $\alpha_i$. This implies $\<\gamma,\alpha'\><0$. Otherwise we would have $\<\gamma,\alpha'\>=0$ and we already have $\<\gamma,\alpha_i\>=0$ for all $i$ which is impossible since either $\alpha'$ and $\alpha_i$ span $\RR^k$ and the Euler-Ringel form is nonsingular.
\end{proof}

\subsection{Weight filtration of $C_\ast(Q;\ZZ)$ for Dynkin quivers} We will now determine the incidence numbers $n_{[\alpha_\ast][\beta_\ast]}$.%A key point is that, in Lemma \ref{lem: first weight lemma}, the exceptional cells $[\alpha_\ast]$ corresponding to $M_\gamma$ which are projective in $\cA b(\beta_\ast)$ do not occur in the boundary of $[\beta_\ast]$ because of (3) in the following proposition.

\begin{prop}\label{prop:formula for d<beta>}
The boundary of $[\beta_\ast]=[\beta_1,\cdots,\beta_k]$ in the chain complex $C_\ast(Q;\ZZ)$ is given by
\[
	d[\beta_\ast]=\sum n_{[\alpha_\ast][\beta_\ast]} [\alpha_\ast]
\]
where the sum is over all $\alpha_\ast$ which are hom-orthogonal subsets of $\Phi^+(\beta_\ast)$ having $k-1$ elements and with coefficient $n_{[\alpha_\ast][\beta_\ast]}=\pm1$ or $0$, where either
\begin{enumerate}
\item $wt(\alpha_\ast)=wt(\beta_\ast)$ in which case one of the roots $\alpha_i$ is equal to the sum of two of the roots $\beta_j$ and the remaining $\alpha$'s are equal to the remaining $\beta$'s, or
\item %$\alpha_i$ are hom-orthogonal roots in $\Phi^+(\beta_\ast)$ and 
$wt(\alpha_\ast)>wt(\beta_\ast)$, or
\item $wt(\alpha_\ast)\not \ge wt(\beta_\ast)$ in which case $n_{[\alpha_\ast][\beta_\ast]}=0$.
\end{enumerate}
Furthermore, in cases (1) and (2), $n_{[\alpha_\ast][\beta_\ast]}=\pm1$ is the sign of the change of basis matrix from the basis $[\beta_1,\cdots,\beta_k]$ to the basis $[\alpha_1,\cdots,\alpha_{k-1},\gamma]$, each ordered up to even permutation, where $\gamma$ is the unique positive root so the {$M_\gamma^\perp\cap \cA b(\beta_\ast)=\cA b(\alpha_\ast)$}.
\end{prop}

\begin{proof} (1) is the only way that the $k-1$ positive roots can add up to $\sum\beta_i$. By Lemma \ref{lem: first weight lemma}, (1) and (2) occur when the corresponding module $M_\gamma$ is not projective. (3) occurs when $M_\gamma$ is projective. Then $[\alpha_\ast]=[\beta_1,\cdots,\widehat{\beta_i},\cdots,\beta_k]$ which is right perpendicular to the projective root $\gamma=\pi_i$. But these terms occur twice as summands of $d\beta_\ast$ with opposite sign corresponding to the vertices $\pi_i$ and $-\pi_i$ in the spherical semi-invariant picture for $\beta_\ast$. So, they cancel. The signs comes from the definition of induced orientation on the boundary of a disk. The plane $\RR\alpha_\ast$ plays the role of the tangent plane to the unit sphere in $\RR\beta_\ast$ at the vector $\gamma$. The induced orientation is $\epsilon(\beta_\ast,\alpha_\ast)$. The two cancelling terms have signs given by the bases $[\beta_1,\cdots,\widehat{\beta_i},\cdots,\beta_k,\pi_i]$ and $[\beta_1,\cdots,\widehat{\beta_i},\cdots,\beta_k,-\pi_i]$ which are opposite.

The formula for the sign is the standard convention for the orientation of the boundary of an oriented manifold which in this case is the $k$-disk $E(\beta_\ast)$.
\end{proof}

We define a cell $[\beta_\ast]=[\beta_1,\cdots,\beta_k]$ to be \emph{minimal} if the sum of any two of the roots $\beta_i$ is not a root.

{
\begin{cor} Let $[\beta_\ast]=[\beta_1,\cdots,\beta_k]$ be an ordered set of hom-orthogonal roots for $\Phi^+(Q_\vare)$. Then the following are equivalent.
\begin{enumerate}
\item $\cA b(\beta_\ast)$ is semi-simple. 
%\item The sum of any two of the roots $\beta_i$ is not a root.
\item $d[\beta_\ast]=0$ where $d$ is the boundary map of the chain complex $C_\ast(Q;\ZZ)$.
% d:C_k(Q;\ZZ)\to C_{k-1}(Q;\ZZ)
\item $[\beta_\ast]$ is minimal.
\end{enumerate}
\end{cor}

%We call $[\beta_\ast]$ a \emph{minimal cell} or \emph{minimal set} if these conditions hold. 

\begin{proof}
$(1)\then (2)$: When $\cA b(\beta_\ast)$ is semi-simple, there are no $[\alpha_\ast]$ as described in the Proposition. So, $d[\beta_\ast]=0$.

$(2)\then (3)$: If the sum of two of the roots $\beta_i$ were a root then we would get at least one term in the expansion of $d[\beta_\ast]$ by Proposition \ref{prop:formula for d<beta>}.

$(3)\then (1)$: If $\cA(\beta_\ast)$ were not semi-simple, two of the roots would extend each other. Say, $M_{\beta_i}\to E\to M_{\beta_j}$. Since $\Hom_\Lambda(M_{\beta_i},M_{\beta_j})=0$, the middle term is indecomposable and $\alpha=\beta_i+\beta_j=\undim E$ would be a root contradicting (3).
\end{proof}
}

\begin{cor}
The chain complex $C_\ast(Q;\ZZ)$ is filtered by weight in the sense that the additive subgroup generated by the cells of weight $\ge w$ form a subcomplex $C_\ast(Q;\ZZ)_w$.\qed
\end{cor}

\begin{defn}\label{def: subquotient complex of wt w}
For any quiver $Q$ and weight $w$ we define the subquotient complex:
\[
	C_\ast(Q;\ZZ)_{(w)}:=C_\ast(Q;\ZZ)_w/\sum_{w'>w}C_\ast(Q;\ZZ)_{w'}
\]
\end{defn}

%\begin{defn} Given hom-orthogonal roots $[\alpha_\ast]=[\alpha_1,\cdots,\alpha_\ell]$ and $[\beta_\ast]=[\beta_1,\cdots,\beta_k]$ we say that $[\beta_\ast]$ is a \emph{resolution} of $[\alpha_\ast]$ and $[\alpha_\ast]$ is a \emph{degeneration} of $[\beta_\ast]$ if $k>\ell$ and there is an epimorphism $p:\{1,\cdots,k\}\onto \{1,\cdots,\ell\}$ so that each $\alpha_j$ is the sum of all $\beta_i$ for $i\in p^{-1}(j)$. Note that $[\beta_\ast]$ has no degenerations if and only if it is minimal. The \emph{degree} of $[\beta_\ast]=[\beta_1,\cdots,\beta_k]$ is $k$.\end{defn}

%\begin{eg} In type $A_5$ with straight orientation, $[\beta_{04},\beta_{45},\beta_{12},\beta_{23}]$  is a resolution of $[\beta_{05},\beta_{13}]$ since $\beta_{05}=\beta_{04}+\beta_{45}$ and $\beta_{13}=\beta_{12}+\beta_{23}$. The cell $[\beta_{05},\beta_{13}]$ is minimal since $\beta_{05}+\beta_{13}$ is not a root and $[\beta_{04},\beta_{45},\beta_{12},\beta_{23}]$ has no resolutions. \end{eg}

%Note that, for any $[\beta_\ast]$ which is not minimal, the degeneration of $[\beta_\ast]$ with minimal degree is necessarily minimal. 
We also need the following theorem from \cite{K},\cite{DW2} but only for $Q$ of type $A_n$.

\begin{thm}[Generic decomposition theorem]\label{thm: generic decomposition}
Let $Q$ be a quiver of Dynkin type. Then any $w\in \NN^n$ can be written uniquely as a positive linear combination of ext-orthogonal positive roots: $w=\sum m_i\alpha_i$.
\end{thm}

The decomposition $w=\sum m_i\alpha_i$ is called the \emph{generic decomposition} of $w$.

\subsection{Semi-simple categories in type $A$}
%We restrict to the case when $Q$ is of type $A_n$.
We now describe quivers of type $A_n$ and their weights. Let {$Q_\vare$ be the quiver of type $A_n$
\[
	Q_\vare:\ 1-2-\cdots-n
\]
with orientation of the arrows given by a \emph{sign function} $\vare=(\vare_1,\vare_2,\cdots,\vare_{n-1})\in\{+,-\}^{n-1}$ as follows. The arrow $i\ot(i+1)$ points left if $\vare_i=+$ and points right $i\to (i+1)$ if $\vare_i=-$.
} We also use the notation
\[
	\beta_{ij}=e_{i+1}+e_{i+2}+\cdots+e_j \quad 0\le i<j\le n
\]
for the positive roots of type $A_n$ where $e_j$ is the $j$-th simple root of the root system $A_n$.
%\end{defn}

%The positive roots are labeled $\beta_{ij}$ where $0\le i<j\le n$ as described in Definition \ref{def: roots of type An}.

\begin{defn}\label{def: admissible and basic weights}
We define a weight $w$ to be \emph{admissible} if there is at least one cell of weight $w$, i.e., $w$ is a sum of hom-orthogonal roots. An admissible weight $w$ is \emph{basic} if there is exactly one cell of weight $w$.
\end{defn}

The plan is as follows.

\begin{enumerate}
\item We give a numerical characterization of admissible weights in Lemma \ref{lem:characterization of admissible weights}.
\item We characterize basic weights in Proposition \ref{prop:characterization of basic weights}. Basic cells are cycles: $de^k_\alpha=0$.
\item (key step) For every nonbasic admissible weight we show, in Lemma \ref{lem: subquotient of nonbasic is acyclic}, that the subquotient complex $C(Q)_{(w)}=C(Q;\ZZ)_{(w)}$ of Definition \ref{def: subquotient complex of wt w} is acyclic.
\item For $w$ a basic weight of degree $k$, the corresponding subquotient complex $C(Q)_{(w)}$ is $\ZZ$ in degree $k$ and zero elsewhere.
\item We conclude in Corollary \ref{cor:homology is given by basic ss sets} that the cohomology of $X(Q)$ is freely generated by the set of basic cells which we identify by their weight.
\item Finally, we enumerate the set of basic weights in Theorem \ref{thm: rank of cohomology is ballot nos}.
\end{enumerate}

%{\color{red}$Q_\varepsilon: 1-2-\cdots-n$ with arrow $i-i+1$ pointing left if $\varepsilon=+$ and pointing right if $\varepsilon=-$}

\begin{rem} For $Q$ a quiver of finite representation type, in particular for $Q=Q_\vare$, if $\alpha,\beta\in \Phi^+(Q)$, $hom(\alpha,\beta)=0$ if and only if $\<\alpha,\beta\>\le0$ and $ext(\alpha,\beta)=0$ if and only if and $\<\alpha,\beta\>\ge0$. Thus $\alpha,\beta$ are hom-ext-orthogonal if and only if $\<\alpha,\beta\>=0=\<\beta,\alpha\>$.
\end{rem}

\begin{defn}\label{def of noncrossing}
We say that the {half-open} intervals $(i,j],(k,\ell]$ are \emph{noncrossing} if $i,j,k,\ell$ are distinct and one of the following holds%(1) has more complicated description in general. But, these cases are when the supports overlap: $i,j,k,\ell$ are distinct and:
	\begin{enumerate}
	\item $k<i<j<\ell$ and $\varepsilon_i=\varepsilon_j$.
	\item $i<k<\ell<j$ and $\varepsilon_k=\varepsilon_\ell$.
	\item $i<k<j<\ell$ and $\varepsilon_k\neq\varepsilon_j$.
 	\item $k<i<\ell<j$ and $\varepsilon_i\neq\varepsilon_\ell$.
	\item $j< k$ or $\ell< i$.
\end{enumerate}
\end{defn}

\begin{lem}\label{lem: hom-orthogonal roots in An}
Let $(i,j], (k,\ell]$ be half-open intervals in $(0,n]$.
\begin{enumerate}
\item[(a)] When $i,j,k,\ell$ are distinct the following are equivalent.
\begin{enumerate}
\item[(i)] $\beta_{ij},\beta_{k\ell}$ are hom-orthogonal.
\item[(ii)] $\beta_{ij},\beta_{k\ell}$ are ext-orthogonal.
\item[(iii)] $(i,j],(k,\ell]$ are {noncrossing}.
\end{enumerate}
\item[(b)] When $i,j,k,\ell$ are not distinct then one of the following holds.
\begin{itemize}
	\item $\beta_{ij},\beta_{k\ell}$ are hom-orthogonal but not ext-orthogonal and either $j=k$ or $i=\ell$.
	\item $\beta_{ij},\beta_{k\ell}$ are ext-orthogonal but not hom-orthogonal and either $i=k$ or $j=\ell$.
\end{itemize}
\end{enumerate}
\end{lem}

\begin{proof}
When $i,j,k,\ell$ are distinct an easy computation gives:
\[
	\<\beta_{ij},\beta_{k\ell}\>+\<\beta_{k\ell},\beta_{ij}\>=0
\]
So, $hom(\beta_{ij},\beta_{k\ell})=ext(\beta_{k\ell},\beta_{ij})$ and $hom(\beta_{k\ell},\beta_{ij})=ext(\beta_{ij},\beta_{k\ell})$ and we see that (1) and (2) are equivalent. Definition \ref{def of noncrossing} lists all possible ways that $i,j,k,\ell$ can be distinct. The values of $\vare$ are those which make $\<\beta_{ij},\beta_{k\ell}\>=0$ in each case. So, (3) is equivalent to (1) and (2). The statement for $i,j,k,\ell$ not distinct is clear.
\end{proof}

% and one of the three equivalent conditions in the above lemma hold. The lemma says that $\beta_{ij},\beta_{k\ell}$ are hom-ext-orthogonal if and only if the intervals $(i,j],(k,\ell]$ are {noncrossing}. 
The following is the well-known formula for the generic decomposition of any $w\in\NN^n$ for a quiver of type $A_n$.

\begin{thm}\label{thm: description of generic decomposition of w}\cite{A}
Let $Q_\vare$ be a quiver of type $A_n$ with orientation given by $\vare$. Then, for any $w=(w_1,\cdots,w_n)\in\NN^n$ define the intervals $(a_i,b_i]$ of length $b_i-a_i=w_i$ recursively, for $1\le i\le n$, as follows.
\begin{enumerate}
\item $a_1=0$ and $b_1=w_1$.
\item  If $\vare_i=+$ then $a_{i+1}=a_i$ and $b_{i+1}=a_{i}+w_{i+1}$.
\item If $\vare_i=-$ then $b_{i+1}=b_i$  and $a_{i+1}=b_{i}-w_{i+1}$.
\end{enumerate}
Then the number of times that $\beta_{ij}$ occurs in the generic decomposition of $w$ is equal to the number of integers $c$ so that $c\notin (a_i,b_i]$, $c\notin (a_{j+1},b_{j+1}]$ and $c\in(a_k, b_k]$ for all $i<k\le j$.
\end{thm}

We refer to each $c$ in the description above as a \emph{height} of $\beta_{ij}$. We are particularly interested in roots $\beta_{ij}$ in the generic decomposition $w$ of minimal and maximal height so that $k\in (i,j]$. We will refer to parts of the proof of this theorem later.

\begin{proof}
We need to show that any two roots $\beta_{ij}$ and $\beta_{k\ell}$ in the decomposition of $w$ given by the theorem are ext-orthogonal. First note that $i\neq \ell$ since $w_i<w_{i+1}$ and $w_\ell>w_{\ell+1}$. Similarly $j\neq k$. So, if $i,j,k,\ell$ are not distinct, then $\beta_{ij}$ and $\beta_{k\ell}$ are ext-orthogonal by Lemma \ref{lem: hom-orthogonal roots in An}. So, we may assume that $i,j,k,\ell$ are distinct. 

If the heights of the two roots are equal then their supports are disjoint and separated, thus noncrossing. So, suppose that the roots have different heights, say $\beta_{ij}$ has height $c_1$ and $\beta_{k\ell}$ has height $c_2<c_1$. There are four cases corresponding to the first four cases of (3) in Lemma \ref{lem: hom-orthogonal roots in An}. We consider only the second case: $i<k<j<\ell$. In that case the existence of $\beta_{k\ell}$ below $\beta_{ij}$ with $i<k<j$ implies that $\vare_k=-$ since, otherwise, $a_k=a_{k+1}$ and any root $\alpha$ which starts at $k$ must be above any root in the decomposition $\beta_{pq}$ with $p<k<q$. Similarly $\vare_j=+$. By (3)(b) in Lemma \ref{lem: hom-orthogonal roots in An}, the roots $\beta_{ij}$ and $\beta_{k\ell}$ are noncrossing. The other three cases are similar.
\end{proof}

The graphical representation of this is given by plotting the point $(k,c)\in\ZZ^2$ for which $1\le k\le n$ and $c\in (a_k,b_k]$. Then we connect any pair of points $(k,c),(k+1,c)$ with the same height and consecutive first coordinate. For example, if $n=7$, $\vare=(+,-,-,+,+,+)$ and $w=(1,2,3,3,2,1,2)$ we get:
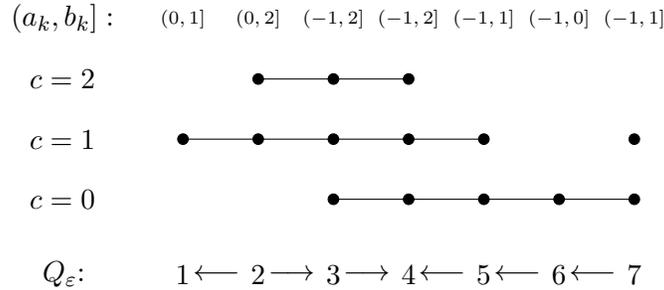
\begin{center}
\begin{tikzpicture}%[>=Stealth]
\foreach \x in {1,2,3,4,5,7}\filldraw (\x,.8) circle [radius=2pt];
\foreach \x in {2,3,4}\filldraw (\x,1.6) circle [radius=2pt];
\foreach \x in {3,4,5,6,7}\filldraw (\x,0) circle [radius=2pt];
%\filldraw (3,0) circle [radius=2pt];
\begin{scope}
\draw (-.6,2.4) node{$ (a_k,b_k]:$};
\draw (1,2.4) node{\tiny$ (0,1]$};
\draw (2,2.4) node{\tiny$ (0,2]$};
\draw (3,2.4) node{\tiny$ (-1,2]$};
\draw (4,2.4) node{\tiny$ (-1,2]$};
\draw (5,2.4) node{\tiny$ (-1,1]$};
\draw (6,2.4) node{\tiny$ (-1,0]$};
\draw (7,2.4) node{\tiny$ (-1,1]$};
\end{scope}
\begin{scope}[xshift=.4cm]
\draw (-1,0) node{$c=0$};
\draw (-1,.8) node{$c=1$};
\draw (-1,1.6) node{$c=2$};
\end{scope}
	\draw (1,.8)--(5,.8);
	\draw (2,1.6)--(4,1.6);
	\draw (3,0)--(7,0);
	\draw(-1,-1) node[right]{$Q_\vare$:};
\foreach \x in {1,2,...,7}\filldraw (\x,-1) node{$\x$};	
\foreach \x in {1,4,5,6} \draw(\x,-1)node[right]{$\longleftarrow$};	
\foreach \x in {2,3} \draw(\x,-1)node[right]{$\longrightarrow$};
\end{tikzpicture}
\end{center}
giving the generic decomposition $w=\beta_{05}+\beta_{14}+\beta_{27}+\beta_{67}$. The theorem says that, above vertex $k$, there are $w_k$ points with consecutive integer $y$-coordinates with the same lower bound as for $k+1$ when $\vare_{k}=+$ and the same upper bound as for $k+1$ if $\vare_{k}=-$. This gives ext-orthogonal roots adding up to $w$ since, for example, $(0,5],(2,7]$ are noncrossing since $\vare_2\neq\vare_5$ and $(0,5],(1,4]$ are noncrossing since $\vare_1=\vare_4$.

\begin{rem}\label{rem: when numbers are the same}
When $w_k=w_{k+1}$ as in the case $w_3=w_4=3$ in the above example, the intervals are equal $(a_k,b_{k}]=(a_{k+1},b_{k+1}]$ and thus we have parallel line segments connecting all the dots above vertex $k$ to those above vertex $k+1$. Equivalently, no roots of the form $\beta_{ik}$ or $\beta_{kj}$ occur in the generic decomposition of $w$.
\end{rem}

\begin{defn}\label{def: blocks}
Let $[\alpha_\ast]=[\alpha_1,\cdots,\alpha_k]$ be a minimal cell. Then each root $\alpha_s$ is equal to $\beta_{ij}$ for some $0\le i<j\le n$. And the intervals $(i,j]$ are pairwise noncrossing. Let $(p_t,q_t], 1\le t\le m$ be the maximal intervals in the support of $[\alpha_\ast]$ numbered so that\[
	0\le p_1<q_1<p_2<q_2<\cdots<p_m<q_m\le n\,.
\]
Then we define the \emph{blocks} $B_t=B_{p_tq_t}\in \NN^n$ of $[\alpha_\ast]$ to be the portion of the weight of $[\alpha_\ast]$ which has support in $(p_t,q_t]$. We also say that $B_r$ are the \emph{blocks} of $w=wt(\alpha_\ast)$ since they depend only on $w$. Thus $w=\sum_{1\le i\le k} \alpha_i=\sum_{1\le t\le m} B_{p_tq_t}$ is the sum of its blocks. In particular, a weight $w$ is defined to be a \emph{block} if and only if its support is a single interval $(p,q]$. 
\end{defn}

%We arrange the blocks so that
%The blocks $B_t$ of a minimal cell $[\alpha_\ast]$ are determined by the weight $w=\sum \alpha_i$ of $[\alpha_\ast]$.

Recall that a weight $w\in\NN^n$ is \emph{admissible} if it is the weight of some hom-orthogonal set $\beta_\ast$. This includes $w=0$ which is the weight of the empty hom-orthogonal set. The weight $w=(1,2,3,3,2,1,2)$ in the above example is not admissible. However, $w'=(1,2,3,3,2,1,1)$ is admissible since $w'=\beta_{05}+\beta_{14}+\beta_{27}$ is a hom-orthogonal decomposition of $w'$.

\begin{lem}\label{lem:characterization of admissible weights}
A weight $w\in\NN^n$ is admissible if and only if $|w_i-w_{i+1}|\le 1$ for all $0\le i\le n$ with the convention that $w_0=0=w_{n+1}$. Furthermore, the generic decomposition of an admissible weight $w$ gives a hom-orthogonal decomposition of $w$.%{\color{red} should be true for any orientation}
\end{lem}

\begin{proof}
The condition is clearly necessary. For example if $w_{i+1}\ge w_i+{2}$ then any $\beta_\ast$ with weight $w$ will have two objects $\beta_{ij}$ and $\beta_{ik}$ one of which is a subroot of the other and are therefore not hom-orthogonal. Conversely, suppose that $|w_i-w_{i+1}|\le 1$ for all $0\le i\le n$. Then no two roots in the generic decomposition of $w$ will start or end at the same place. So, they will be noncrossing.
%To see that the condition is sufficient, express $w$ as the sum of the smallest possible number of roots. This is the number of indices $i$ with the property that $w_i<w_{i+1}$. This is also the number of indices $j$ so that $w_j>w_{j+1}$. If we view each $i$ as an open parenthesis ``('' and each $j$ as a closed parenthesis ``)''  then we see that they are paired. Then the collection of roots $\beta_{ij}$ for such pairs $i,j$ forms a hom-orthogonal set of roots with total weight $\sum \beta_{ij}=w$. So, $w$ is admissible.
\end{proof}

%The pairing between $i$ and $j$ used in the above proof can be described as follows. Given $j$, the corresponding $i$ is the largest integer $i<j$ so that $w_i<w_j$. Since the degree of this cell $[\beta_{ij}]$ is minimal, it is minimal. The following lemma shows that there are no other minimal objects with the same weight.

\begin{lem}\label{lem:uniqueness of ss sets}
For each admissible weight $w$ there is a unique minimal set $\alpha_\ast$ of weight $w$.% {\color{red} should be true for any orientation}
\end{lem}

\begin{proof} Existence follows from the previous Lemma \ref{lem:characterization of admissible weights}. Uniqueness follows from the Generic Decomposition Theorem \ref{thm: generic decomposition}. 
\end{proof}

\subsection{Face operators and cut sets}

{We will define ``face operators'' and use them to lay the ground work to prove in the next subsection that the subquotient complex $C(Q)_{(w)}$ for nonbasic $w$ are acyclic. The first step is to show that a weight $w$ is not basic if and only if it is in the image of one of the face operators $\partial_k^\ast$ which we now define.

\begin{defn}
For any sign function $\vare=(\vare_1,\cdots,\vare_{n-1})$ and any $1\le k\le n-1$ let $\partial_k\vare$ and $s_k\vare$ denote the sign functions of lengths $n-2$ and $n$ given by deleting and repeating $\vare_k$ respectively. Thus, $Q_{\partial_k\vare}$ is obtained from $Q_\vare$ by collapsing the $k$-th arrow and $Q_{s_k\vare}$ is obtained from $Q_\vare$ by repeating the $k$-th arrow. The $k$-th \emph{face operator} is defined to be the functor
\[
	\partial_k^\ast:mod\text-KQ_{\partial_k\vare}\to mod\text-KQ_\vare
\]
which takes a representation $M$, repeats the value $M_k$ of $M$ at vertex $k$, then inserts the identity map between the two copies of $M_k$. The $k$-th \emph{degeneracy operator}
\[
	s_k^\ast:mod\text-KQ_{s_k\vare}\to mod\text-KQ_\vare
\]
is defined to be the functor which takes a representation $M$, deletes the vector space $M_{k+1}$ and inserts the linear map $M_{k}\to M_{k+2}$ (or $M_{k+2}\to M_{k}$) given by composing the morphisms $M_{k}\to M_{k+1}\to M_{k+2}$ (or $M_{k}\ot M_{k+1}\ot M_{k+2}$). %We call these functors the $k$-th \emph{face} and \emph{degeneracy} operators respectively.
\end{defn}

From this description, the following proposition is clear.

\begin{prop}
The functors $\partial_k^\ast,s_k^\ast$ satisfy the following.
\begin{enumerate}
\item $\partial_k^\ast:mod\text-KQ_{\partial_k\vare}\to mod\text-KQ_\vare$ is a full and faithful exact embedding whose image is equivalent to the wide subcategory of $mod\text-KQ_\vare$ of all representations for which the $k$-th arrow of $Q_\vare$ is an isomorphism $M_k\cong M_{k+1}$.
\item $s_k^\ast:mod\text-KQ_{s_k\vare}\to mod\text-KQ_\vare$ is an exact epimorphism.
\item $\partial_k s_k\vare=\vare=\partial_{k+1}s_k\vare$ and the compositions $s_k^\ast\circ \partial_k^\ast$ and $s_k^\ast\circ \partial_{k+1}^\ast$ are the identity functor on $mod\text-KQ_\vare$.
\end{enumerate}
\end{prop}

%Suppose that $w$ is an admissible weight with $w_k=w_{k+1}>0$. Then we will ``cut'' $w$ at the $k$-th arrow $k\to k+1$ (or $k\ot k+1$) to obtain a non-minimal cell with weight $w$.

\begin{defn} We define the \emph{resolution set} $R(w)$ of any admissible weight $w$ to be the set of all integers $k$ so that $w_k=w_{k+1}>0$. 
\end{defn}

For any $k\in R(w)$ we will show how to ``cut'' $w$ at the $k$-th arrow $k\to k+1$ (or $k\ot k+1$) to obtain a non-minimal cell with weight $w$. 

\begin{rem}\label{rem: properties of R(w)}Let $I(w)=\{k\,|\, w_k<w_{k+1}\}$, $J(w)=\{k\,|\, w_k>w_{k+1}\}$. If $[\alpha_\ast]=[\alpha_1,\cdots,\alpha_k]$ is any cell of weight $w$ and $\beta_{ij}\in \alpha_\ast$ then we must have
\begin{enumerate}
\item $i\in I(w)\cup R(w)$ and $j\in J(w)\cup R(w)$.
\item If $k\in R(w)$ appears as a subscript of some element of $\alpha_\ast$, it appears exactly twice, once as a right subscript and once as a left subscript. (It must appear the same number of times on both sides since $w_k=w_{k+1}$. Since $\beta_{ik},\beta_{jk}$ are not hom-orthogonal this number is at most one.)
\item If $\beta_{ik},\beta_{kj}$ are both elements of $\alpha_\ast$ then $k\in R(w)$ and $\beta_{ij}$ is hom-orthogonal to all other elements of $\alpha_\ast$. (This follows from the calculation $\<\alpha_p,\beta_{ij}\>=\<\alpha_p,\beta_{ik}\>+\<\alpha_p,\beta_{kj}\>\le0$ and, similarly, $\<\beta_{ij},\alpha_p\>\le0$ for any other element $\alpha_p$ of $\alpha_\ast$.)
\end{enumerate}
\end{rem}

We define the \emph{cut set} $C(\alpha_\ast)$ of $\alpha_\ast$ to be the set of all $k\in R(w)$ which occurs as an subscript of some element of $\alpha_\ast$. By Remark \ref{rem: properties of R(w)}(2), $\beta_{ik},\beta_{kj}$ must both occur as elements of $\alpha_\ast$.

\begin{lem}\label{lem: cut sets give resolutions}
Every subset $S\subseteq R(w)$ is the cut set of a unique hom-orthogonal set $\alpha_\ast$ of weight $w$.
\end{lem}

We will use the notation $\beta_w(S)$ for the unique hom-orthogonal set with cut set $S$.

\begin{proof}
(Existence) Let $S$ be any subset of $R(w)$. Then $S$ corresponds to a set of arrows in $Q_\vare$ so that the value of $w_k$ is the same at the start and end of each of these arrows. Apply the degeneracy operator to each arrow to obtain a larger quiver $Q_{\vare'}$ where $\vare'=s_{j_1}s_{j_2}\cdots s_{j_m}\vare$ where $S=\{j_1,j_2,\cdots,j_m\}$ in increasing order. Then apply face operators to $w$ to repeat the value of $w_k=w_{k+1}$ at the new vertices. For example, if $w=(1,2,3,3,2,1,1)$ and $S=\{3,6\}$, then we get $w'=\partial_3^\ast \partial_6^\ast(w)=(1,2,3,3,3,2,1,1,1)$. (In general, $w'=\partial_{j_1}^\ast \partial_{j_2}^\ast\cdots \partial_{j_m}^\ast w$.) Now decrease the value of $w'$ at the new vertices by one to obtain $w''$ which is still admissible. In the example, $w''=(1,2,3,{\underline2},3,2,1,{\underline0},1)$. 

Let $\alpha(S)=[\alpha_i]$ be the generic decomposition of $w''$. Let $\beta(S)$ be the image of $\alpha(S)$ in $\Phi^+(Q_\vare)$ under the degeneracy operators which delete all the new vertices. Then we claim that $\beta(S)$ is a hom-orthogonal cell for $w$ with cut set $C(\beta(S))=S$. The reason is that $\alpha(S)$ necessarily has elements $\beta_{ij_1}$ and $\beta_{j_1+1,k}$ for some $i,k$ where $j_1$ is the smallest element of $S$. For any other element $\alpha_p$ of $\alpha_\ast$, $s_{j_1+1}^\ast$ must take $\alpha_p$ and $\beta_{ij_1}$ to hom-orthogonal roots since $\alpha_p$ and $\beta_{ij_1}$ are in the image of the face operator $\partial_{j_1+1}^\ast$ which is exact and $s_{j_1+1}^\ast\circ \partial_{j_1+1}^\ast$ is the identity operator. Repeat this argument for the other elements of $S$. This proves that $C(\beta(S))$ contains $S$. To see that $C(\beta(S))=S$ note that for every $k\in R(W)\backslash S$, the subscript $\tilde k$ corresponding to $k$ in $w''$ has the property that $\beta_{i\tilde k}$ and $\beta_{\tilde kj}$ do not occur in any generic decomposition of $w''$ by Remark \ref{rem: when numbers are the same}. Therefore, $k$ is not in $C(\beta(S))$.

(Uniqueness) Let $\alpha_\ast$ be any hom-orthogonal set with weight $w$ and cut set $C(\alpha_\ast)=S\subseteq R(w)$. Then we claim that $\alpha_\ast=\beta(S)$ the set constructed above. The reason is that both sets must lift to $\alpha(S)$ the unique minimal cell of weight $w''$. The lifting is given as follows. For each $\beta_{ij}$ in the set $\alpha_\ast$, if $j$ are not in the set $S$ then we lift $\beta_{ij}$ to $\partial^\ast(\beta_{ij})$ where $\partial^\ast$ is the composition of the face operators which repeat each vertex $j_i$ in $S$. If $j\in S$ then we lift $\beta_{ij}$ to $\partial^\ast(\beta_{ij})$ then decrease the last nonzero coordinate by one. In this way, the liftings of the elements of $\alpha_\ast$ will add up to $w''$ and not to $w'$ in the notation of the existence proof. This procedure lists the elements of $\alpha_\ast$ to roots which are both hom and ext-orthogonal. Therefore, the lifting must be equal to $\alpha(S)$ and $\alpha_\ast$ must be equal to $\beta(S)$.
\end{proof}
}

\begin{eg}In the graphical example for Theorem \ref{thm: description of generic decomposition of w}, $n=7$, $\vare=(+,-,-,+,+,+)$ and $(1,2,3,3,2,1,2)$ is not admissible. But $w=(1,2,3,3,2,1,1)$ would be admissible with generic decomposition $w=\beta_{05}+\beta_{14}+\beta_{27}$. (When we delete the root $\beta_{56}$, the remaining roots remain ext-orthogonal.) In the figure, the isolated dot on the right should be deleted. The resolution set is $R(1233211)=\{3,6\}$ since $w_3=w_4$ and $w_6=w_7$ and $\beta_w(3,6)=(\beta_{05},\beta_{14},\beta_{23},\beta_{36},\beta_{67})$. This comes from the generic decomposition $(\beta_{06},\beta_{15},\beta_{23},\beta_{47},\beta_{89})$ of $w''=(123232101)$. If the sign were $(+,-,+,+,+,+)$, the root $\beta_{14}$ (at the top of the diagram) would have been cut at $3$ instead of $\beta_{27}$ (at the bottom of the diagram).
\end{eg}

\begin{prop}\label{prop:characterization of basic weights} Let $w\in\NN^n$ be an admissible weight. Then the following are equivalent.
\begin{enumerate}
\item $w$ is basic.
\item There exists a unique hom-orthogonal set $\beta_\ast$ with weight $w$.
\item $R(w)$ is empty, i.e., $w_k\neq w_{k+1}$ for all $i$ except in the case $w_k=0=w_{k+1}$. 
\end{enumerate}
\end{prop}

\begin{proof} $(2) \then (1)$ follows from the definition of a basic weight (Definition \ref{def: admissible and basic weights}).

$(1)\then (3)$ %\ki
{If $R(w)$ is not empty then, by Lemma \ref{lem: cut sets give resolutions}, there is a cell $[\alpha_\ast]$ with nonempty cut set $R(w)$. Such an $[\alpha_\ast]$ is not ext-orthogonal. So, $w$ is not basic.}

$(3)\then(2)$ %\ki
{By Lemma \ref{lem: cut sets give resolutions}, hom-orthogonal sets are in bijection with subsets of $R(w)$. When $R(w)$ is empty, it has only one subset.} \end{proof}

\begin{cor}
Let $w=B_{ij}$ be a basic weight consisting of one block with support $(i,j]$ and let $\beta_\ast=[\beta_1,\cdots,\beta_k]$ be the unique hom-orthogonal set with weight $w$. {Then $j-i=2k-1$.}\end{cor}

\begin{proof} The sum $\sum_{i\le t\le j} |w_{t+1}-w_t|$ is equal to $2k$, twice the number of roots, since each root contributes 2 to this sum. Since $w$ consists of one block, the summands $|w_{t+1}-w_t|$ are equal to 1. So, the sum $2k$ is equal to the number of terms which is $j-i+1$.
\end{proof}

\subsection{Non-basic weights}

Let $w$ be a non-basic weight for $Q_\vare$. Then we will show that the subquotient complex $C_\ast(Q_\vare)_{(w)}$ given in Definition \ref{def: subquotient complex of wt w} has zero homology.

\begin{rem}\label{rem: formula which implies acyclic}
By {Lemma \ref{lem: cut sets give resolutions}}, {modulo terms of higher weight,} the boundary of $\beta_w(s_1,\cdots,s_r)$ is equal to the sum 
\[
	d(\beta_w(s_1,\cdots,s_r))=\sum_{i=1}^r \pm\beta_w(s_1,\cdots,\widehat{s_i},\cdots,s_r)
\]
of $r$ terms, each with coefficient plus or minus 1.
\end{rem}

\begin{lem}\label{lem: subquotient of nonbasic is acyclic}
For any non-basic weight $w$, the subquotient complex $C_\ast(Q_\vare)_{(w)}$ is acyclic.% {\color{red} should be true for any orientation}
\end{lem}

\begin{proof} {This follows from Remark \ref{rem: formula which implies acyclic} by induction on $r$. Let $S$ be any nonempty subset of $R(w)$. Let $C_\ast(S)$ be the subcomplex of $C_\ast(Q_\vare)_{(w)}$ generated by all $\beta_w(T)$ where $T\subseteq S$. Then we claim that $C_\ast(S)$ is acyclic. When $S=\{s_1\}$ has only one element then $C_\ast(S)$ is a chain complex with two generators $\beta_w(\emptyset)$ and $\beta_w(s_1)$ and $d\beta_w(s_1)=\pm\beta_w(\emptyset)$. So, $C_\ast(s_1)$ is acyclic.

When $S$ has at least two elements let $S=T\cup \{s_0\}$. Then $C_\ast(T)$ is a subcomplex of $C_\ast(S)$ which is acyclic by induction on the size of $S$. The quotient complex $C_\ast(S)/C_\ast(T)$ is also acyclic since it is has the same number of generators $\beta_w(T'\cup \{s_0\})$, $T'\subseteq T$, as has $C_\ast(T)$ and satisfies the formula analogous to Remark \ref{rem: formula which implies acyclic}. Therefore the extension $C_\ast(S)$ of $C_\ast(T)$ by $C_\ast(S)/C_\ast(T)$ is acyclic.}
\end{proof}

This proves the following theorem.

\begin{thm}\label{thm: generators of homology of GAn}
The homology of the associated graded complex $\bigoplus_w C_\ast(Q_\vare)_{(w)}$ is freely generated by the basic hom-orthogonal sets $\beta_\ast$. 
\end{thm}

\begin{cor}\label{cor:homology is given by basic ss sets}
The homology of the space $X(Q_\vare)$ is freely generated by the basic hom-orthogonal sets $\beta_\ast$. Furthermore, $\beta_\ast$ is uniquely determined by its weight which is any basic weight.%{\color{red} should be true for any orientation}
\end{cor}

\begin{proof}
Minimal sets are cycles in $C_\ast(Q_\vare)$. By the theorem, the basic hom-orthogonal cells $[\beta_\ast]$ generate the homology of the chain complex. It remains to show that no integer linear combination of such cycles is a boundary.

Suppose not. Let $z$ be an integer linear combination of basic cells of degree $k$ which is the boundary of a $k+1$ chain: $z=dc$. Let $w$ be a weight which is minimal in lexicographic order so that $c_w\neq 0$ where $c_w$ is the component of $c$ of weight $w$. Choose $c$ so that this minimal weight $w$ is maximal in lexicographic order. Then $w$ is non-basic since, otherwise, $dc_w=0$ and $c$ can be replaced with $c-c_w$ contradicting the maximality of the minimal weight $w$. This implies that $z_w=0$. So, the image of $c_w$ in $C_\ast(Q_\vare)_{(w)}$ is a cycle and therefore a boundary. Say, $c_w=dx$ in $C_\ast(Q_\vare)_{(w)}$. In the chain complex $C_\ast(Q_\vare)$, the boundary of $x$ may have higher weight terms. So, $c-dx$ has no terms of weight $w$ but has new higher weight terms. This contradicts the maximality of $w$ in all cases. So, we conclude that $z$ is not a boundary and no linear combination of basic minimal cells is a boundary.

Equivalently, the homology of $C_\ast(Q_\vare)$ is isomorphic to the homology of the associated graded chain complex.
\end{proof}

It remains to determine the list of all basic weights.

\subsection{Basic weights}

By Corollary \ref{cor:homology is given by basic ss sets}, $H_k(X(Q_\vare);\ZZ)$ is free abelian for every $n,k$. Let $r(n,k)$ denote its rank. Then $r(n,k)$ is the number of basic weights $w\in \NN^n$ of degree $k$. We show that these numbers are equal to the ``ballot numbers'' by showing that they satisfy the same recursion.

\begin{defn}\label{def of ballot number}
The \emph{ballot number} $b(j,k)$ is defined to be the number of ways in which $j$ ``yes'' votes and $k$ ``no'' votes can be cast in an ordered sequence in such a way that the number of ``yes'' votes is always greater than or equal to the number of ``no'' votes. In particular $b(j,k)=0$ if $j<k$.
\end{defn}

Since the count starts at $(0,0)$ and votes are cast one at a time by assumption, we have the following recursion: $b(j,k)=0$ unless $j\ge k\ge0$, $b(0,0)=1$ and
\[
	b(j,k)=b(j-1,k)+b(j,k-1)
\]
for $j\ge1$. Recall that the $j$-th \emph{Catalan number} is
\[
	C_j=\frac1{j+1}\binom{2j}j.
\]
It is a well-known property of Catalan numbers that $b(j,j)=C_j$. An extension of this observation is the following recursion.

\begin{lem}\label{lem: recursion for b}
For $m\ge k\ge1$ we have
\[
	b(m,k)=b(m-1,k)+\sum_{j=1}^k b(m-j,k-j)C_{j-1}.
\]
\end{lem}

\begin{proof}
There are two cases.

Case 1: The last vote cast was ``yes''. There are $b(m-1,k)$ ways this could happen.

Case 2: The last vote cast was ``no''. Consider the difference $m-k\ge0$ between the number of ``yes'' votes and the number of ``no'' votes. In case two this number was $m-k+1\ge 1$ before the last vote. Since this difference starts and ends at a smaller number, this difference must have been equal to $m-k$ at some earlier point. Let $j>0$ be minimal so that the last $2j$ votes were tied $j$ in favor and $j$ against. There are $C_{j-1}$ ways these last $2j$ votes could have been cast since the last vote was ``no'' and the first must have been ``yes''. So, there are $C_{j-1}b(m-j,k-j)$ ways that this could happen. 

Adding up all possible cases, we get the stated recursion.
\end{proof}

By Lemma \ref{lem:characterization of admissible weights} and Proposition \ref{prop:characterization of basic weights}, the basic weights $w\in\NN^n$ of degree $k$ are all sequences of nonnegative integers $w=(w_1,\cdots,w_n)$ satisfying the following conditions where $w_0=0=w_{n+1}$ by convention.
\begin{enumerate}
\item	$w_{i+1}=w_i+1\text{ or }w_{i+1}=\max(w_i-1,0)$ for all $i=0,\cdots,n$.
\item There are exactly $k$ values of $i$ for which $w_{i+1}=w_i+1$.
\end{enumerate}
We recall that a \emph{block} of this weight $w$ is a maximal sequence of consecutive positive coordinates. By condition (1) each block has odd length. For example
\[
	w=(12123210012101)
\]
has three blocks $B_{07},B_{9,12},B_{13,14}$ of lengths $7,3,1$ and degrees $4,2,1$ respectively. There is only one possible block of length 1 and of length 3 which are as given in the example. However, there are two possible blocks of length 5: 12321 and 12121. And there are 5 possible blocks of weight 7 (with the same support):
\[
	1234321,1232321,1212321,1232121,1212121
\]
Also, a block of length $2j+1$ has degree $j+1$.
\begin{lem}\label{lem: Catalan number of blocks of full support}
The number of possible blocks of length $2j+1$ with a given support is given by the Catalan number $C_j$.
So, there are a total of $(n-2j)C_j$ blocks of length $2j+1$.%{\color{red} should be true for any orientation. The description of the blocks is different for different orientations. But the number of them is the same.}
\end{lem}

\begin{proof}
There are $n-2j$ interval $(p,q]$ of length $2j+1$ in $(0,n]$ and, for each such interval, there is a 1-1 correspondence between blocks of length $2j+1$ and Dyck paths of length $2j$ given by  $f(i)=w_{i+1}-1$ for $0\le i\le 2j$. So, there are $(n-2j)C_j$ blocks. %subtracting one from each $w_i$ in the support.
\end{proof}
%The proof is left to the reader.

\begin{lem}\label{lem: recursion for r}
The ranks $r(n,k)$ are uniquely determined by the following recursion: $r(n,0)=1$ for all $n\ge0$ and for $k>0$ we have:
\[
	r(n,k)=\begin{cases} 0 & \text{if } n\le 2k-2\\
  r(n-1,k)+\sum_{1\le j\le k} r(n-2j,k-j)C_{j-1}  & \text{otherwise}
    \end{cases}
\]
where, for convenience of notation, we use the convention that $r(-1,0)=1$.
\end{lem}

\begin{proof} Since $X(Q_\vare)$ is connected, we have $r(n,0)=1$ for $n\ge0$. The convention $r(-1,0)=1$ is used to define the term $r(n-2j,k-j)$ when $n=2k-1$ and $j=k$. To get from $w_0=0$ to $w_{n+1}=0$ with $k$ steps up and $k$ steps down we must have at least $n+1\ge 2k$. So, $r(n,k)=0$ when $n+1<2k$.

Now consider all basic weight $w$ with $n,k\ge1$. There are two cases.

Case 1: $w_n=0$. In that case $(w_1,\cdots,w_{n-1})$ is a basic weight of degree $k$. So, there are $r(n-1,k)$ weights in this case.

Case 2: $w_n=1$. Let $2j-1$ be the length of the last block of $w$. Then $w_{n-2j}=0$ and $w'=(w_1,\cdots,w_{n-2j-1})$ is a basic weight of degree $k-j$. Since there are $C_{j-1}$ possibilities for the last block of $w$ and there are $r(n-2j,k-j)$ possibilities for $w'$ we have $C_{j-1}r(n-2j,k-j)$ possibilities for $w$ in this case. This proves the recursion.
\end{proof}

\begin{thm}\label{thm: rank of cohomology is ballot nos}
Let $Q$ be a quiver of type $A_n$. Then the integral homology group $H_k(X(Q);\ZZ)$ of the picture space $X(Q)$ is a free abelian group with rank equal to the ballot number $b(n-k+1,k)$ for all $k\ge0$.
\end{thm}

\begin{proof}
Lemmas \ref{lem: recursion for b} and \ref{lem: recursion for r} imply $
	r(n,k)=b(n-k+1,k)
$. The theorem follows.
\end{proof}

\begin{cor}\label{cor: cohomological dimension of X(An)}
Let $Q$ be any quiver of type $A_n$. Then $H_k(X(Q);\ZZ)=0$ for $k>\frac{n+1}2$ and is nonzero for $0\le k\le\frac{n+1}2$. When $n=2k-1$, $H_k(X(Q))$ is free of rank $C_{k}$.
\end{cor}

\begin{proof}
This follows from the observation that $b(m,k)=0$ for $k>m$, $b(k,k)=C_{k-1}$ and $b(n,k)\neq0$ for $k\le m$.
\end{proof}

Note that, for any Dynkin quiver $Q$ with $n$ vertices, $X(Q)$ is an $n$-dimensional CW complex. So, we  always have: $H_k(X(Q))=0$ for $k>n$. Returning to type $A_n$ we have:% the following.

\begin{rem} Let $Q$ be a quiver of type $A_n$. Then the rank of $H^k(X(Q))$ is given as follows for $n\le 9$.
\[
\begin{array}{cccccccc}
n & \lfloor\frac {n+1}2\rfloor & rk\,H^0& rk\,H^1& rk\,H^2& rk\,H^3& rk\,H^4& rk\,H^5\\
0 & 0 & 1 & 0\\
1 & 1 & 1 & 1\\
2 & 1 & 1 & 2 \\
3 & 2 & 1 & 3 & 2\\
4 & 2 & 1 & 4 & 5\\
5 & 3 & 1 & 5 & 9 & 5\\
6 & 3 & 1 & 6 & 14 & 14\\
7 & 4 & 1 & 7 & 20 & 28 & 14\\
8 & 4 & 1 & 8 & 27 & 48 & 42\\
9 &5 &1 & 9 & 35 & 75 & 90 & 42
\end{array}
\]
These numbers are easy to compute: each nonzero rank is the sum of the number above it and above and to the left of it (similar to Pascal's triangle).
\end{rem}

%{\xcolor{blue} Summary: 
\begin{summ}In Section \ref{sec4}, we showed that the homology of the space $X(A_n)$ is freely generated by ``basic weights''. These are disjoint unions of ``blocks''. Blocks are enumerated using Catalan numbers and the basic weights are enumerated by ballot numbers.
\end{summ}
%}

%\newpage

%\newpage

\section{Cup product structure}\label{sec5}

We now determine the cup product structure on the cohomology ring $H^\ast(X(Q_\vare);\ZZ)$. We use the fact that $X(Q_\vare)$ is a $K(\pi,1)$ for the picture group $G_0(Q_\vare)$. This is proved in detail in \cite{IT13} for any modulated quiver of finite representation type and in \cite{Cat0} for $\vare=(+,+,\cdots,+)$. So, we deal with the cohomology of the group $G_0(Q_\vare)$ instead of the space $X(Q_\vare)$. Since the homology is freely generated by the set of basic weights $w$, the cohomology is also freely generated as an additive group by the dual elements $w^\ast$. We will show that, as a ring, the cohomology is generated by the duals $w^\ast$ to weights $w$ having only one block. We call such generators \emph{dual blocks}. Theorem \ref{thm: cup product structure of Hast(G_0(An))} gives the complete list of relations: The cup product of dual blocks is nonzero if and only if their ``extended supports'' are pairwise disjoint. %First, we need an intrinsic (instead of computational) description of the dual blocks with ``full support.''

\subsection{Subgroups of $G_0(Q_\vare)$}

As a special case of Theorem \ref{thm: presentation of the picture group determined by the spherical semi-invariant picture} we have the following description of the picture group $G_0(Q_\vare)$.

\begin{prop}\label{prop: presentation of G_0(Qe)}
For $\vare=(\vare_1,\cdots,\vare_{n-1})$, the group $G_0(Q_\vare)$ has generators $x_{ij}=x(\beta_{ij})$ for $0\le i<j\le n$ modulo the following relations where $[x,y]:=y^{-1}xyx^{-1}$.
\begin{enumerate}
\item $[x_{ij}, x_{k\ell}]=1$ when $\beta_{ij},\beta_{k\ell}$ are hom-ext-orthogonal.
\item $[x_{ij},x_{jk}]=x_{ik}$ if $\vare_j=+$.
\item $[x_{jk},x_{ij}]=x_{ik}$ if $\vare_j=-$.
\end{enumerate}
Consequently, a minimal set of generators for $G_0(Q_\vare)$ is given by $\{x_{p-1,p}\,|\, 1\le p\le n\}$.
\end{prop}

As an example, $x_{ij},x_{k\ell}$ commute when their \emph{extended supports} $[i,j],[k,\ell]$ are disjoint.

For every $0\le p<q\le n$ let $\Q_{pq}$ be the full subquiver of the quiver $Q_\vare$ with vertex set $(p,q]:=\{p+1,\cdots,q\}=[p+1,q]$:
\[
\xymatrixrowsep{20pt}\xymatrixcolsep{10pt}
\xymatrix{%begin xy matrix
\Q_{pq}: & (p+1)\ar@{-}[r] &(p+2)\ar@{-}[r]& (p+3)\ar@{-}[r] & \cdots \ar@{-}[r]& q
	}%end xy matrix
\]
where the arrow $(p+i)\to (p+i+1)$ points right if $\vare_{p+i}=-$ and points left $(p+i)\ot (p+i+1)$ if $\vare_{p+i}=+$. Thus $\Q_{pq}\cong Q_{\vare'}$ is a quiver of type $A_{q-p}$ with orientation given by $\vare'=(\vare_{p+1},\cdots,\vare_{q-1})$. Let $G_0(\Q_{pq})$ be the picture group of $\Q_{pq}$. Thus $G_0(\Q_{pq})$ has generators $x_{ij}$ where $p\le i<j\le q$ with those relations listed in Proposition \ref{prop: presentation of G_0(Qe)} having the property that all letter are generators of $G_0(\Q_{pq})$. Let $s_{pq}:G_0(\Q_{pq})\to G_0(Q_\vare)$ be the group homomorphism induced by the inclusion map on generators: $s_{pq}(x_{ij})=x_{ij}$.

More generally, we have:% the following.

\begin{defn}
For any subset $J\subseteq \{1,2,\cdots,n\}$ let $\Q_J$ be the full subquiver of $Q_\vare$ with vertex set $J$, i.e., $(Q_\vare)_J$ has one vertex for each $j\in J$ and one arrow for every pair of consecutive integers $j,j+1\in J$. Then $\Q_J$ is a disjoint union of connected subquivers:
\[
	\Q_J=\coprod_{i=1}^m \Q_{p_iq_i}
\]
where $J=\coprod (p_i,q_i]$ is a minimal decomposition of $J$ as a disjoint union of intervals. Let $G_0(\Q_J)$ be the picture group of $\Q_J$. By definition this is generated by all $x_{ij}$ where $(i,j]\subseteq J$ with those relations as listed in Proposition \ref{prop: presentation of G_0(Qe)} all or whose letters are generators of $G_0(\Q_J)$. %Let $s_J:G_0(\Q_J)\to G_0(Q_\vare)$ be the group homomorphism induced by the inclusion map on generators: $s_J(x_{ij})=x_{ij}$.
\end{defn}

\begin{lem}\label{lem: rJ circ sJ = id}
There are unique group homomorphisms $s_J:G_0(\Q_J)\to G_0(Q_\vare)$ and $r_J:G_0(Q_\vare)\to G_0(\Q_J)$ given on generators by $s_J(x_{ij})=x_{ij}$ and
\[
	r_J(x_{ij})=\begin{cases} x_{ij} & \text{if } (i,j]\subseteq J\\
    1& \text{otherwise}
    \end{cases}
\] 
Furthermore, the composition $G_0(\Q_J)\xrarrow{s_J} G_0(Q_\vare)\xrarrow{r_J} G_0(\Q_J)$ is the identity.
\end{lem}

\begin{proof} It is clear that $s_J$ defines a homomorphism.
The map $r_J$ defines a homomorphism since it respects the relations of $G_0(Q_\vare)$ since $(i,k]=(i,j]\cup (j,k]$. For example, the relation $[x_{ij},x_{jk}]=x_{ik}$ in $G_0(Q_\vare)$ becomes the relation $x_{ij}x_{ij}^{-1}=1$ in $G_0(\Q_J)$ if $(i,j]\subseteq J$ and $(j,k]\nsubseteq J$. The composition $r_{J}\circ s_J$ is the identity since it is the identity on generators.\end{proof}

%This proposition can be generalized as follows.
Since $s_J:G_0(\Q_J)\into G_0(Q_\vare)$ is a split monomorphism sending each generator of $G_0(\Q_J)$ to a generator of $G_0(Q_\vare)$ with the same name, we will identify $G_0(\Q_J)$ with its image in $G_0(Q_\vare)$ and consider $s_J$ as an inclusion map. For the next statement we use the terminology that two half open intervals $(p_1,q_1]$, $(p_2,q_2]$ are \emph{separated} if the closed intervals $[p_1,q_1]$, $[p_2,q_2]$ are disjoint. Then every subset $J\subseteq (0,n]$ can be expressed uniquely as a union of separated intervals $J=\coprod_{1\le i\le m} (p_i,q_i]$ where $0\le p_1<q_1<p_2<q_2<\cdots<p_m<q_m\le n$.

%For the sequel we use the easy observation that $G_0(Q\coprod Q')=G_0(Q)\times G_0(Q')$ and $\Phi^+(Q\coprod Q')=\Phi^+(Q)\times \Phi^+(Q')$.

\begin{prop}\label{prop: product of G_0(Q) inside G_0(An)}
Let $J$ be the union of the separated intervals $(p_i,q_i]$. Then the subgroups $G_0(\Q_{p_i,q_i})$ of $G_0(Q_\vare)$ commute with each other and $G_0(\Q_J)=\prod G_0(\Q_{p_iq_i})$ is their internal direct product. Furthermore, the projection morphism $r_J:G_0(Q_\vare)\onto G_0(\Q_J)$ is equal to the product of projection morphisms:%So, we have inclusion and projection morphisms
\[
	r_J=\prod r_{p_iq_i}: G_0(Q_\vare)\onto \prod G_0(\Q_{p_i,q_i})=G_0(\Q_J).
\]
%whose composition is the identity.
\end{prop}

\begin{proof}
For $i\neq j$ the generators of $G_0(\Q_{p_i,q_i})$ commute with those of $G_0(\Q_{p_j,q_j})$ since they have extended supports in $[p_i,q_i]$ which are disjoint. The rest is clear.
\end{proof}

Since each $H_k(G_0(\Q_{p_iq_i}))$ is free abelian, we have the following version of the K\"unneth formula.

\begin{cor}
The homology groups and cohomology ring of $G_0(\Q_J)$ are the graded tensor products of the homology groups and cohomology rings of $G_0(\Q_{p_iq_i})$, e.g., 
\[
	H_k(G_0(\Q_J))\cong \bigoplus_{k_1+\cdots+k_m=k} H_{k_1}(G_0(\Q_{p_1q_1}))\otimes \cdots\otimes H_{k_m}(G_0(\Q_{p_mq_m}))
\]
and similarly for cohomology with multiplication given by the Koszul sign rule.
\end{cor}

The standard notation for the element of $H_k(G_0(\Q_J))$ corresponding to $a_1\otimes \cdots\otimes a_m$ is $a_1\times \cdots\times a_m$ and similarly for cohomology. If $b_1\times \cdots \times b_m,c_1\times \cdots\times c_m$ are cohomology classes, the Koszul sign rule gives
\[
	(b_1\times \cdots\times b_m)(c_1\times\cdots \times c_m)=(-1)^{\sum_{i<j}\deg b_i\deg c_j} b_1c_1\times\cdots\times b_mc_m
\]
If $\deg a_i=\deg c_i=k_i$ for each $i$, we have the evaluation rule:
\[
	[c_1\times\cdots\times c_m,a_1\times\cdots\times a_m]=(-1)^{\sum_{i<j} k_ik_j}\<c_1,a_1\>\cdots\<c_m,a_m\>
\]

Since picture spaces are $K(\pi,1)$'s for the picture groups, there are continuous mappings $X(\Q_J)\to X(Q_\vare)\to X(\Q_J)$ unique up to homotopy which induce the group homomorphisms $G_0(\Q_J)\to G_0(Q_\vare)\to G_0(\Q_J)$. The retraction $X(Q_\vare)\to X(\Q_J)$ is not easy to describe. But the inclusion map $X(\Q_J)\to X(Q_\vare)$ is easy since $X(Q_\vare)$ is a cell complex with one cell for every wide subcategory of $mod\text-KQ_\vare$ and $mod\text-K\Q_J$ is one of these wide subcategories, it is the abelian subcategory $\cA b(e_i\,|\, i\in J)$ with simple objects $S_i$ where $i\in J$.

\begin{prop}
For any subset $J\subseteq(0,n]$ let $j:X(\Q_J)\into X(Q_\vare)$ be the inclusion map sending $X(\Q_J)$ to the cell of $X(Q_\vare)$ corresponding to the wide subcategory $\cA b(e_i\,|\, i\in J)$ where $e_i=\undim S_i$ is the $i$-th unit vector. Then $\pi_1(X(\Q_J))=G_0(\Q_J)$ and $\pi_1(j)=s_J:G_0(\Q_J)\into G_0(Q_\vare)$.
\end{prop}

\begin{proof}
The isomorphism $G_0(Q_\vare)\cong \pi_1(X(Q_\vare))$ is given by sending the generator $x_{i-1,i}$ of $G_0(Q_\vare)$ to the homotopy class of the oriented loop given by the 1-cell $X(\Q_{i-1,i})$. When $i\in J$, $X(\Q_{i-1,i})\subseteq X(\Q_J)\subseteq X(Q_\vare)$. Therefore the inclusion map $X(\Q_J)\into X(Q_\vare)$ is the identity on the generators $x_{i-1,i}$ of $\pi_1(X(\Q_J))=G_0(\Q_J)$.
\end{proof}

\begin{cor}
The image of the split monomorphism $H_\ast(G_0(\Q_J))\to H_\ast(G_0(Q_\vare))$ induced by $s_J:G_0(\Q_J)\into G_0(Q_\vare)$ is spanned by all basic weights $w$ with support in $J$. %More generally, suppose that $0\le p_1<q_1<\cdots<p_k<q_k\le n$. Then the image of the split monomorphism $\bigoplus H_\ast(G_0(\Q_{p_iq_i}))\into H_\ast(G_0(Q_\vare))$ is spanned by all basic weights with support in $\coprod (p_i,q_i]$.
\end{cor}

\begin{proof}
By Theorem \ref{thm: generators of homology of GAn}, the basic generators of the homology of $G_0(Q_\vare)$ are represented by cycles made up of single cells $[\beta_\ast]$ which are the basic hom-orthogonal sets. Such a cell is contained in the space $X(\Q_J)$ if the support of $\beta_\ast$ is contained in $J$. Therefore, any basic generator of the homology of $G_0(Q_\vare)$ with support in $J$ lies in the image of the homology of $G_0(\Q_J)$. A simple dimension count will verify that these generators span all of $H_\ast(G_0(\Q_J))$.
\end{proof}

\subsection{Dual blocks} We now consider blocks $B$ (minimal basic weights) with support $(p,q]$ where $q-p=2k-1$. By Corollary \ref{cor: cohomological dimension of X(An)}, $H_k(G_0(\Q_{pq}))$ has rank $C_k=\frac1{k+1}\binom{2k}k$ and, by Corollary \ref{cor:homology is given by basic ss sets}, it has a basis given by cycles with basic weights (of degree $k$). By Lemma \ref{lem: Catalan number of blocks of full support} there are $C_{k-1}$ such cycles whose weights are blocks with full support $(p,q]$.

\begin{defn}\label{defn: dual blocks}
Let $B_i$, $i=1,\cdots,C_{k-1}$, be the blocks with full support $(p,q]$. Let 
\[
	B_i^\ast\in H^k(G_0(\Q_{pq});\ZZ)=\Hom(H_k(G_0(\Q_{pq}));\ZZ)
\]
be the dual cohomology classes represented by the cocycle sending $B_i$ to 1 and all other basic weight cycles to 0. We define the \emph{dual blocks} of $Q_\vare$ with support in $(p,q]$ to be the images of these cohomology classes under $r_{pq}^\ast:H^k(G_0(\Q_{pq});\ZZ)\to H^k(G_0(Q_\vare);\ZZ)$ for all $p,q$ (with $q-p$ odd). The \emph{degree} of these dual classes is $k=(q-p+1)/2$.
\end{defn}

The main theorem of this section is that the dual blocks $r_{pq}^\ast(B_i^\ast)$ generate the ring $H^\ast(G_0(Q_\vare))$. For fixed $k,p,q$ let $K(\Q_{pq})$ denote the direct summand of $H^k(G_0(\Q_{pq})$ freely generated by the $C_{k-1}$ dual classes $B_i^\ast$ and let $K_{pq}(Q_\vare)=r_{pq}^\ast K(\Q_{pq})\subseteq H^k(G_0(Q_\vare))$. The first statement we need to show is that the cup product of dual blocks is nonzero if and only if their supports are separated.

\begin{lem}\label{lem: cup product is zero}
Let $(p_1,q_1]$, $(p_2,q_2]$ be two intervals of odd length in $(0,n]$. Then the mapping $K_{p_1q_1}(Q_\vare)\otimes K_{p_1q_1}(Q_\vare)\to H^\ast(G_0(Q_\vare))$ given by cup product $a_1\otimes a_2\mapsto a_1 a_2$ is a split monomorphism if $[p_1,q_1]$, $[p_2,q_2]$ are disjoint and is zero otherwise.
\end{lem}

\begin{proof}
Let $J=(p_1,q_1]\cup(p_2,q_2]$. Then the split retractions $r_{p_iq_i}:G_0(Q_\vare)\onto G_0(\Q_{p_iq_i})$ factors through the split retraction $r_J:G_0(Q_\vare)\onto G_0(\Q_J)$. 
When $(p_1,q_1]$, $(p_2,q_2]$ are separated, $G_0(\Q_J)=G_0(\Q_{p_1q_1})\times G_0(\Q_{p_2q_2})$ and the cup product of any two elements $a_i=r^\ast_{p_iq_i}(b_i)\in K_{p_iq_i}(Q_\vare)$, $b_i\in K(\Q_{p_iq_i})$ is given by
\[
	a_1 a_2=r_J^\ast(b_1\times b_2)
\]
Since $r_J$ is a split epimorphism, $r_J^\ast$ is a split monomorphism. So, by the K\"unneth formula, the mapping $K(\Q_{p_1q_1})\otimes K(\Q_{p_2q_2}) \to H^\ast(G_0(Q_\vare))$ which sends $b_1\otimes b_2$ to $r_J^\ast(b_1\times b_2)$ is a split monomorphism. Since $r_{p_iq_i}^\ast:K(\Q_{p_iq_i})\cong K_{p_iq_i}(Q_\vare)$, the lemma follows in this case.

When $(p_1,q_1]$, $(p_2,q_2]$ are not separated, $J$ has length $\le 2k_1+2k_2-2$. By Corollary \ref{cor: cohomological dimension of X(An)}, this gives $H^{k_1+k_2}(G_0(\Q_J))=0$. So, for any $a_i=r_{p_iq_i}^\ast(b_i)\in K_{p_iq_i}(Q_\vare)$, $b_i\in K(\Q_{p_i,q_i})$ % is zero since
\[
	a_1 a_2=r_J^\ast(r_1^\ast(b_1) r_2^\ast(b_2))\in r_J^\ast(H^{k_1+k_2}(G_0(\Q_J))=0
\]
where $r_i$ are the split projection maps $r_i:G_0(\Q_J)\to G_0(\Q_{p_iq_i})$.
\end{proof}

A similar argument proves:% the following.

\begin{lem}\label{lem: when cup product is nonzero}
Let $(p_i,q_i]$ be $m$ intervals of odd length $q_i-p_i=2k_i-1$ in $(0,n]$. Then the mapping $\bigotimes K_{p_iq_i}(Q_\vare) \to H^{\sum k_i}(G_0(Q_\vare))$ given by cup product $a_1\otimes \cdots\otimes a_m\mapsto a_1 \cdots a_m$ is a split monomorphism if $(p_i,q_i]$ are separated and is zero otherwise.\qed
\end{lem}

When the intervals $(p_i,q_i]$ are separated, we denote the image of the split monomorphism $\bigotimes K_{p_iq_i}(Q_\vare) \into H^{\sum k_i}(G_0(Q_\vare))$ by $K_J(Q_\vare)$ where $J=\coprod (p_i,q_i]$. We also use the notation $K(\Q_J)=(r_J^\ast)^{-1}K_J(Q_\vare)=\bigotimes K(\Q_{p_iq_i})$ which is a direct summand of $H^{\sum k_i}(G_0(\Q_J))$. We call the number $\sum k_i$ the \emph{degree} of $J$. The main theorem of this section can be rephrased to the statement that $H^\ast(G_0(Q_\vare))$ is the direct sum of all $K_j(Q_\vare)$ for all subsets $J$ of $(0,n]$ which are unions of separated intervals of odd length. This includes the empty set where, by convention, we have $K_\emptyset(Q_\vare)=\ZZ=H^0(G_0(Q_\vare))$.

\begin{lem}
Let $I,J$ be subsets of $(0,n]$. Then the image of the composite map
\[
	G_0(\Q_I)\xrarrow{s_I}G_0(Q_\vare)\xrarrow{r_J}G_0(\Q_J)
\]
is equal to $G_0(\Q_{I\cap J})\subseteq G_0(\Q_J)$.
\end{lem}

This easy observation implies:%has the following consequence.

\begin{lem}\label{lem: filtration condition}
Let $I,J$ be subsets of $(0,n]$ and let $k$ be the degree of $J$. The restriction map
\[
	s_I^\ast: K_J(Q_\vare)\to H^k(G_0(\Q_I))
\]
is a split monomorphism if $J\subseteq I$ and is zero otherwise.
\end{lem}

\begin{proof} Let $J=\coprod (p_i,q_i]$. Using the lemma and the fact that $K_J(Q_\vare)\cong K(\Q_J)$ which is a direct summand of $H^k(G_0(\Q_J))$, it suffices to show that the restriction map $s^\ast:H^k(G_0(\Q_J))\to H^k(G_0(\Q_{I\cap J}))$ is a split monomorphism on $K(\Q_J)=\bigotimes K(\Q_{p_iq_i})$ when $J\subseteq I$ and is zero otherwise. The first case is obvious since $I\cap J=J$ and $s:G_0(\Q_J)\to G_0(\Q_J)$ is the identity map in that case. So, it suffices to show that $s^\ast:H^k(G_0(\Q_J))\to H^k(G_0(\Q_{I}))$ is zero when $I$ is a proper subset of $J$. Let $I_i=I\cap (p_i,q_i]$. Then $I\subsetneq J$ implies that $I_i\subsetneq (p_i,q_i]$ for some $i$. In that case, $s_i^\ast:H^{k_i}(G_0(\Q_{p_iq_i}))\to H^{k_i}(G_0(\Q_{I_i}))$ is zero on $K(\Q_{p_iq_i})$ by definition of $K(\Q_{p_iq_i})$. This implies that the induced map on tensor products
\[
	K(\Q_J)=\bigotimes K(\Q_{p_iq_i}) \to \bigotimes H^{k_i}(G_0(\Q_{I_i}))\subseteq H^k(G_0(\Q_J))
\]
must also be zero.
\end{proof}

\begin{thm}\label{thm: block decomposition of cohomology}
For each $k>0$, $H^k(G_0(Q_\vare))$ is the direct sum of $K_J(Q_\vare)$ for all $J\subseteq (0,n]$ of degree $k$ which are unions of separated intervals of odd length.
\end{thm}

\begin{proof}
We show first that the $K_J(Q_\vare)$ are linearly independent. Suppose that $J_i$ are subsets of $(0,n]$ of the required kind and $a_i\in K_{J_i}(Q_\vare)$ so that $\sum a_i=0$. Let $I=J_j$ be maximal. Then, by Lemma \ref{lem: filtration condition}, $s_I^\ast:H^k(G_0(Q_\vare))\to H^k(G_0(\Q_I))$ is a monomorphism on $K_{J_j}(Q_\vare)$ and is zero on all other $K_{J_i}(Q_\vare)$. So, $s_I^\ast(0)=s_I^\ast(a_j)\neq0$ which is not possible.

By counting ranks we now see that, for fixed $k$, $\bigoplus K_J(Q_\vare)$ is subgroup of $H^k(G_0(Q_\vare))$ of full rank. So, it remains to show that the quotient $H^k(G_0(Q_\vare))/\bigoplus K_J(Q_\vare)$, a finite additive group, is zero, i.e., it has no $p$-torsion for any prime $p$. Equivalently, we need to show the following. Let $x\in H^k(G_0(Q_\vare))$ so that $px=\sum a_i$ where $a_i\in K_{J_i}(Q_\vare)$. Then we need to show that, for each $i$, $a_i=pb_i$ for some $b_i\in K_{J_i}(Q_\vare)$. 

To show this, suppose not. Then there is a $j$ so that $a_j$ is not divisible by $p$. Chose $j$ so that $I=J_j$ is maximal. Then we get the equation:
\[
	s_I^\ast(px)=ps_I(x)=s_I^\ast(a_j)+\sum s_I^\ast(pb_i)
\]
where the sum is over all $i$ so that $J_j\subsetneq J_i$. So, $s_I^\ast(a_j)$ is divisible by $p$ in $H^k(G_0(\Q_I))$. But $s_I^\ast:K_{J_j}(Q_\vare)\to H^k(G_0(\Q_I))$ is a split monomorphism. So, $a_j$ must be divisible by $p$ contradicting the choice of $j$. We conclude that $H^k(G_0(Q_\vare))$ is the direct sum of all $K_J(Q_\vare)$ of degree $k$ for $k>0$.
\end{proof}

\subsection{Cohomology of $G_0(Q_\vare)$}

Theorem \ref{thm: block decomposition of cohomology} completes the description of the cohomology of $G_0(Q)$ for any quiver $Q$ of type $A_n$ which we summarize in the following theorem.

\begin{thm}\label{thm: cup product structure of Hast(G_0(An))}
%Let $Q$ be a quiver of type $A_n$. 
The integral cohomology of the picture group $G_0(Q)$ of any quiver $Q$ of type $A_n$ is generated, as a ring, by the dual blocks $r_{pq}^\ast(B_i^\ast)\in H^k(G_0(Q))$ which satisfy:% the following.
\begin{enumerate}
\item $q-p=2k-1$
\item The support of the dual block is $(p,q]$.
\item Given $k,p,q$ there are $C_{k-1}=\frac1k \binom{2k-2}{k-1}$ dual blocks $r_{pq}^\ast(B_i^\ast)$.
\end{enumerate}
The cup product of any collection of dual blocks is nonzero if and only if their extended supports $[p,q]$ are pairwise disjoint. Furthermore, as an additive group, $H^\ast(G_0(Q);\ZZ)$ is freely generated by the nonzero cup products of dual blocks (including the empty product $1\in H^0(G_0(Q))$). %{\color{red} should be true for any orientation}
\end{thm}

\begin{proof}
Since dual blocks are defined (Definition \ref{defn: dual blocks}) to be dual to the basic blocks (basic weights with one block), they are in 1-1 correspondence with these weights which are enumerated in Lemma \ref{lem: Catalan number of blocks of full support}. Lemma \ref{lem: when cup product is nonzero} gives the stated characterization of when cup products of dual blocks are nonzero. Theorem \ref{thm: block decomposition of cohomology} proves the last statement that the nonzero products of dual blocks freely generate the cohomology as a $\ZZ$-module.
\end{proof}

By the universal coefficient theorem for cohomology we can extend this theorem to any coefficient ring $A$ and we observe that the statement does not depend on $\vare$.

\begin{cor}
Let $Q$ be a quiver of type $A_n$. Then the cohomology ring of $G_0(Q)$ with coefficients in any commutative ring $A$ is generated by the dual blocks modulo only the relation that the product of dual weights is zero when their extended supports intersect (and the usual graded commutativity rules). As an $A$-module $H^\ast(G_0(Q);A)$ is freely generated by the nonzero products of dual weights.\qed
\end{cor}

\begin{cor}
Let $Q$ be a quiver of type $A_n$. Then the cohomology ring $H^\ast(G_0(Q);A)$ for any commutative ring $A$ is independent of the choice of orientation.\qed
\end{cor}

\begin{eg}
For a quiver $Q$ of type $A_3$ with any orientation, the cohomology ring is $H^\ast(G_0(Q))$ is generated by three dual blocks $a_1,a_2,a_3$ of degree $1$ and one dual block $b$ of degree $2$ modulo the relation that the product of any two generators is zero except for $a_1a_3=-a_3a_1$. So, $H^0(G_0(Q))=\ZZ$, $H^1(G_0(Q))\cong\ZZ^3$ with basis $a_1,a_2,a_3$, $H^2(G_0(Q))\cong\ZZ^2$ with basis $b,a_1a_3$ and $H^k(G_0(Q))=0$ for $k\ge3$.

For a quiver $Q'$ of type $A_5$ with any orientation, the ring $H^\ast(G_0(Q'))$ is generated by $a_1,\cdots,a_5$ of degree 1, $b_1,b_2,b_3$ of degree 2 and $c_1,c_2$ of degree 3.
\begin{enumerate}
\item $H^0(G_0(Q'))=\ZZ$.
\item $H^1(G_0(Q'))\cong\ZZ^5$ with basis $a_1,\cdots,a_5$.
\item $H^2(G_0(Q'))\cong\ZZ^9$ with basis $b_1,b_2,b_3$ and six nonzero products $a_ia_j$ with $|j-i|\ge2$.
\item $H^3(G_0(Q'))\cong\ZZ^5$ with basis $c_1,c_2$, $a_1b_3,b_1a_5$, $a_1a_3a_5$.
\end{enumerate}
\end{eg}

When $Q$ is a quiver of type $A_n$, the group $G_0(Q)$ depends on the orientation of $Q$. This is a computer calculation using GAP which was carried out by D.Ruberman. Although the group depends on the orientation, {as we have shown, the homology and cohomology of the groups are independent of the orientation}. We believe that this holds more generally, i.e., the cohomology of the picture group should be independent of the orientation of the quiver and depend only on the underlying Dynkin diagram. According to He Wang, the difference between these groups can also be detected by the Massey product structure of their cohomology rings. See \cite{SW} for more about this. These are questions for further research.

%{\xcolor{blue} Summary: 
\begin{summ}Section \ref{sec5} determines the cup product structure of the integral cohomology of $G_0(Q)$ for any quiver $Q$ of type $A_n$ and shows it is independent of the orientation.
\end{summ}
%}

\section*{Acknowledgements}

Research for this paper was generously supported by grants from the National Science Foundation and the National Security Agency. The first author was supported by NSA Grant \#H98230-13-1-0247, the second author was supported by NSF Grants \#DMS-1103813 and \#DMS-0901185 and the third author was supported by NSF Grant \#DMS-1400740. 

Special thanks go to Kent Orr who was a coauthor of previous papers in this series. He and the first author previously used ``pictures'' in their research on link invariants \cite{IOr}. The authors would also like to thank Ira Gessel and Olivier Bernardi for explaining to them the properties of ballot numbers. Discussions with Danny Ruberman on the topology of picture groups and picture spaces were also very helpful. Finally, He Wang gave us a lot to think about with his comments about Massey products.
%\bibliography{onlineBib2010}

\begin{thebibliography}{aa}

\bibitem[A]{A} S. Abeasis, \emph{Codimension 1 orbits and semi-invariants for the representations of an oriented graph of type $A_n$}, Trans. Amer. Math. Soc. \textbf{282} (1984), 463Ð485.

\bibitem[BMRRT]{BMRRT} Aslak~Bakke Buan, Robert~J. Marsh, Idun Reiten, Marcus Reineke and Gordana Todorov, \emph{Tilting theory and cluster combinatorics}, Adv. Math. \textbf{204} (2006), no.~2, 572--618.



\bibitem[DW]{DW}
Harm Derksen and Jerzy Weyman, \emph{Semi-invariants of quivers and saturation for {L}ittlewood-{R}ichardson coefficients}, J. Amer. Math. Soc. \textbf{13} (2000), no.~3, 467--479 (electronic).


\bibitem[DW2]{DW2}
Harm Derksen and Jerzy Weyman, \emph{On the Canonical Decomposition of
Quiver Representations}, Compositio Mathematica \textbf{133} (2002), 245--265.



\bibitem[DR]{DR}
Vlastimil Dlab and Claus~Michael Ringel, \emph{Indecomposable representations of graphs and algebras}, Mem. Amer. Math. Soc. {\textbf6} (1976), no.~173, v+57.

\bibitem[FZ]{FZ-Y} Sergei Fomin and Andrei Zelevinsky, \emph{Y-systems and generalized associahedra}, Ann. of Math., {\bf158} (2003), 977--1018.

%\bibitem[FR]{FR}Sergei Fomin and Nathan Reading, \emph{Generalized cluster complexes and Coxeter combinatorics}, IMRN International Mathematics Research Notices 2005, No. 44, 2709--2757.

\bibitem[H]{Hall} Marshall Hall Jr., \emph{The theory of groups}, New York (1959).

\bibitem[Ha]{Hatcher} Allen Hatcher, \emph{Algebraic Topology}, Cambridge Univ. Press, Cambridge, 2002

\bibitem[H:AlgGp]{Hum:AlgGrps}
James~E. Humphreys, \emph{Linear algebraic groups}, Graduate texts in Mathematis, vol.~21, Springer.

\bibitem[H:RefGp]{Hum:ReflGps}
James~E. Humphreys, \emph{Reflection groups and {C}oxeter groups}, Cambridge Studies in Advanced Mathematics, vol.~29, Cambridge University Press, Cambridge, 1990.

\bibitem[I79]{Ithesis}
Kiyoshi Igusa, \emph{The $Wh_3(\pi)$-invariant for pseudoisotopies}, PhD Thesis, Princeton University 1979.

\bibitem[I14]{Cat0}
Kiyoshi Igusa, \emph{The category of noncrossing partitions}, arXiv:1411.0196.%\href{http://arxiv.org/abs/1411.0196}{http://arxiv.org/abs/1411.0196}.



\bibitem[IK]{IK} Kiyoshi Igusa and John Klein, \emph{The {B}orel regulator map on pictures. {I}{I}. {A}n example from {M}orse theory}, $K$-Theory \textbf{7} (1993), no.~3, 225--267.

\bibitem[IOTW09]{IOTW1}
Kiyoshi Igusa, Kent Orr, Gordana Todorov, and Jerzy Weyman, \emph{Cluster complexes via semi-invariants}, Compos. Math. \textbf{145} (2009), no.~4, 1001--1034.
  
\bibitem[IOTW15]{IOTW3}
Kiyoshi Igusa, Kent Orr, Gordana Todorov, and Jerzy Weyman, \emph{Modulated semi-invariants}, arXiv:1507.03051.%\href{http://arxiv.org/abs/1507.03051}{http://arxiv.org/abs/1507.03051}.

\bibitem[IOr]{IOr}
Kiyoshi Igusa and Kent~E. Orr, \emph{Links, pictures and the homology of nilpotent groups}, Topology \textbf{40} (2001), no.~6, 1125--1166.
  
  
\bibitem[IT16]{IT13}% renumbered according to year.
 Kiyoshi Igusa and Gordana Todorov, \emph{Signed exceptional sequences and the cluster morphism category},  \href{http://people.brandeis.edu/~igusa/Papers/SignedExceptional.pdf}{http://people.brandeis.edu/~igusa/Papers/SignedExceptional.pdf}. 
 
  
\bibitem[IT17]{IT14}% renumbered according to year.
 Kiyoshi Igusa and Gordana Todorov, \emph{Picture groups and maximal green sequences},  \href{http://people.brandeis.edu/~igusa/Papers/GreenSeq.pdf}{http://people.brandeis.edu/~igusa/Papers/GreenSeq.pdf}. 
 
  
 
\bibitem[InTh]{InTh} Colin Ingalls and Hugh Thomas, \emph{Noncrossing partitions and representations of quivers}, Compos. Math. \textbf{145} (2009), no. 6, 1533--1562.

  
\bibitem[K]{K}  Victor G. Kac, \emph{Infinite root systems, representations of graphs and invariant theory. II}, J. Algebra 78 (1982), no. 1, 141Ð162.

\bibitem[Ke]{Keller} Bernhard Keller,
  	\emph{Quiver mutation and combinatorial {DT}-invariants}, {FPSAC}, vol {13}, {2013}, p9--20.

\bibitem[LP]{LamPasha} Thomas Lam and Pavlo Pylyavskyy, \emph{Total positivity of loop groups II: Chevalley generators}, arXiv:0906.0610.

\bibitem[Lo]{Loday} Jean-Louis Loday, \emph{Homotopical syzygies}, Contemporary Math. 265 (2000), 99--127.

\bibitem[LS]{LyndonS} Roger C. Lyndon and Paul E. Schupp, \emph{Combinatorial Group Theory}, Springer-Verlag, Berlin-New York, 1977, reprinted in ``Classics in Mathematics'' series, Springer-Verlag,
Berlin, 2001.

\bibitem[MRZ]{MRZ} Robert Marsh, Markus Reineke and Andre Zelevinsky, \emph{Generalized associahedra via quiver representations}, Trans. Amer. Math. Soc. 355 (10) (2003) 4171Ð4186.

\bibitem[R]{R} Nathan Reading, \emph{Universal geometric cluster algebras}, Math. Zeit., 277 (2014), 499--547.

\bibitem[S92]{S92} Aidan Schofield, \emph{General Representations of quivers}, Proc. London Math. Soc. (3) \textbf{65} (1992), 46-64.

\bibitem[SW]{SW} Alexander I. Suciu and He Wang, \emph{Formality properties of finitely generated groups and Lie algebras}, arXiv:1504.08294.

\bibitem[ST]{ST} David Speyer and Hugh Thomas, \emph{Acyclic cluster algebras revisited}, ``Algebras, quivers and representations, Proceedings of the Abel Symposium 2011 (2013), 275--298.



  
\end{thebibliography}

 %\newpage
 % \tableofcontents 

%%%%%%%%%%%%%%%%%%%%%%%%%%%%%%%%%%%%%%%%%%%
%%%%%%%%%%%%%%%%%%%%%%%%%%%%%%%%%%%%%%%%%%%
	%%%%%%%%%%%%%%%%%%%%%%%%%%%%%%%%%%%
						\end{document}